\documentclass[paper=a4,english,fontsize=11pt,parskip=half,abstract=true,index=totoc]{scrartcl}
\usepackage{babel}
\usepackage[utf8]{inputenc}
\usepackage[T1]{fontenc}
\usepackage[left=20mm,right=20mm,top=30mm,bottom=30mm]{geometry}
\usepackage{amsmath}
\usepackage{amsthm}
\usepackage{amssymb}
\usepackage{enumerate} 
\usepackage{thmtools}
\usepackage{mathtools}
\mathtoolsset{centercolon} 
\usepackage{empheq}
\usepackage[bookmarks=true,
            pdftitle={An invitation to formal power series},
            pdfauthor={Benjamin Sambale},
            pdfkeywords={},
            hypertexnames=false,
            pdfstartview={FitH}]{hyperref}

\newtheorem{Thm}{Theorem}[section]
\newtheorem{Lem}[Thm]{Lemma}

\newtheorem{Cor}[Thm]{Corollary}

\theoremstyle{definition}
\newtheorem{Def}[Thm]{Definition}
\newtheorem{Rmk}[Thm]{Remark}
\newtheorem{Ex}[Thm]{Example}
\newtheorem{A}[Thm]{Exercise}

\numberwithin{equation}{section}

\allowdisplaybreaks[1]

\renewcommand{\phi}{\varphi}
\newcommand{\ii}{\mathrm{i}}

\newcommand{\ZZ}{\mathbb{Z}}
\newcommand{\CC}{\mathbb{C}}
\newcommand{\QQ}{\mathbb{Q}}
\newcommand{\RR}{\mathbb{R}}
\newcommand{\NN}{\mathbb{N}}
\newcommand{\FF}{\mathbb{F}}

\newcommand{\GL}{\operatorname{GL}}

\newcommand{\tr}{\operatorname{tr}}
\newcommand{\sgn}{\operatorname{sgn}}
\newcommand{\per}{\operatorname{per}}
\newcommand{\diag}{\operatorname{diag}}
\newcommand{\adj}{\operatorname{adj}}
\newcommand{\id}{\operatorname{id}}
\newcommand{\TT}{\mathrm{t}}
\newcommand{\stir}[2]{\genfrac{[}{]}{0pt}{}{#1}{#2}}
\newcommand{\stirr}[2]{\genfrac{\{}{\}}{0pt}{}{#1}{#2}}
\newcommand{\gauss}[2]{\genfrac{\langle}{\rangle}{0pt}{}{#1}{#2}}
\newcommand{\res}{\operatorname{res}}

\newcommand{\inv}{\operatorname{inv}}

\AtBeginDocument{
  \DeclareFontShape{T1}{cmr}{m}{scit}{<->ssub*cmr/m/sc}{}%
}


\makeatletter
\def\thm@space@setup{%
  \thm@preskip=\parskip \thm@postskip=0pt
}
\makeatother

\title{An invitation to formal power series\footnote{This version differs significantly from the published article at Jahresbericht DMV}}
\author{Benjamin Sambale\footnote{Institut für Algebra, Zahlentheorie und Diskrete Mathematik, Leibniz Universität Hannover, Welfengarten 1, 30167 Hannover, Germany,
\href{mailto:sambale@math.uni-hannover.de}{sambale@math.uni-hannover.de}}}
\date{\today}
\dedication{{\small Dedicated to the memory of Christine Bessenrodt.}}

\begin{document}
\frenchspacing
\maketitle
\begin{abstract}\noindent
This is a lecture on the theory of formal power series developed entirely without any analytic machinery. Combining ideas from various authors we are able to prove Newton's binomial theorem, Jacobi's triple product, the Rogers--Ramanujan identities and many other prominent results. 
We apply these methods to derive several combinatorial theorems including Ramanujan's partition congruences, generating functions of Stirling numbers and Jacobi's four-square theorem. We further discuss formal Laurent series and multivariate power series and end with a proof of MacMahon's master theorem.
\end{abstract}


\tableofcontents
\renewcommand{\sectionautorefname}{Section}

\section{Introduction}

In a first course on abstract algebra students learn the difference between polynomial (real-valued) functions familiar from high school and formal polynomials defined over arbitrary fields. In courses on analysis they learn further that certain “well-behaved” functions possess a Taylor series expansion, i.\,e. a power series which converges in a neighborhood of a point. On the other hand, only specialized courses cover the formal world of power series where no convergence questions are asked. 

The purpose of these expository notes is to give a far-reaching introduction to formal power series without appealing to any analytic machinery (we only use an elementary discrete metric). In doing so, we go well beyond a dated account undertaken by Niven~\cite{Niven} in 1969 (for instance, Niven cites Euler's pentagonal number theorem without proof). An alternative approach with different emphases can be found in Tutte~\cite{Tutte1,Tutte2}.
To illustrate the usefulness of formal power series we offer several combinatorial applications including some deep partition identities due to Ramanujan and others. This challenges the statement “While the formal analogies with ordinary calculus are undeniably beautiful, strictly speaking one can’t go much beyond Euler that way…” from the introduction of the recent book by Johnson~\cite{Johnson}.
While most proofs presented here are not new, they are scattered in the literature spanning five decades and, to my knowledge, cannot be found in a unified treatment. Our main source of inspiration is the accessible book by Hirschhorn~\cite{Hirschhorn} (albeit based on analytic reasoning) in combination with numerous articles cited when appropriate. The work on these notes was initiated by lectures on combinatorics and discrete mathematics at the universities of Jena and Hannover. 
I hope that the present notes may serve as the basis of seminars for undergraduate and graduate students alike. The prerequisites do not go beyond a basic abstract algebra course (in \autoref{seclaurent}, some knowledge of algebraic and transcendental field extensions is assumed).

The material is organized as follows: In the upcoming section we define the ring of formal power series over an arbitrary field and discuss its basic properties. Thereafter, we introduce our toolkit consisting of compositions, derivations and exponentiations of power series. In the following section we extend the theory to formal Laurent series with the goal of proving the Lagrange--Bürmann inversion formula. 
In \autoref{secmain} we first establish the binomial theorems of Newton and Gauss and later obtain Jacobi's famous triple product identity, Euler's pentagonal number theorem and the Rogers--Ramanujan identities. In the subsequent section we apply the methods to combinatorial problems to obtain a number of generating functions. Most notably, we prove Ramanujan's partition congruences (modulo $5$ and $7$) as well as his so-called “most beautiful” formula. Another section deals with Stirling numbers, permutations, Faulhaber's formula and the Lagrange--Jacobi four-square theorem. 
Subsequently, multivariate power series enter the picture. We give proofs of identities of Vieta, Girard--Newton and Waring on symmetric polynomials. We continue by developing multivariate versions of Leibniz' differentiation rule, Faà di Bruno's rule and the inverse function theorem.
In the final section we go somewhat deeper by taking matrices into account. After establishing the Lagrange--Good inversion formula, we culminate by proving MacMahon's master theorem. 
In the appendix we review some algebraic properties of power series, which are rarely needed in combinatorics. For instance, we show that the ring of power series in finitely many indeterminates is a unique factorization domain, and we prove Puiseux' theorem on the algebraic closure of the ring of Laurent series.
In all parts of this work we often indicate analytic counterparts and connections to other areas. Furthermore, a few exercises are included.

\section{Definitions and basic properties}

The sets of positive and non-negative integers are denoted by $\NN=\{1,2,\ldots\}$ and $\NN_0=\{0,1,\ldots\}$ respectively.

\begin{Def}\hfill
\begin{enumerate}[(i)]
\item 
The letter $K$ will always denote a (commutative) field. In this section there are no requirements on $K$, but at later stages we need that $K$ has characteristic $0$ or contains some roots of unity. At this point we often replace $K$ by $\CC$ for convenience (and not for making analytic arguments available). This is not much loss of generality, since our theorems always involve at most countable many field elements, say $a_1,a_2,\ldots$, and $\QQ(a_1,a_2,\ldots)$ can be embedded into $\CC$.

\item 
A (formal) \emph{power series}\index{power series} over $K$ is just an infinite sequence $\alpha=(a_0,a_1,\ldots)$ with \emph{coefficients}\index{coefficient} $a_0,a_1,\ldots\in K$. 
The set of power series forms a $K$-vector space denoted by $K[[X]]$\index{*KXX@$K[[X]]$} with respect to the familiar componentwise operations:
\[\alpha+\beta:=(a_0+b_0,a_1+b_1,\ldots),\qquad \lambda\alpha:=(\lambda a_0,\lambda a_1,\ldots),\]
where $\beta=(b_0,b_1,\ldots)\in K[[X]]$ and $\lambda\in K$. 
We identify the elements $a\in K$ with the \emph{constant}\index{power series!constant} power series $(a,0,0,\ldots)$. In general, we call $a_0$ the \emph{constant term}\index{constant term}\index{coefficient!constant term} of $\alpha$ and set \index{*Inf@$\inf(\alpha)$}
\[\inf(\alpha):=\inf\{n\in\NN_0:a_n\ne 0\}\] 
with $\inf(0)=\inf\varnothing=\infty$ (as a group theorist I avoid calling $\inf(\alpha)$ the order of $\alpha$ as in many sources).
\item To motivate a multiplication on $K[[X]]$ we introduce an \emph{indeterminate}\index{indeterminate} $X$ and its powers 
\[X^0:=1=(1,0,\ldots),\qquad X=X^1=(0,1,0,\ldots),\qquad X^2=(0,0,1,0,\ldots),\qquad\ldots.\]
We can now formally write 
\[\alpha=\sum_{n=0}^\infty a_nX^n.\]
If there exists some $d\in\NN_0$ with $a_n=0$ for all $n>d$, then $\alpha$ is called a (formal) \emph{polynomial}.\index{polynomial} 
The smallest $d$ with this property is the \emph{degree}\index{degree} $\deg(\alpha)$\index{*Deg@$\deg(\alpha)$} of $\alpha$ (by convention $\deg(0)=-\infty$). In this case, $a_{\deg(\alpha)}$ is the \emph{leading coefficient}\index{coefficient!leading} and $\alpha$ is called \emph{monic}\index{polynomial!monic} if $a_{\deg(\alpha)}=1$.
The set of polynomials (inside $K[[X]]$) is denoted by $K[X]$. \index{*KX@$K[X]$}

\item 
We borrow from the usual multiplication of polynomials (sometimes called \emph{Cauchy product}\index{Cauchy product} or \emph{discrete convolution}\index{convolution}) to define
\[\boxed{\alpha\cdot \beta:=\sum_{n=0}^\infty\Bigl(\sum_{k=0}^na_kb_{n-k}\Bigr)X^n}\]
for arbitrary $\alpha,\beta\in K[[X]]$ as above. 
\end{enumerate}
\end{Def}

Note that $1,X,X^2,\ldots$ is a $K$-basis of $K[X]$, but not of $K[[X]]$. Indeed, $K[[X]]$ has no countable basis.
Opposed to a popular trend to rename $X$ to $q$ (as in \cite{Hirschhorn}), we always keep $X$ as “formal” as possible.

\begin{Lem}\label{intdom}
With the above defined addition and multiplication $(K[[X]],+,\cdot)$ is an integral domain with identity $1$, i.\,e. $K[[X]]$ is a commutative ring such that $\alpha\cdot\beta\ne 0$ for all $\alpha,\beta\in K[[X]]\setminus\{0\}$. Moreover, $K$ and $K[X]$ are subrings of $K[[X]]$.
\end{Lem}
\begin{proof}
Most axioms follows from the definition in a straight-forward manner. To prove the associativity of $\cdot$, let $\alpha=(a_0,\ldots)$, $\beta=(b_0,\ldots)$ and $\gamma=(c_0,\ldots)$ be power series. The $n$-th coefficient of $\alpha\cdot(\beta\cdot\gamma)$ is
\[\sum_{i=0}^na_i\sum_{j=0}^{n-i}b_jc_{n-i-j}=\sum_{i+j+k=n}a_ib_jc_k=\sum_{i=0}^n\Bigl(\sum_{j=0}^ia_jb_{i-j}\Bigr)c_{n-i},\]
which happens to be the $n$-th coefficient of $(\alpha\cdot\beta)\cdot\gamma$. 

Now let $\alpha\ne 0\ne\beta$ with $k:=\inf(\alpha)$ and $l:=\inf(\beta)$. Then the $(k+l)$-th coefficient of $\alpha\cdot\beta$ is $\sum_{i=0}^{k+l}a_ib_{k+l-i}=a_kb_l\ne 0$. In particular, $\inf(\alpha\cdot\beta)=\inf(\alpha)+\inf(\beta)$ and $\alpha\cdot\beta\ne 0$.

Since $K\subseteq K[X]\subseteq K[[X]]$ and the operations agree in these rings, it is clear that $K$ and $K[X]$ are subrings of $K[[X]]$ (with the same neutral elements). 
\end{proof}

The above proof does not require $K$ to be a field. It works more generally for integral domains and this is needed later in \autoref{defmulti}.
From now on we will usually omit the multiplication symbol $\cdot$ and apply multiplications always before additions. For example, $\alpha\beta-\gamma$ is shorthand for $(\alpha\cdot\beta)+(-\gamma)$. 
Moreover, we often omit the summation index in writing $\sum a_nX^n$ if it is clear from the context. 
The scalar multiplication is compatible with the ring multiplication, i.\,e. $\lambda(\alpha\beta)=(\lambda\alpha)\beta=\alpha(\lambda\beta)$ for $\alpha,\beta\in K[[X]]$ and $\lambda\in K$. This turns $K[[X]]$ into a $K$-algebra.

\begin{Ex}\label{bsppower}\hfill
\begin{enumerate}[(i)]
\item The following power series can be defined for any $K$:
\[1-X,\qquad\sum_{n=0}^\infty X^n,\qquad\sum nX^n,\qquad\sum(-1)^nX^n.\]
We compute
\[(1-X)\sum_{n=0}^\infty X^n=\sum_{n=0}^\infty X^n-\sum_{n=1}^\infty X^n=1.\]

\item For a field $K$ of characteristic $0$ (like $K=\QQ$, $\RR$ or $\CC$) we can define the (formal) \emph{exponential series}\index{exponential series}
\[\boxed{\exp(X):=\sum_{n=0}^\infty\frac{X^n}{n!}=1+X+\frac{X^2}{2}+\frac{X^3}{6}+\ldots\in K[[X]].}\]\index{*Exp@$\exp(X)$}
We will never write $e^X$ for the exponential series, since Euler's number $e$ simply does not live in the formal world.
\end{enumerate}
\end{Ex}

\begin{Def}\hfill
\begin{enumerate}[(i)]
\item We call $\alpha\in K[[X]]$ \emph{invertible}\index{power series!invertible} if there exists some $\beta\in K[[X]]$ such that $\alpha\beta=1$. As usual, $\beta$ is uniquely determined and we write $\alpha^{-1}:=1/\alpha:=\beta$.
As in any ring, the invertible elements form the group of units denoted by $K[[X]]^\times$. \index{*KXXtimes@$K[[X]]^\times$}

\item For $\alpha,\beta,\gamma\in K[[X]]$ we write more generally $\alpha=\frac{\beta}{\gamma}$ if $\alpha\gamma=\beta$ (regardless whether $\gamma$ is invertible or not).
For $k\in\NN_0$ let $\alpha^k:=\alpha\ldots\alpha$ with $k$ factors and $\alpha^{-k}:=(\alpha^{-1})^k$ if $\alpha\in K[[X]]^\times$.

\item For $\alpha\in K[[X]]$ let $(\alpha):=\bigl\{\alpha\beta:\beta\in K[[X]]\bigr\}$\index{*AA@$(\alpha)$} be the principal ideal generated by $\alpha$. 
\end{enumerate}
\end{Def}

\begin{Lem}\label{leminv}
Let $\alpha=\sum a_nX^n\in K[[X]]$. Then the following holds
\begin{enumerate}[(i)]
\item $\alpha$ is invertible if and only if $a_0\ne 0$. Hence, $K[[X]]^\times=K[[X]]\setminus(X)$. 

\item If there exists some $m\in\NN$ with $\alpha^m\in K$, then $\alpha\in K$. In particular, the elements of finite order in $K[[X]]^\times$ lie in $K^\times$. 
\end{enumerate}
\end{Lem}
\begin{proof}\hfill
\begin{enumerate}[(i)]
\item Let $\beta=\sum b_nX^n\in K[[X]]$ such that $\alpha\beta=1$. Then $a_0b_0=1$ and $a_0\ne 0$. Assume conversely that $a_0\ne 0$. We define $b_0,b_1,\ldots\in K$ recursively by $b_0:=1/a_0$ and
\[b_k:=-\frac{1}{a_0}\sum_{i=1}^{k}a_ib_{k-i}\in K\]
for $k\ge 1$. Then
\[\sum_{i=0}^ka_ib_{k-i}=\begin{cases}
1&\text{if }k=0,\\
0&\text{if }k>0.
\end{cases}\]
Hence, $\alpha\beta=1$ where $\beta:=\sum b_nX^n$.

\item We may assume that $m>1$ and $a:=\alpha^m\in K^\times$. For any prime divisor $p$ of $m$ it holds that $(\alpha^{m/p})^p\in K$. Thus, by induction on $m$, we may assume that $m=p$. By way of contradiction, suppose $\alpha\notin K$ and let $n:=\min\{k\ge 1:a_k\ne 0\}$. The $n$-th coefficient of $\alpha^p$ is $pa_0^{p-1}a_n=0$. Since $\alpha$ is invertible (indeed $\alpha^{-1}=a^{-1}\alpha^{p-1}$), we know $a_0\ne 0$ and conclude that $p=0$ in $K$ (i.\,e. $K$ has characteristic $p$). 
Now we investigate the coefficient of $X^{np}$ in $\alpha^p$. Obviously, it only depends on $a_0,\ldots,a_{np}$. 
Since $p$ divides $\binom{p}{k}=\frac{p(p-1)\ldots(p-k+1)}{k!}$ for $0<k<p$, the binomial theorem yields $(a_0+a_1X)^p=a_0^p+a_1^pX^p$. This familiar rule extends inductively to any finite number of summands. Hence,
\[(a_0+\ldots+a_{np}X^{np})^p=a_0^p+a_n^pX^{np}+a_{n+1}^pX^{(n+1)p}+\ldots+a_{np}^{p}X^{np^2}.\]
In particular, the $np$-th coefficient of $\alpha^p$ is $a_n^p\ne 0$; a contradiction to $\alpha^p\in K$.
If $\alpha$ has finite order $m$, then $\alpha^m=1\in K$ and therefore $\alpha\in K^\times$. 
\qedhere
\end{enumerate}
\end{proof}

\begin{Ex}\label{bspinv}\hfill
\begin{enumerate}[(i)]
\item By \autoref{bsppower} we obtain the familiar formula for the (formal) \emph{geometric series}\index{geometric series} 
\[\frac{1}{1-X}=\sum X^n\] 

\item For any $\alpha\in K[[X]]\setminus\{1\}$ and $n\in\NN$ an easy induction yields
\[\sum_{k=0}^{n-1}\alpha^k=\frac{1-\alpha^n}{1-\alpha}.\]

\item For distinct $a,b\in K\setminus\{0\}$ one has the \emph{partial fraction decomposition}\index{partial fraction decomposition}
\begin{equation}\label{partial}
\frac{1}{(a+X)(b+X)}=\frac{1}{b-a}\Bigl(\frac{1}{a+X}-\frac{1}{b+X}\Bigr),
\end{equation}
which can be generalized depending on the algebraic properties of $K$.
\end{enumerate}
\end{Ex}

We now start forming infinite sums of power series. To justify this process we introduce a discrete norm, which behaves much simpler than the euclidean norm on $\CC$, for instance.

\begin{Def}\label{defnorm}
For $\alpha=\sum a_nX^n\in K[[X]]$ let 
\[|\alpha|:=2^{-\inf(\alpha)}\in\RR\]\index{*A@$\lvert\alpha\rvert$}
be the \emph{norm}\index{norm} of $\alpha$ with the convention $|0|=2^{-\infty}=0$.
\end{Def}

The number $2$ in \autoref{defnorm} can of course be replaced by any real number greater than $1$.
Note that $\alpha$ is invertible if and only if $|\alpha|=1$.
The following lemma turns $K[[X]]$ into an ultrametric space.

\begin{Lem}\label{ultra}
For $\alpha,\beta\in K[[X]]$ we have 
\begin{enumerate}[(i)]
\item $|\alpha|\ge 0$ with equality if and only if $\alpha=0$,
\item $|\alpha\beta|=|\alpha||\beta|$,
\item $|\alpha+\beta|\le\max\{|\alpha|,|\beta|\}$ with equality if $|\alpha|\ne|\beta|$.
\end{enumerate}
\end{Lem}
\begin{proof}\hfill
\begin{enumerate}[(i)]
\item This follows from the definition.

\item Without loss of generality, let $\alpha\ne 0\ne\beta$. We have already seen in the proof of \autoref{intdom} that $\inf(\alpha\beta)=\inf(\alpha)+\inf(\beta)$. 

\item From $a_n+b_n\ne 0$ we obtain $a_n\ne 0$ or $b_n\ne 0$. It follows that $\inf(\alpha+\beta)\ge\min\{\inf(\alpha),\inf(\beta)\}$. 
This turns into the \emph{ultrametric inequality}\index{ultrametric inequality} $|\alpha+\beta|\le\max\{|\alpha|,|\beta|\}$. If $\inf(\alpha)>\inf(\beta)$, then clearly $\inf(\alpha+\beta)=\inf(\beta)$. \qedhere
\end{enumerate}
\end{proof}

\begin{Thm}\label{vollständig}
The distance function $d(\alpha,\beta):=|\alpha-\beta|$\index{*DAB@$d(\alpha,\beta)$} for $\alpha,\beta\in K[[X]]$ turns $K[[X]]$ into a complete metric space. 
\end{Thm}
\begin{proof}
Clearly, $d(\alpha,\beta)=d(\beta,\alpha)\ge 0$ with equality if and only if $\alpha=\beta$. Hence, $d$ is symmetric and positive definite. The triangle inequality follows from \autoref{ultra}:
\begin{align*}
d(\alpha,\gamma)&=|\alpha-\gamma|=|\alpha-\beta+\beta-\gamma|\le\max\bigl\{|\alpha-\beta|,|\beta-\gamma|\bigr\}\\
&\le|\alpha-\beta|+|\beta-\gamma|=d(\alpha,\beta)+d(\beta,\gamma).
\end{align*}
Now let $\alpha_1,\alpha_2,\ldots\in K[[X]]$ be a Cauchy sequence with $\alpha_m=\sum a_{m,n}X^n$ for $m\ge 1$. For every $k\ge 0$ there exists some $M_k\ge 1$ such that $|\alpha_m-\alpha_{M_k}|<2^{-k}$ for all $m\ge M_k$. This shows $a_{m,n}=a_{M_k,n}$ for all $m\ge M_k$ and $n\le k$. Without loss of generality, we may assume that $M_0\le M_1\le\ldots$. We define
\[a_k:=a_{M_k,k}\]
and $\alpha=\sum a_kX^k$. Then $|\alpha-\alpha_m|<2^{-k}$ for all $m\ge M_k$, i.\,e. $\lim_{m\to\infty}\alpha_m=\alpha$. Therefore, $K[[X]]$ is complete with respect to $d$. 
\end{proof}

Note that $K[[X]]$ is the completion of $K[X]$ with respect to $d$. In order words: power series can be regarded as Cauchy series of polynomials.
For convergent sequences $(\alpha_k)_k$ and $(\beta_k)_k$ we have 
\[\lim_{k\to\infty}(\alpha_k+\beta_k)=\lim_{k\to\infty}\alpha_k+\lim_{k\to\infty}\beta_k,\qquad\lim_{k\to\infty}(\alpha_k\beta_k)=\lim_{k\to\infty}\alpha_k\cdot\lim_{k\to\infty}\beta_k.\]
The infinite sum
\[\sum_{k=1}^\infty\alpha_k:=\lim_{n\to\infty}\sum_{k=1}^n\alpha_k\]
can only converge if $(\alpha_k)_k$ is a \emph{null sequence},\index{null sequence} that is, $\lim_{k\to\infty}|\alpha_k|=0$. 
Surprisingly and in stark contrast to euclidean spaces, the converse is also true as we are about to see. This crucial fact makes the arithmetic of formal power series much simpler than the analytic counterpart.

\begin{Lem}\label{infsum}
For every null sequence $\alpha_1,\alpha_2,\ldots\in K[[X]]$ the series
$\sum_{k=1}^\infty\alpha_k$ and $\prod_{k=1}^\infty(1+\alpha_k)$ converge, i.\,e. they are well-defined in $K[[X]]$.
\end{Lem}
\begin{proof}
By \autoref{vollständig} it suffices to show that the partial sums form Cauchy sequences. For $\epsilon>0$ let $N\ge 0$ such that $|\alpha_k|<\epsilon$ for all $k\ge N$. Then, for $k>l\ge N$, we have
\begin{align*}
\Bigl|\sum_{i=1}^k\alpha_i-\sum_{i=1}^l\alpha_i\Bigr|&=\Bigl|\sum_{i=l+1}^k\alpha_i\Bigr|\overset{\ref{ultra}}{\le}\max\bigl\{|\alpha_i|:i=l+1,\ldots,k\bigr\}<\epsilon,\\
\Bigl|\prod_{i=1}^k(1+\alpha_i)-\prod_{i=1}^l(1+\alpha_i)\Bigr|&=\prod_{i=1}^l\underbrace{|1+\alpha_i|}_{\le 1}\Bigl|\prod_{i=l+1}^k(1+\alpha_i)-1\Bigr|\le\Bigl|\sum_{\varnothing\ne I\subseteq\{l+1,\ldots,k\}}\prod_{i\in I}\alpha_i\Bigr|\\&\le\max\bigl\{|\alpha_i|:i=l+1,\ldots,k\bigr\}<\epsilon.\qedhere
\end{align*}
\end{proof}

We often regard finite sequences as null sequences by extending them silently.
Let $\alpha_1,\alpha_2,\ldots\in K[[X]]$ be a null sequence and $\alpha_k=\sum a_{k,n}X^n$ for $k\ge 1$. For every $n\ge 0$ only finitely many of the coefficients $a_{1,n},a_{2,n},\ldots$ are non-zero. This shows that the coefficient of $X^n$ in
\begin{equation}\label{infsums}
\sum_{k=1}^\infty\alpha_k=\sum_{n=0}^\infty\Bigl(\sum_{k=1}^\infty a_{k,n}\Bigr)X^n
\end{equation}
depends on only finitely many terms. The same reasoning applies to the $\prod_{k=1}^\infty(1+\alpha_k)$.

For $\gamma\in K[[X]]$ and null sequences $(\alpha_k)$, $(\beta_k)$ it holds that $\sum\alpha_k+\sum\beta_k=\sum(\alpha_k+\beta_k)$ and $\gamma\sum\alpha_k=\sum\gamma\alpha_k$ as expected.

\begin{Cor}\hfill
\begin{enumerate}[(i)]
\item Let $(\alpha_k)$ be a null sequence and $\pi\colon \NN\to\NN$ a bijection. Then
\[\sum_{k=1}^\infty\alpha_k=\sum_{k=1}^\infty\alpha_{\pi(k)}.\]

\item \textup(discrete \emph{Fubini's theorem}\textup)\index{Fubini's theorem} Let $\alpha_{k,n}\in K[[X]]$ such that $\lim_{k+n\to\infty}\alpha_{k,n}=0$. Then
\[\sum_{k=1}^\infty\sum_{n=1}^\infty\alpha_{k,n}=\sum_{n=1}^\infty\sum_{k=1}^\infty\alpha_{k,n}.\]
\end{enumerate}
\end{Cor}
\begin{proof}\hfill
\begin{enumerate}[(i)]
\item For every $n\in\NN$ there exists some $N\in\NN$ such that $\pi(k)>n$ for all $k>N$. Hence, 
\[\Bigl|\sum_{k=1}^N\alpha_k-\sum_{k=1}^N\alpha_{\pi(k)}\Bigr|\le\max\bigl\{|\alpha_k|:k>n\bigr\}\to0.\]

\item This follows from
\[
\Bigl|\sum_{k=1}^\infty\sum_{n=1}^\infty\alpha_{k,n}-\sum_{n=1}^N\sum_{k=1}^\infty\alpha_{k,n}\Bigr|=\Bigl|\sum_{k=1}^\infty\sum_{n=1}^\infty\alpha_{k,n}-\sum_{k=1}^\infty\sum_{n=1}^N\alpha_{k,n}\Bigr|=\Bigl|\sum_{k=1}^\infty\sum_{n=N+1}^\infty\alpha_{k,n}\Bigr|\xrightarrow{N\to\infty}0.\qedhere
\]
\end{enumerate}
\end{proof}

\begin{Ex}\hfill
\begin{enumerate}[(i)]
\item For $\alpha\in (X)$ we have $|\alpha^n|=|\alpha|^n\le 2^{-n}\to 0$ and therefore $\sum\alpha^n=\frac{1}{1-\alpha}$. 
So we have substituted $X$ by $\alpha$ in the geometric series. This will be generalized in \autoref{defsub}.

\item Since every non-negative integer has a unique $2$-adic expansion, we obtain
\[\prod_{k=0}^\infty(1+X^{2^k})=1+X+X^2+\ldots=\frac{1}{1-X}.\]
Equivalently, 
\[\prod_{k=0}^\infty(1+X^{2^k})=\prod\frac{(1+X^{2^k})(1-X^{2^k})}{1-X^{2^k}}=\prod\frac{1-X^{2^{k+1}}}{1-X^{2^k}}=\frac{1}{1-X}.\]
More interesting series will be discussed in \autoref{seccomb}.

\item It is not always allowed to interchange limits and sums. For instance, if $\delta_{k,n}\in K[[X]]$ is the Kronecker-Delta, then 
\[\lim_{n\to\infty}\sum_{k=1}^\infty\delta_{k,n}=1\ne 0=\sum_{k=1}^\infty\lim_{n\to\infty}\delta_{k,n}.\]
\end{enumerate}
\end{Ex}

\begin{A}
Show that
\[\prod_{k=1}^\infty (1+X^k)(1-X^{2k-1})=1.\]
\end{A}

\section{The toolkit}

\begin{Def}\label{defsub}
Let $\alpha=\sum a_nX^n\in K[[X]]$ and $\beta\in K[[X]]$ such that $\alpha\in K[X]$ or $\beta\in (X)$. We define
\[\boxed{\alpha\circ\beta:=\alpha(\beta):=\sum_{n=0}^\infty a_n\beta^n.}\]\index{*AB@$\alpha(\beta)$}\index{*AcircB@$\alpha\circ\beta$}
\end{Def}

If $\alpha$ is a polynomial, it is clear that $\alpha(\beta)$ is a valid power series, while for $\beta\in (X)$ the convergence of $\alpha(\beta)$ is guaranteed by \autoref{infsum}. In the following we will silently assume that one of these conditions is fulfilled.
Observe that $|\alpha(\beta)|\le|\alpha|$ if $\beta\in(X)$. 

\begin{Ex}
For $\alpha=\sum a_nX^n\in K[[X]]$ we have $\alpha(0)=a_0$ and $\alpha(X^2)=\sum a_nX^{2n}$. On the other hand for $\alpha=\sum X^n$ we are not allowed to form $\alpha(1)$. 
\end{Ex}

\begin{Lem}\label{lemcomp}
For $\alpha,\beta,\gamma\in (X)$ and every null sequence $\alpha_1,\alpha_2,\ldots \in K[[X]]$ we have
\begin{align}
\Bigl(\sum\alpha_k\Bigr)\circ\beta&=\sum \alpha_k(\beta),\label{distcirc}\\
\Bigl(\prod(1+\alpha_k)\Bigr)\circ\beta&=\prod(1+\alpha_k(\beta)),\label{distcirc2}\\
\alpha\circ(\beta\circ\gamma)&=(\alpha\circ\beta)\circ\gamma.\label{associativ}
\end{align} 
\end{Lem}
\begin{proof}
Since $|\alpha_k(\beta)|\le|\alpha_k|\to 0$ for $k\to\infty$, all series are well-defined. Using the notation from \eqref{infsums} we deduce:
\[
\Bigl(\sum\alpha_k\Bigr)\circ\beta=\sum_{n=0}^\infty\Bigl(\sum_{k=1}^\infty a_{k,n}\Bigr)\beta^n=\sum_{k=1}^\infty\Bigl(\sum_{n=0}^\infty a_{k,n}\beta^n\Bigr)=\sum\alpha_k(\beta).
\]
We begin proving \eqref{distcirc2} with only two factors, say $\alpha_1=\sum a_nX^n$ and $\alpha_2=\sum b_nX^n$:
\[(\alpha_1\alpha_2)\circ\beta=\sum_{n=0}^\infty\Bigl(\sum_{k=0}^na_kb_{n-k}\Bigr)\beta^n=\sum_{n=0}^\infty\sum_{k=0}^n(a_k\beta^k)(b_{n-k}\beta^{n-k})=(\alpha_1\circ\beta)(\alpha_2\circ\beta).\]
Inductively, \eqref{distcirc2} holds for finitely many factors. 
Hence, 
\begin{align*}
\Bigl|\Bigl(\prod(1+\alpha_k)\Bigr)\circ\beta-\prod_{k=1}^n(1+\alpha_k(\beta))\Bigr|&=\Bigl|\Bigl(\prod_{k=1}^\infty(1+\alpha_k)-\prod_{k=1}^n(1+\alpha_k)\Bigr)\circ\beta\Bigr|\\
&\le\Bigl|\prod_{k=1}^\infty(1+\alpha_k)-\prod_{k=1}^n(1+\alpha_k)\Bigr|\xrightarrow{n\to\infty}0.
\end{align*}
Finally, setting $\alpha=\sum a_nX^n$ we have
\[
\alpha\circ(\beta\circ\gamma)=\sum a_n(\beta\circ\gamma)^n\overset{\eqref{distcirc2}}{=}\sum a_n(\beta^n\circ\gamma)\overset{\eqref{distcirc}}{=}\Bigl(\sum a_n\beta^n\Bigr)\circ\gamma=(\alpha\circ\beta)\circ\gamma.\qedhere
\]
\end{proof}

We warn the reader that in general 
\[\alpha\circ\beta\ne\beta\circ\alpha,\qquad \alpha\circ(\beta\gamma)\ne(\alpha\circ\beta)(\alpha\circ\gamma),\qquad \alpha\circ(\beta+\gamma)\ne\alpha\circ\beta+\alpha\circ\gamma.\] 
Nevertheless, the last statement can be corrected for the exponential series (\autoref{lemfunc}).

\begin{Thm}\label{revgroup}
The set $K[[X]]^\circ:=(X)\setminus(X^2)\subseteq K[[X]]$\index{*KXXcirc@$K[[X]]^\circ$} forms a group with respect to $\circ$.
\end{Thm}
\begin{proof}
Let $\alpha,\beta\in K[[X]]^\circ$. Then $\alpha(\beta)\in K[[X]]^\circ$, i.\,e. $K[[X]]^\circ$ is closed under $\circ$. The associativity holds by \eqref{associativ}.
By definition, $X\in K[[X]]^\circ$ and $X\circ\alpha=\alpha=\alpha\circ X$. 

To construct inverses we argue as in \autoref{leminv}. Let $\alpha^k=\sum_{n=0}^\infty a_{k,n}X^n$ for $k\in\NN_0$. Since $a_{1,0}=0$, also $a_{k,n}=0$ for $n<k$ and $a_{n,n}=a_{1,1}^n\ne 0$. We define recursively $b_0:=0$, $b_1:=\frac{1}{a_{1,1}}\ne 0$ and 
\[b_n:=-\frac{1}{a_{n,n}}\sum_{k=1}^{n-1}b_ka_{k,n}\]
for $n\ge 2$. Setting $\beta:=\sum b_nX^n\in K[[X]]^\circ$, we obtain
\[\beta(\alpha)=\sum_{k=0}^\infty b_k\alpha^k=\sum_{k=1}^\infty\sum_{n=0}^\infty b_ka_{k,n}X^n=\sum_{n=0}^\infty\Bigl(\sum_{k=1}^nb_ka_{k,n}\Bigr)X^n=X.\]
Now replacing $\alpha$ by $\beta$, we find $\gamma\in K[[X]]^\circ$ such that $\gamma\circ\beta=X$. Hence,
\[\gamma=\gamma\circ X=\gamma\circ \beta\circ\alpha=X\circ\alpha=\alpha\] 
and $\alpha\circ\beta=X$. 
\end{proof}

For $\alpha\in K[[X]]^\circ$, we call the unique $\beta\in K[[X]]^\circ$ with $\alpha(\beta)=X=\beta(\alpha)$ the \emph{reverse}\index{reverse}\index{power series!reverse} of $\alpha$. To avoid confusion with the inverse $\alpha^{-1}$ (which is not defined here), we refrain from introducing a symbol for the reverse.

\begin{Ex}\hfill
\begin{enumerate}[(i)]
\item Let $\alpha$ be the reverse of $X+X^2+\ldots=\frac{X}{1-X}$. Then 
\[X=\frac{\alpha}{1-\alpha}\]
and it follows that $\alpha=\frac{X}{1+X}=X-X^2+X^3-\ldots$. This is an example of a \emph{Möbius transformation}.\index{Möbius transformation}
In general, it is much harder to find a closed-form expression for the reverse. We do so for the exponential series with the help of formal derivatives (\autoref{deflog}). Later we provide the explicit Lagrange--Bürmann inversion formula (\autoref{lagrange}) using the machinery of Laurent series.

\item For the field $\FF_p$ with $p$ elements (where $p$ is a prime), the subgroup $N_p:=X+(X^2)$\index{*Np@$N_p$} of $\FF_p[[X]]^\circ$ is called \emph{Nottingham group}.\index{Nottingham group} It has been shown by Leedham-Green\index{Leedham-Green} and Weiss\index{Weiss} (as mentioned in \cite{Nottingham}) that every finite $p$-group is a subgroup of $N_p$, so it must have a very rich structure.
Let $\alpha^{\circ 1}:=\alpha$ and $\alpha^{\circ n}:=\alpha\circ\alpha^{\circ(n-1)}$\index{*Acircn@$\alpha^{\circ n}$} for $\alpha\in N_p$ and $n\ge2$. \emph{Sen's theorem}\index{Sen's theorem} \cite{Sen} asserts that 
\[\inf(\alpha^{\circ p^{n-1}}-X)\equiv\inf(\alpha^{\circ p^n}-X)\pmod{p^n}\] 
for $n\ge 1$ as long as $\alpha^{\circ p^n}\ne X$.
\end{enumerate}
\end{Ex}

\begin{A}
Compute the “first” coefficients of the reverse of $X-X^3\in\CC[[X]]^\circ$. Identify a pattern by using \url{https://oeis.org/}.
\end{A}

\begin{Lem}[Functional equation]\label{lemfunc}\index{functional equation!for exponential series}
For every null sequence $\alpha_1,\alpha_2,\ldots\in (X)\subseteq\CC[[X]]$,
\begin{equation}\label{func2}
\boxed{\exp\Bigl(\sum\alpha_k\Bigr)=\prod\exp(\alpha_k).}
\end{equation}
In particular, $\exp(kX)=\exp(X)^k$ for $k\in\ZZ$.
\end{Lem}
\begin{proof}
Since $\sum\alpha_k\in (X)$ and $\exp(\alpha_k)\in 1+\alpha_k+\frac{\alpha_k^2}{2}+\ldots$, both sides of \eqref{func2} are well-defined. For two summands $\alpha,\beta\in (X)$ we compute
\begin{align*}
\exp(\alpha+\beta)&=\sum\frac{(\alpha+\beta)^n}{n!}=\sum_{n=0}^\infty\sum_{k=0}^n\binom{n}{k}\frac{\alpha^k\beta^{n-k}}{n!}\\
&=\sum_{n=0}^\infty\sum_{k=0}^n\frac{\alpha^k\beta^{n-k}}{k!(n-k)!}=\sum\frac{\alpha^n}{n!}\cdot\sum\frac{\beta^n}{n!}=\exp(\alpha)\exp(\beta).
\end{align*}
By induction, we obtain \eqref{func2} for finitely many summands. This implies
\begin{align*}
\Bigl|\exp\Bigl(\sum\alpha_k\Bigr)-\prod_{k=1}^n\exp(\alpha_k)\Bigr|&=\Bigl|\exp\Bigl(\sum_{k=1}^n\alpha_k+\sum_{k=n+1}^\infty\alpha_k\Bigr)-\exp\Bigl(\sum_{k=1}^n\alpha_k\Bigr)\Bigr|\\
&=\Bigl|\exp\Bigl(\sum_{k=1}^n\alpha_k\Bigr)\Bigr|\Bigl|\exp\Bigl(\sum_{k=n+1}^\infty\alpha_k\Bigr)-1\Bigr|\xrightarrow{n\to\infty}0.
\end{align*}

For the second claim let $k\in\NN_0$. Then $\exp(kX)=\exp(X+\ldots+X)=\exp(X)^k$. Since 
\[\exp(kX)\exp(-kX)=\exp(kX-kX)=\exp(0)=1,\] 
we also have $\exp(-kX)=\exp(kX)^{-1}=\exp(X)^{-k}$. 
\end{proof}

\begin{Def}
For $\alpha=\sum a_nX^n\in K[[X]]$ we call 
\[\boxed{\alpha':=\sum_{n=1}^\infty na_nX^{n-1}\in K[[X]]}\]\index{*Ader@$\alpha'$}
the (formal) \emph{derivative}\index{derivative}\index{power series!derivative} of $\alpha$. Moreover, let $\alpha^{(0)}:=\alpha$ and $\alpha^{(n)}:=(\alpha^{(n-1)})'$\index{*Adern@$\alpha^{(n)}$} the $n$-th derivative for $n\in\NN$.
\end{Def}

It seems natural to define formal \emph{integrals}\index{integral} as counterparts, but this is less useful, since in characteristic $0$ we have $\alpha=\beta$ if and only if $\alpha'=\beta'$ and $\alpha(0)=\beta(0)$. 

\begin{Ex}
As expected we have $1'=0$, $X'=1$ as well as
\[\exp(X)'=\sum_{n=1}^\infty n\frac{X^{n-1}}{n!}=\sum_{n=0}^\infty\frac{X^n}{n!}=\exp(X).\]
Note however, that $(X^p)'=0$ if $K$ has characteristic $p$.
\end{Ex}

In characteristic $0$, derivatives provide a convenient way to extract coefficients of power series. 
For $\alpha=\sum a_nX^n\in \CC[[X]]$ we see that
$\alpha^{(0)}(0)=\alpha(0)=a_0$, $\alpha'(0)=a_1$, $\alpha''(0)=2a_2,\ldots,\alpha^{(n)}(0)=n!a_n$. Hence, \emph{Taylor's theorem}\index{Taylor's theorem} (more precisely, the \emph{Maclaurin series})\index{Maclaurin series} holds
\begin{equation}\label{taylor}
\boxed{\alpha=\sum_{n=0}^\infty\frac{\alpha^{(n)}(0)}{n!}X^n.}
\end{equation}
Over arbitrary fields we are not allowed to divide by $n!$. Alternatively, one may use the $k$-th \emph{Hasse derivative}\index{Hasse derivative} defined by
\[H^k(\alpha):=\sum_{n=k}^\infty\binom{n}{k}a_nX^{n-k}\]\index{*Hk@$H^k(\alpha)$}
(the integer $\binom{n}{k}$ can be embedded in any field). Note that $k!H^k(\alpha)=\alpha^{(k)}$ and $\alpha=\sum_{n=0}^\infty H^n(\alpha)(0)X^n$. In the following we restrict ourselves to complex power series.

\begin{Lem}\label{der}
For $\alpha,\beta\in \CC[[X]]$ and every null sequence $\alpha_1,\alpha_2,\ldots\in \CC[[X]]$ the following rules hold:
\begin{align*}
\Bigl(\sum\alpha_k\Bigr)'&=\sum\alpha_k'&&(\emph{sum rule})\index{sum rule},\\
(\alpha\beta)'&=\alpha'\beta+\alpha\beta'&&(\emph{(finite) product rule})\index{product rule},\\
\Bigl(\prod(1+\alpha_k)\Bigr)'&=\prod(1+\alpha_k)\sum\frac{\alpha_k'}{1+\alpha_k},&&(\emph{(infinite) product rule}),\\
\Bigl(\frac{\alpha}{\beta}\Bigr)'&=\frac{\alpha'\beta-\alpha\beta'}{\beta^2}&&(\emph{quotient rule}),\index{quotient rule}\\
(\alpha\circ\beta)'&=\alpha'(\beta)\beta'&&(\emph{chain rule}).\index{chain rule}
\end{align*}
\end{Lem}
\begin{proof}\hfill
\begin{enumerate}[(i)]
\item\label{sumr} Using the notation from \eqref{infsums}, we have
\[\Bigl(\sum\alpha_k\Bigr)'=\Bigl(\sum_{n=0}^\infty\sum_{k=1}^\infty a_{k,n}X^n\Bigr)'=\sum_{n=1}^\infty\sum_{k=1}^\infty n a_{k,n}X^{n-1}=\sum_{k=1}^\infty\Bigl(\sum_{n=1}^\infty na_{k,n}X^{n-1}\Bigr)=\sum\alpha_k'.\]

\item\label{produktr} By \eqref{sumr} we may assume $\alpha=X^k$ and $\beta=X^l$. In this case,
\[(\alpha\beta)'=(X^{k+l})'=(k+l)X^{k+l-1}=kX^{k-1}X^{l}+lX^{l-1}X^{k}=\alpha'\beta+\beta'\alpha.\]

\item\label{produktinf} Without loss of generality, suppose $\alpha_k\ne-1$ for all $k\in\NN$ (otherwise both sides vanish). 
Let $|\alpha_k|<2^{-N-1}$ for all $k>n$. The coefficient of $X^N$ on both sides of the equation depends only on $\alpha_1,\ldots,\alpha_n$.
From \eqref{produktr} we verify inductively:
\[\Bigl(\prod_{k=1}^n(1+\alpha_k)\Bigr)'=\prod_{k=1}^n(1+\alpha_k)\sum_{l=1}^n\frac{\alpha_l'}{1+\alpha_l}\]
for all $n\in\NN$. Now the claim follows with $N\to\infty$.

\item By \eqref{produktr},
\[\alpha'=\Bigl(\frac{\alpha}{\beta}\beta\Bigr)'=\Bigl(\frac{\alpha}{\beta}\Bigr)'\beta+\frac{\alpha\beta'}{\beta}.\]

\item By \eqref{produktinf}, the \emph{power rule}\index{power rule} $(\alpha^n)'=n\alpha^{n-1}\alpha'$ holds for $n\in\NN_0$. The sum rule implies
\[(\alpha\circ\beta)'=\Bigl(\sum a_n\beta^n\Bigr)'=\sum a_n(\beta^n)'=\sum_{n=1}^\infty na_n\beta^{n-1}\beta'=\alpha'(\beta)\beta'.\qedhere\]
\end{enumerate}
\end{proof}

The product rule implies the rather trivial \emph{factor rule}\index{factor rule} $(\lambda\alpha)'=\lambda\alpha'$ as well as \emph{Leibniz' rule}\index{Leibniz' rule} \[(\alpha\beta)^{(n)}=\sum_{k=0}^n\binom{n}{k}\alpha^{(k)}\beta^{(n-k)}\] 
for $\lambda\in\CC$ and $\alpha,\beta\in\CC[[X]]$. A generalized version of the latter and a chain rule for higher derivatives are proven in \autoref{secmult}.

\begin{A}
Let $\alpha,\beta\in(X)$ such that $\beta\notin(X^2)$. Prove \emph{L'Hôpital's rule}\index{L'Hôpital's rule} $\frac{\alpha}{\beta}(0)=\frac{\alpha'(0)}{\beta'(0)}$.
\end{A}

\begin{Ex}\label{deflog}
Define the (formal) \emph{logarithm}\index{logarithm} by the \emph{Mercator series}\index{Mercator series}
\[\boxed{\log(1+X):=\sum_{n=1}^\infty\frac{(-1)^{n-1}}{n}X^n=X-\frac{X^2}{2}+\frac{X^3}{3}\mp\ldots\in \CC[[X]].}\]\index{*Log@$\log(1+X)$}
By \autoref{revgroup}, $\alpha:=\exp(X)-1$ possesses a reverse and $\log(\exp(X))=\log(1+\alpha)\in\CC[[X]]^\circ$.
Since
\[\log(1+X)'=1-X+X^2\mp\ldots=\sum(-X)^n=\frac{1}{1+X},\]
the chain rule yields
\[\log(1+\alpha)'=\frac{\alpha'}{1+\alpha}=\frac{\exp(X)}{\exp(X)}=1.\]
This shows that $\log(\exp(X))=X$. Therefore, $\log(1+X)$ is the reverse of $\alpha=\exp(X)-1$ as expected from analysis. Equivalently, $\exp(\log(1+X))=1+X$. 
Moreover, $\log(1-X)=-\sum_{n=1}^\infty \frac{X^n}{n}$.  
\end{Ex}

The only reason why we called the power series $\log(1+X)$ instead of $\log(X)$ or just $\log$ is to keep the analogy to the natural logarithm (as a real function).

\begin{Lem}[Functional equation]\index{functional equation!for logarithm}
For every null sequence $\alpha_1,\alpha_2,\ldots\in (X)\subseteq\CC[[X]]$,
\begin{equation}\label{funclog}
\boxed{\log\Bigl(\prod(1+\alpha_k)\Bigr)=\sum\log(1+\alpha_k).}
\end{equation}
\end{Lem}
\begin{proof}
\begin{align*}
\log\Bigl(\prod(1+\alpha_k)\Bigr)&=\log\Bigl(\prod\exp(\log(1+\alpha_k))\Bigr)\overset{\eqref{func2}}{=}\log\Bigl(\exp\Bigl(\sum\log(1+\alpha_k)\Bigr)\Bigr)\\
&=\sum\log(1+\alpha_k).\qedhere
\end{align*}
\end{proof}

\begin{Ex}\label{bsplog}
By \eqref{funclog},
\[
\log\Bigl(\frac{1}{1-X}\Bigr)=-\log(1-X)=\sum_{n=1}^\infty\frac{X^n}{n}.
\]
\end{Ex}

\begin{Def}\label{defpower}
For $c\in\CC$ and $\alpha\in (X)$ let
\[\boxed{(1+\alpha)^c:=\exp\bigl(c\log(1+\alpha)\bigr).}\]\index{*Ac@$(1+\alpha)^c$}
If $c=1/k$ for some $k\in\NN$, we write more customary $\sqrt[k]{1+\alpha}:=(1+\alpha)^{1/k}$ and in particular $\sqrt{1+\alpha}:=\sqrt[2]{1+\alpha}$.
\end{Def}

By \autoref{lemfunc}, 
\[(1+\alpha)^c(1+\alpha)^d=\exp\bigl(c\log(1+\alpha)+d\log(1+\alpha)\bigr)=(1+\alpha)^{c+d}\]
for every $c,d\in\CC$ as expected. Consequently, $\sqrt[k]{1+\alpha}^k=1+\alpha$ for $k\in\NN$, i.\,e. $\sqrt[k]{1+\alpha}$ is a $k$-th \emph{root}\index{root} of $1+\alpha$ with constant term $1$. 
Suppose that $\beta\in\CC[[X]]$ also satisfies $\beta^k=1+\alpha$ and has constant term $1$. Then
\[\beta=\exp(\log(\beta))=\exp\bigl(k^{-1}\log(1+\alpha)\bigr)=\sqrt[k]{1+\alpha}.\]
Consequently, $\sqrt[k]{1+\alpha}$ is the unique $k$-th root of $1+\alpha$ with constant term $1$.

\begin{Ex}
Suppose that $\alpha=\sum_{i=1}^\infty a_iX^i$ has integral coefficients. Let $k=p_1^{r_1}\ldots p_l^{r_l}$ be the prime factorization of $k$. It was recently shown by Pomerat--Straub~\cite{PomeratStraub} that $\sqrt[k]{1+\alpha}\in\ZZ[[X]]$ if and only if 
\[\alpha\equiv (1+a_{p_i^{r_i}}X+a_{2p_i^{r_i}}X^2+\ldots)^{p_i^{r_i}}\pmod{p_i^{r_i+1}}\]
for $i=1,\ldots,l$. In particular, $\sqrt[p]{1+aX}\in\ZZ[[X]]\iff p^2\mid a$ for every prime $p$ and $a\in\ZZ$. 
\end{Ex}

The inexperienced reader may find the following exercise helpful.

\begin{A}
Check that the following power series in $\CC[[X]]$ are well-defined:\index{trigonometric series}
\begin{align*}
\sin(X)&:=\sum_{n=0}^\infty\frac{(-1)^n}{(2n+1)!}X^{2n+1},\index{*Sin@$\sin(X)$}&\cos(X)&:=\sum_{n=0}^\infty\frac{(-1)^n}{(2n)!}X^{2n},\index{*Cos@$\cos(X)$}\\
\tan(X)&:=\frac{\sin(X)}{\cos(X)},\index{*Tan@$\tan(X)$}&\sinh(X)&:=\sum_{k=0}^\infty\frac{X^{2k+1}}{(2k+1)!},\index{*Sinh@$\sinh(X)$}\\
\arcsin(X)&:=\sum_{n=0}^\infty\frac{(2n)!}{(2^nn!)^2}\frac{X^{2n+1}}{2n+1},\index{*ArcSin@$\arcsin(X)$}&\arctan(X)&:=\sum_{k=0}^\infty\frac{(-1)^k}{2k+1}X^{2k+1}.\index{*Arctan@$\arctan(X)$}
\end{align*}
Show that
\begin{enumerate}[(a)]
\item\label{euler} \textup(\textsc{Euler}'s formula\textup)\index{Euler's formula} $\exp(\ii X)=\cos(X)+\ii\sin(X)$ where $\ii=\sqrt{-1}\in\CC$.
\item $\sin(2X)=2\sin(X)\cos(X)$ and $\cos(2X)=\cos(X)^2-\sin(X)^2$.\\
\textit{Hint:} Use \eqref{euler} and separate real from non-real coefficients.
\item \textup(\textsc{Pythagorean} identity\textup)\index{Pythagorean identity} $\cos(X)^2+\sin(X)^2=1$.
\item $\sinh(X)=\frac{1}{2}(\exp(X)-\exp(-X))$. 
\item $\sin(X)'=\cos(X)$ and $\cos(X)'=-\sin(X)$.
\item $\arctan\circ\tan=X$.\\
\textit{Hint:} Mimic the argument for $\log(1+X)$.
\item $\arctan(X)=\frac{\ii}{2}\log\Bigl(\frac{\ii+X}{\ii-X}\Bigr)$.
\item $\arcsin(X)'=\frac{1}{\sqrt{1-X^2}}$.
\item $\arcsin\circ \sin=X$.
\end{enumerate}
\end{A}

\section{Laurent series}\label{seclaurent}

Every integral domain $R$ can be embedded into its \emph{field of fractions}\index{field of fractions} consisting of the formal fractions $\frac{r}{s}$ where $r,s\in R$ and $s\ne 0$. For our ring $K[[X]]$ these fractions have a more convenient shape.

\begin{Def}
A (formal) \emph{Laurent series}\index{Laurent series} in the indeterminate $X$ over the field $K$ is a sum of the form
\[\alpha=\sum_{k=m}^\infty a_kX^k\]
where $m\in\ZZ$ and $a_k\in K$ for $k\ge m$ (i.\,e. we allow $X$ to negative powers). 
We often write $\alpha=\sum_{k=-\infty}^\infty a_kX^k$ assuming that $\inf(\alpha)=\inf\{k\in\ZZ:a_k\ne 0\}$ exists.
The set of all Laurent series over $K$ is denoted by $K((X))$.\index{*KXX2@$K((X))$}
Laurent series can be added and multiplied like power series:
\[\alpha+\beta=\sum_{k=-\infty}^\infty(a_k+b_k)X^k,\qquad\alpha\beta=\sum_{k=-\infty}^\infty\Bigl(\sum_{l=-\infty}^\infty a_lb_{k-l}\Bigr)X^k\]
(one should check that the inner sum is finite). Moreover, the norm $|\alpha|$ and the derivative $\alpha'$ are defined as for power series. 
\end{Def}

If a Laurent series is a finite sum, it is naturally called a \emph{Laurent polynomial}.\index{Laurent polynomial} The ring of Laurent polynomials is denoted by $K[X,X^{-1}]$,\index{*KYY@$K[X,X^{-1}]$} but plays no role in the following. In analysis one allows double infinite sums, but then the product is no longer well defined as in $\Bigl(\sum_{n=-\infty}^\infty X^n\Bigr)^2$. 

\begin{Thm}
The field of fractions of $K[[X]]$ is naturally isomorphic to $K((X))$. In particular, $K((X))$ is a field.
\end{Thm}
\begin{proof}
Repeating the proof of \autoref{intdom} shows that $K((X))$ is a commutative ring. Let $\alpha\in K((X))\setminus\{0\}$ and $k:=\inf(\alpha)$. By \autoref{leminv}, $X^{-k}\alpha\in K[[X]]^\times$. Hence, $X^{-k}(X^{-k}\alpha)^{-1}\in K((X))$ is the inverse of $\alpha$. This shows that $K((X))$ is a field. By the universal property of the field of fractions $Q(K[[X]])$, the embedding $K[[X]]\subseteq K((X))$ extends to a (unique) field monomorphism $f\colon Q(K[[X]])\to K((X))$. If $k=\inf(\alpha)<0$, then $f\bigl(\frac{X^{-k}\alpha}{X^{-k}}\bigr)=\alpha$ and $f$ is surjective.  
\end{proof}

Of course, we will view $K[[X]]$ as a subring of $K((X))$. In fact, $K[[X]]$ is the \emph{valuation ring}\index{valuation ring} of $K((X))$, i.\,e. $K[[X]]=\{\alpha\in K((X)):|\alpha|\le 1\}$. 
The field $K((X))$ should not be confused with the field of \emph{rational functions}\index{rational function} $K(X)$,\index{*KX2@$K(X)$} which is the field of fractions of $K[X]$. 

If $\alpha\in K((X))$ and $\beta\in K[[X]]^\circ$, the substitute $\alpha(\beta)$ is still well-defined and \autoref{lemcomp} remains correct ($\alpha$ deviates from a power series by only finitely many terms).  

\begin{A}
Compute $(X+X^{-1})^{-1}\in\CC((X))$ as a Laurent series.
\end{A}

\begin{Def}
The (formal) \emph{residue}\index{residue} of $\alpha=\sum a_kX^k\in K((X))$ is defined by $\res(\alpha):=a_{-1}$. \index{*Res@$\res(\alpha)$}
\end{Def}

The residue is a $K$-linear map such that $\res(\alpha')=0$ for all $\alpha\in K((X))$. 

\begin{Lem}\label{lemres}
For $\alpha,\beta\in \CC((X))$ we have
\begin{flalign*}
(i)&&\res(\alpha'\beta)&=-\res(\alpha\beta'),\\
(ii)&&\res(\alpha'/\alpha)&=\inf(\alpha)&&(\alpha\ne 0),\\
(iii)&&\res(\alpha)\inf(\beta)&=\res(\alpha(\beta)\beta')&&(\beta\in (X)).
\end{flalign*}
\end{Lem}
\begin{proof}\hfill
\begin{enumerate}[(i)]
\item\label{res1} This follows from the product rule
\[0=\res((\alpha\beta)')=\res(\alpha'\beta)+\res(\alpha\beta').\]

\item\label{res2} Let $\alpha=X^k\gamma$ with $k=\inf(\alpha)$ and $\gamma\in \CC[[X]]^\times$. Then
\[\frac{\alpha'}{\alpha}=\frac{kX^{k-1}\gamma+X^k\gamma'}{X^k\gamma}=kX^{-1}+\gamma'\gamma^{-1}.\]
Since $\gamma^{-1}\in \CC[[X]]$, it follows that $\res(\alpha'/\alpha)=k=\inf(\alpha)$.

\item Since $\res$ is a linear map, we may assume that $\alpha=X^k$. If $k\ne -1$, then \[\res(\alpha(\beta)\beta')=\res(\beta^k\beta')=\frac{1}{k+1}\res\bigl((\beta^{k+1})'\bigr)=0=\res(\alpha)=\res(\alpha)\inf(\beta).\]
If $k=-1$, then
\[\res(\alpha(\beta)\beta')=\res(\beta'/\beta)\overset{\eqref{res2}}{=}\inf(\beta)=\res(\alpha)\inf(\beta).\qedhere\]
\qedhere
\end{enumerate}
\end{proof}

\begin{Thm}[\textsc{Lagrange--Bürmann}'s inversion formula]\label{lagrange}\index{Lagrange--Bürmann's inversion formula}
The reverse of $\alpha\in\CC[[X]]^\circ$ is
\[\boxed{\sum_{k=1}^\infty \frac{\res(\alpha^{-k})}{k}X^k.}\]
\end{Thm}
\begin{proof}
The proof is influenced by \cite{Gessel}.
Let $\beta\in \CC[[X]]^\circ$ be the reverse of $\alpha$, i.\,e. $\alpha(\beta)=X$. From $\alpha\in \CC[[X]]^\circ$ we know that $\alpha\ne 0$. In particular, $\alpha$ is invertible in $\CC((X))$. By \autoref{lemcomp}, we have $\alpha^{-k}(\beta)=X^{-k}$. 
Now the coefficient of $X^k$ in $\beta$ turns out to be
\[
\frac{1}{k}\res(kX^{-k-1}\beta)=-\frac{1}{k}\res\bigl((X^{-k})'\beta\bigr)=\frac{1}{k}\res(X^{-k}\beta')=\frac{1}{k}\res\bigl(\alpha^{-k}(\beta)\beta'\bigr)=\frac{1}{k}\res(\alpha^{-k})
\]
by \autoref{lemres}.
\end{proof}

Since \autoref{lagrange} is actually a statement about power series, it should be mentioned that $\res(\alpha^{-k})$ is just the coefficient of $X^{k-1}$ in the power series $(X/\alpha)^k$. This interpretation will be used in our generalization to higher dimensions in \autoref{lagrange2}.

\section{The main theorems}\label{secmain}

For $c\in\CC$ and $k\in\NN$ we extend the definition of usual binomial coefficient by
\[\binom{c}{k}:=\frac{c(c-1)\ldots(c-k+1)}{k!}\in\CC\] \index{*Binom@$\binom{c}{k}$}
(it is useful to know that numerator and denominator both have exactly $k$ factors). 
The next theorem is a vast generalization of the binomial theorem (take $c\in\NN$) and the geometric series (take $c=-1$). 

\begin{Thm}[\textsc{Newton}'s binomial theorem]\label{newton}\index{Newton's binomial theorem}
For $\alpha\in (X)$ and $c\in\CC$ the following holds
\begin{equation}\label{newtoneq}
\boxed{(1+\alpha)^c=\sum_{k=0}^\infty\binom{c}{k}\alpha^k.}
\end{equation}
\end{Thm}
\begin{proof}
It suffices to prove the equation for $\alpha=X$ (we may substitute $X$ by $\alpha$ afterward). By the chain rule,
\[\bigl((1+X)^c\bigr)'=\exp(c\log(1+X))'=c\frac{(1+X)^c}{1+X}=c(1+X)^{c-1}\]
and inductively, $\bigl((1+X)^c\bigr)^{(k)}=c(c-1)\ldots(c-k+1)(1+X)^{c-k}$. Now the claim follows from Taylor's theorem~\eqref{taylor}.
\end{proof}

A striking application of \autoref{newton} will be given in \autoref{Turan}.

\begin{Ex}
Let $\zeta\in\CC$ be an $n$-th root of unity and let $\alpha:=(1+X)^\zeta-1\in (X)$. Then \[\alpha\circ\alpha=\bigl(1+(1+X)^\zeta-1\bigr)^\zeta-1=(1+X)^{\zeta^2}-1\] 
and inductively $\alpha\circ\ldots\circ\alpha=(1+X)^{\zeta^n}-1=X$. In particular, the order of $\alpha$ in the group $\CC[[X]]^\circ$ divides $n$. Thus, in contrast to the group $K[[X]]^\times$ studied in \autoref{leminv}, the group $\CC[[X]]^\circ$ possesses “interesting” elements of finite order.
\end{Ex}

Since we do not call our indeterminate $q$ (as in many sources), it makes no sense to introduce the $q$-\emph{Pochhammer symbol}\index{Pochhammer symbol} $(q;q)_n$. Instead, we devise a non-standard notation in reminiscence of the binomial coefficient.

\begin{Def}\label{defgauss}
For $n\in\NN_0$ let $X^n!:=(1-X)(1-X^2)\ldots(1-X^n)$.\index{*Xnfac@$X^n$#!} For $0\le k\le n$ we call
\[\gauss{n}{k}:=\frac{X^n!}{X^k!X^{n-k}!}=\frac{1-X^n}{1-X^k}\ldots\frac{1-X^{n-k+1}}{1-X}\in\CC[[X]]\]\index{*Gaus@$\gauss{n}{k}$}
a \emph{Gaussian coefficient}.\index{Gaussian coefficient} If $k<0$ or $k>n$ let $\gauss{n}{k}:=0$.
\end{Def}

As for the binomial coefficients, we have $\gauss{n}{0}=\gauss{n}{n}=1$ and $\gauss{n}{k}=\gauss{n}{n-k}$ for all $n\in\NN_0$ and $k\in\ZZ$. Moreover, $\gauss{n}{1}=\frac{1-X^n}{1-X}=1+X+\ldots+X^{n-1}$.
The familiar recurrence formula for binomial coefficients needs to be altered as follows.

\begin{Lem}
For $n\in\NN_0$ and $k\in\ZZ$,
\begin{equation}\label{lemgauss}
\gauss{n+1}{k}=X^k\gauss{n}{k}+\gauss{n}{k-1}=\gauss{n}{k}+X^{n+1-k}\gauss{n}{k-1}.
\end{equation}
\end{Lem}
\begin{proof}
For $k>n+1$ or $k<0$ all parts are $0$. Similarly, for $k=n+1$ or $k=0$ both sides equal $1$. Finally, for $1\le k\le n$ it holds that
\begin{align*}
X^k\gauss{n}{k}+\gauss{n}{k-1}&=\Bigl(X^k\frac{1-X^{n-k+1}}{1-X^k}+1\Bigr)\frac{X^n!}{X^{k-1}!X^{n-k+1}!}=\frac{1-X^{n+1}}{1-X^k}\frac{X^n!}{X^{k-1}!X^{n+1-k}!}\\
&=\gauss{n+1}{k}=\gauss{n+1}{n+1-k}=X^{n+1-k}\gauss{n}{n+1-k}+\gauss{n}{n-k}\\
&=\gauss{n}{k}+X^{n+1-k}\gauss{n}{k-1}.\qedhere
\end{align*}
\end{proof}

Since $\gauss{n}{0}$ and $\gauss{n}{1}$ are polynomials, \eqref{lemgauss} shows inductively that all Gaussian coefficients are polynomials. We may therefore evaluate $\gauss{n}{k}$ at $X=1$. Indeed \eqref{lemgauss} becomes the recurrence for the binomial coefficients if $X=1$. Hence $\gauss{n}{k}(1)=\binom{n}{k}$. This can be seen more directly by writing
\[\gauss{n}{k}=\frac{\frac{1-X^n}{1-X}\ldots\frac{1-X^{n-k+1}}{1-X}}{\frac{1-X^k}{1-X}\ldots\frac{1-X}{1-X}}=\frac{(1+X+\ldots+X^{n-1})\ldots(1+X+\ldots+X^{n-k})}{(1+X+\ldots+X^{k-1})\ldots(1+X)1}.\]
We will interpret the coefficients of $\gauss{n}{k}$ in \autoref{partseries}.

\begin{Ex}
\[\gauss{4}{2}=X^2\gauss{3}{2}+\gauss{3}{1}=X^2(1+X+X^2)+(1+X+X^2)=1+X+2X^2+X^3+X^4.\]
\end{Ex}

\begin{Thm}[\textsc{Gauss}' binomial theorem]\label{GaussBT}\index{Gauss' binomial theorem}
For $n\in\NN$ and $\alpha\in\CC((X))$ the following holds
\begin{empheq}[box=\fbox]{align}
\prod_{k=0}^{n-1}(1+\alpha X^k)&=\sum_{k=0}^n\gauss{n}{k}\alpha^kX^{\binom{k}{2}},\label{G0}\\
\prod_{k=0}^\infty(1+\alpha X^k)&=\sum_{k=0}^\infty\frac{\alpha^kX^{\binom{k}{2}}}{X^k!}.\label{E1}
\end{empheq}
\end{Thm}
\begin{proof}\hfill
\begin{enumerate}[(i)]
\item We argue by induction on $n$. For $n=1$ both sides become $1+\alpha$. For the induction step we let all sums run from $-\infty$ to $\infty$ (this will not change their value, but makes index shifts much more transparent):
\begin{align*}
\prod_{k=0}^{n}(1+\alpha X^k)&=(1+\alpha X^n)\sum_{k=-\infty}^\infty\gauss{n}{k}\alpha^kX^{\binom{k}{2}}\\[-6mm]
&=\sum\gauss{n}{k}\alpha^kX^{\binom{k}{2}}+\sum\gauss{n}{k}\alpha^{k+1}X^{n-k}X^{\overbrace{\scriptstyle{\binom{k}{2}+k}}^{\binom{k+1}{2}}}\\
&=\sum\gauss{n}{k}\alpha^kX^{\binom{k}{2}}+\sum X^{n+1-k}\gauss{n}{k-1}\alpha^{k}X^{\binom{k}{2}}\\
&\overset{\eqref{lemgauss}}{=}\sum\gauss{n+1}{k}\alpha^kX^{\binom{k}{2}}.
\end{align*}

\item Since $\inf\bigl(\alpha^kX^{\binom{k}{2}}\bigr)=\frac{1}{2}(k^2-k)+k\inf(\alpha)\to\infty$, the right hand side converges.
For $m\in\ZZ$, the coefficient of $X^m$ in the left hand side of \eqref{E1} depends only on 
\[\prod_{k=0}^{n-1}(1+\alpha X^k)\overset{\eqref{G0}}{=}\sum_{k=0}^n\gauss{n}{k}\alpha^kX^{\binom{k}{2}},\]
as long as $n>m-\inf(\alpha)$. 
Moreover, if $n$ is large enough, $X^m$ does not appear in $\frac{\alpha^kX^{\binom{k}{2}}}{X^k!}$ for $k>n$. 
It is therefore enough to show that $X^m$ does not appear in 
\[\sum_{k=0}^n\gauss{n}{k}\alpha^kX^{\binom{k}{2}}-\sum_{k=0}^n\frac{\alpha^kX^{\binom{k}{2}}}{X^k!}=\sum_{k=0}^n\Bigl(\gauss{n}{k}-\frac{1}{X^k!}\Bigr)\alpha^kX^{\binom{k}{2}}.\]
In fact,
\[\Bigl(\gauss{n}{k}-\frac{1}{X^k!}\Bigr)X^{\binom{k}{2}}=\frac{(1-X^n)\ldots(1-X^{n-k+1})-1}{X^k!}X^{\binom{k}{2}}\in (X^{n-k+1+\binom{k}{2}})\subseteq (X^n).\qedhere\]
\end{enumerate}
\end{proof}

\begin{Rmk}\label{important}
Equation~\eqref{G0} is sometimes attributed to Cauchy, while \eqref{E1} is due to Euler.
We emphasize that in the proof of \autoref{GaussBT}, $\alpha$ is treated as a variable independent of $X$. The proof and the statement are therefore still valid if we substitute $X$ by some $\beta\in(X)$ \emph{without} changing $\alpha$ to $\alpha(\beta)$. 
\end{Rmk}

\begin{A}[\textsc{Rothe}'s binomial theorem]\index{Rothe's binomial theorem}
For $n\in\NN$ and $\alpha,\beta\in\CC((X))$ show that
\[\prod_{k=0}^{n-1}(\alpha+\beta X^k)=\sum_{k=0}^n\gauss{n}{k}\alpha^{n-k}\beta^kX^{\binom{k}{2}}.\]
\textit{Hint:} Replace $\alpha$ by $\alpha^{-1}\beta$ in \eqref{G0}.
\end{A}

The special case $(1-X)^{-n}=\sum_{k=0}^\infty\binom{n+k-1}{k}X^k$ of Newton's binomial theorem somehow “inverts” the ordinary binomial theorem $(1+X)^n=\sum_{k=0}^n\binom{n}{k}X^k$. In the same spirit, the following result inverts Gauss' binomial theorem. 
We will encounter many more such “dual pairs” in \autoref{genstir}, \autoref{vieta} and \eqref{mac1}, \eqref{a2}.

\begin{Thm}\label{coreuler}
For all $\alpha\in\CC[[X]]$,
\begin{align}
\prod_{k=1}^n\frac{1}{1-\alpha X^k}&=\sum_{k=0}^\infty\gauss{n+k-1}{k}\alpha^kX^k,\\
\prod_{k=1}^\infty\frac{1}{1-\alpha X^k}&=\sum_{k=0}^\infty\frac{\alpha^kX^k}{X^k!}.\label{E2}
\end{align}
\end{Thm}
\begin{proof}\hfill
\begin{enumerate}[(i)]
\item Induction on $n$: For $n=1$ we obtain the geometric series $\frac{1}{1-\alpha X}=\sum_{k=0}^\infty\alpha^kX^k$. In general:
\begin{align*}
(1-\alpha X^{n+1})\sum_{k=0}^\infty\gauss{n+k}{k}\alpha^kX^k&=\sum_{k=0}^\infty\gauss{n+k}{k}\alpha^kX^k-X^n\sum_{k=0}^\infty\gauss{n+k}{k}\alpha^{k+1}X^{k+1}\\
&=\sum_{k=0}^\infty\biggl(\gauss{n+k}{k}-X^n\gauss{n+k-1}{k-1}\biggr)\alpha^kX^k\\
&\overset{\eqref{lemgauss}}{=}\sum_{k=0}^\infty\gauss{n+k-1}{k}\alpha^kX^k=\prod_{k=1}^n\frac{1}{1-\alpha X^k}.
\end{align*}

\item Replacing $\alpha$ by $-X\alpha$ in \eqref{E1} yields
\[\prod_{k=1}^\infty(1-\alpha X^k)=\prod_{k=0}^\infty(1-\alpha X^{k+1})=\sum_{k=0}^\infty(-1)^k\frac{\alpha^kX^{\binom{k}{2}+k}}{X^k!}.\]
Now we multiply with the right hand side of \eqref{E2}:
\begin{align*}
\sum_{k=0}^\infty(-1)^k&\frac{\alpha^kX^{\binom{k}{2}+k}}{X^k!}\sum_{k=0}^\infty\frac{\alpha^{k}X^k}{X^k!}=\sum_{n=0}^\infty\sum_{k=0}^n(-1)^k\frac{\alpha^{k+n-k}X^{\binom{k}{2}+n}}{X^k!X^{n-k}!}\\
&=\sum_{n=0}^\infty\frac{\alpha^nX^n}{X^n!}\sum_{k=0}^n(-1)^k\gauss{n}{k}X^{\binom{k}{2}}\overset{\ref{GaussBT}}{=}\sum_{n=0}^\infty\frac{\alpha^nX^n}{X^n!}\prod_{k=0}^{n-1}(1-X^k)=1.\qedhere
\end{align*}
\end{enumerate}
\end{proof}

Unlike Gauss' theorem, \autoref{coreuler} only applies to power series, but not to Laurent series.
If $\alpha\in (X)$, we can apply \eqref{E2} with $\alpha X^{-1}$ to obtain
\begin{equation}\label{E3}
\prod_{k=0}^\infty\frac{1}{1-\alpha X^k}=\sum_{k=0}^\infty\frac{\alpha^k}{X^k!}.
\end{equation}

Finally, we are in a position to derive one of the most powerful theorems on power series.

\begin{Thm}[\textsc{Jacobi}'s triple product identity]\label{jacobi}\index{Jacobi's triple product identity}
For every $\alpha\in \CC((X))\setminus\{0\}$ the following holds
\[\boxed{\prod_{k=1}^\infty(1-X^{2k})(1+\alpha X^{2k-1})(1+\alpha^{-1}X^{2k-1})=\sum_{k=-\infty}^\infty \alpha^kX^{k^2}.}\]
\end{Thm}
\begin{proof}
We follow Andrews~\cite{AndrewsJTP}. It is easy to see that both sides of the equation are well-defined Laurent series. By replacing $\alpha$ with $\alpha^{-1}$ (and $k$ by $-k$ on the right hand side) if necessary, we may assume that $\alpha\in\CC[[X]]$. 
According to \autoref{important} we are allowed to substitute $X$ by $X^2$ and simultaneously $\alpha$ by $\alpha^{-1} X$ in \eqref{E1}:
\begin{align*}
\prod_{k=1}^\infty (1+\alpha^{-1} X^{2k-1})&=\prod_{k=0}^\infty (1+\alpha^{-1} X^{2k+1})=\sum_{k=0}^\infty\frac{\alpha^{-k}X^{k^2}}{(1-X^2)\ldots(1-X^{2k})}\\
&=\prod_{k=1}^\infty\frac{1}{1-X^{2k}}\sum_{k=0}^\infty\alpha^{-k}X^{k^2}\prod_{l=0}^\infty(1-X^{2l+2k+2}).
\end{align*}
Since the inner product vanishes for negative $k$, we can extend the summation to all $k\in\ZZ$. 
A second application of \eqref{E1} with $X^2$ instead of $X$ and $-X^{2k+2}$ in the role of $\alpha$ allows us to rewrite the last product of the right hand side. This shows 
\begin{align*}
\prod_{k=1}^\infty (1+\alpha^{-1} X^{2k-1})(1-X^{2k})&=\sum_{k=-\infty}^\infty\alpha^{-k}X^{k^2}\sum_{l=0}^\infty\frac{(-1)^lX^{l^2+l+2kl}}{(1-X^2)\ldots(1-X^{2l})}\\
&=\sum_{l=0}^\infty \frac{(-\alpha X)^l}{(1-X^2)\ldots(1-X^{2l})}\sum_{k=-\infty}^\infty X^{(k+l)^2}\alpha^{-k-l}.
\end{align*}
After the index shift $k\mapsto -k-l$, the inner sum does not depend on $l$ anymore. We then apply \eqref{E3} on the first sum with $X$ replaced by $X^2$ and $-\alpha X\in (X)$ instead of $\alpha$:
\begin{align*}
\prod_{k=1}^\infty (1+\alpha^{-1} X^{2k-1})(1-X^{2k})=\prod_{k=0}^\infty\frac{1}{1+\alpha X^{2k+1}}\sum_{k=-\infty}^\infty X^{k^2}\alpha^{k}=\prod_{k=1}^\infty\frac{1}{1+\alpha X^{2k-1}}\sum_{k=-\infty}^\infty X^{k^2}\alpha^{k}.
\end{align*}
We are done by rearranging terms.
\end{proof}

\begin{Rmk}\label{import2}
Since the above proof is just a combination of \eqref{E1} and \eqref{E3}, we are still allowed to replace $X$ and $\alpha$ individually.
\end{Rmk}

A (somewhat analytical) proof only making use of \eqref{E1} can be found in \cite{Zhu}. There are numerous purely combinatorial proofs like \cite{Kolitsch,Lewis,Sudler,Sylvester,Wright,Zolnowsky}, which are meaningful for formal power series.

\begin{Ex}\hfill
\begin{enumerate}[(i)]
\item Choosing $\alpha\in\{\pm1,X\}$ in \autoref{jacobi} reveals the following elegant identities:
\begin{align}
\prod_{k=1}^\infty(1-X^{2k})(1+X^{2k-1})^2&=\sum_{k=-\infty}^\infty X^{k^2},\label{JTP1}\\
\prod_{k=1}^\infty\frac{(1-X^k)^2}{1-X^{2k}}=\prod_{k=1}^\infty(1-X^{2k})(1-X^{2k-1})^2&=\sum_{k=-\infty}^\infty(-1)^k X^{k^2},\label{JTP2}\\
\prod_{k=1}^\infty(1-X^{2k})(1+X^{2k})^2&=\frac{1}{2}\sum_{k=-\infty}^\infty X^{k^2+k}=\sum_{k=0}^\infty X^{k^2+k},\label{JTP3}
\end{align}
where in \eqref{JTP3} we made use of the bijection $k\mapsto -k-1$ on $\ZZ$. These formulas are needed in the proof of \autoref{thmLJ}.
In \eqref{JTP3} we find $X$ only to even powers. By equating the corresponding coefficients, we may replace $X^2$ by $X$ to obtain
\[\prod_{k=1}^\infty(1-X^{2k})(1+X^k)=\prod_{k=1}^\infty(1-X^{k})(1+X^{k})^2=\sum_{k=0}^\infty X^{\frac{k^2+k}{2}}.\]
A very similar identity will be proved in \autoref{Jhoch3thm}.

\item Relying on \autoref{import2}, we can replace $X$ by $X^3$ and $\alpha$ by $-X$ at the same time in \autoref{jacobi}. This leads to 
\[\prod_{k=1}^\infty (1-X^{6k})(1-X^{6k-2})(1-X^{6k-4})=\sum_{k=-\infty}^\infty (-1)^kX^{3k^2+k}.\]
Substituting $X^2$ by $X$ yields Euler's celebrated \emph{pentagonal number theorem}:\index{pentagonal number theorem}
\begin{equation}\label{EPTN}
\boxed{\prod_{k=1}^\infty(1-X^k)=\prod_{k=1}^\infty(1-X^{3k})(1-X^{3k-1})(1-X^{3k-2})=\sum_{k=-\infty}^\infty(-1)^kX^{\frac{3k^2+k}{2}}.}
\end{equation}
There is a well-known combinatorial proof of \eqref{EPTN} by Franklin, which is reproduced in the influential book by Hardy--Wright~\cite[Section~19.11]{Hardy}.

\item The following formulas arise in a similar manner by substituting $X$ by $X^5$ and selecting $\alpha\in\{-X,-X^3\}$ afterward 
(this is allowed by \autoref{import2}):
\begin{align}
\prod_{k=1}^\infty (1-X^{5k})(1-X^{5k-2})(1-X^{5k-3})=\sum_{k=-\infty}^\infty (-1)^kX^{\frac{5k^2+k}{2}},\label{mod51}\\
\prod_{k=1}^\infty (1-X^{5k})(1-X^{5k-1})(1-X^{5k-4})=\sum_{k=-\infty}^\infty (-1)^kX^{\frac{5k^2+3k}{2}}.\label{mod52}
\end{align}
This will be used in the proof of \autoref{RRI}.
\end{enumerate}
\end{Ex}

\begin{A}[\textsc{Ramanujan}'s theta function]\index{Ramanujan's theta function}
Let $\alpha,\beta\in\CC((X))$ such that $\alpha\beta\in(X)$. Prove
\[\prod_{k=1}^\infty(1-\alpha^k\beta^k)(1+\alpha^k\beta^{k-1})(1+\alpha^{k-1}\beta^k)=\sum_{k=-\infty}^\infty \alpha^{\frac{k^2+k}{2}}\beta^{\frac{k^2-k}{2}}.\]
\end{A}

\begin{A}
Prove
\begin{flalign*}
(\textnormal{a})&&\sum_{k=-\infty}^\infty X^{k^2}\sum_{k=-\infty}^\infty (-1)^kX^{k^2}&=\Bigl(\sum_{k=0}^\infty (-1)^kX^{2k^2}\Bigr)^2,\\
(\textnormal{b})&&2\sum_{k=-\infty}^\infty X^{k^2}\sum_{k=-\infty}^\infty X^{k^2+k}&=\Bigl(\sum_{k=-\infty}^\infty X^{\frac{k^2+k}{2}}\Bigr)^2&&(\textsc{Cauchy}),\index{Cauchy}\\
(\textnormal{c})&&\Bigl(\sum_{k=-\infty}^\infty X^{k^2}\Bigr)^4&=\Bigl(\sum_{k=-\infty}^\infty(-1)^kX^{k^2}\Bigr)^4+X\Bigl(\sum_{k=-\infty}^\infty X^{k^2+k}\Bigr)^4&&(\textsc{Gauss}).\index{Gauss}
\end{flalign*}
\textit{Hint:} $\alpha^4-\beta^4=(\alpha+\beta)(\alpha-\beta)(\alpha+\ii\beta)(\alpha-\ii\beta)$.
\end{A}

To obtain yet another triple product identity, we first consider a finite version due to Hirschhorn~\cite{Hirschhornpoly}.\index{Hirschhorn}

\begin{Lem}\label{Hirschhorn}
For all $n\in\NN_0$,
\begin{equation}\label{HH}
\prod_{k=1}^n(1-X^k)^2=\sum_{k=0}^n(-1)^k(2k+1)X^{\frac{k^2+k}{2}}\gauss{2n+1}{n-k}.
\end{equation}
\end{Lem}
\begin{proof}
The proof is by induction on $n$: Both sides are $1$ if $n=0$. So assume $n\ge 1$ and let $Q_n$ be the right hand side of \eqref{HH}. 
The summands of $Q_n$ are invariant under the index shift $k\mapsto -k-1$ and vanish for $k>n$. Hence, we may sum over $k\in\ZZ$ and divide by $2$. A threefold application of \eqref{lemgauss} gives:
\begin{align*}
Q_n&=X^n\frac{1}{2}\sum_{k=-\infty}^\infty(-1)^k(2k+1)X^{\frac{k^2-k}{2}}\gauss{2n}{n-k}+\frac{1}{2}\sum(-1)^k(2k+1)X^{\frac{k^2+k}{2}}\gauss{2n}{n-k-1}\\
&=X^n\frac{1}{2}\sum(-1)^k(2k+1)X^{\frac{k^2-k}{2}}\gauss{2n-1}{n-k}+X^{2n}\frac{1}{2}\sum(-1)^k(2k+1)X^{\frac{k^2+k}{2}}\gauss{2n-1}{n-k-1}\\
&\quad+\frac{1}{2}\sum(-1)^k(2k+1)X^{\frac{k^2+k}{2}}\gauss{2n-1}{n-k-1}+X^n\frac{1}{2}\sum(-1)^k(2k+1)X^{\frac{k^2+3k+2}{2}}\gauss{2n-1}{n-k-2}.
\end{align*}
The second and third sum amount to $(1+X^{2n})Q_{n-1}$. We apply the transformations $k\mapsto k+1$ and $k\mapsto k-1$ in the first sum and fourth sum respectively:
\begin{align*}
Q_n&=(1+X^{2n})Q_{n-1}-X^n\frac{1}{2}\sum(-1)^k\biggl((2k+3)X^{\frac{k^2+k}{2}}\gauss{2n-1}{n-k-1}+(2k-1)X^{\frac{k^2+k}{2}}\gauss{2n-1}{n-k-1}\biggr)\\
&=(1+X^{2n})Q_{n-1}-2X^nQ_{n-1}=(1-X^n)^2Q_{n-1}=\prod_{k=1}^n(1-X^k)^2.\qedhere
\end{align*}
\end{proof}

\begin{Thm}[\textsc{Jacobi}]\label{Jhoch3thm}\index{Jacobi}
We have
\begin{equation}\label{Jhoch3}
\boxed{\prod_{k=1}^\infty(1-X^k)^3=\sum_{k=0}^\infty(-1)^k(2k+1)X^{\frac{k^2+k}{2}}.}
\end{equation}
\end{Thm}
\begin{proof}
By \autoref{Hirschhorn}, we have
\begin{align*}
\prod_{k=1}^n(1-X^k)^3&=\sum_{k=0}^n(-1)^k(2k+1)X^{\frac{k^2+k}{2}}\gauss{2n+1}{n-k}\prod_{l=1}^{n}(1-X^l)\\
&=\sum_{k=0}^n(-1)^k(2k+1)X^{\frac{k^2+k}{2}}(1-X^{n-k+1})\ldots(1-X^n)(1-X^{n+k+2})\ldots(1-X^{2n+1}).
\end{align*}
Now the claim follows easily by comparing the coefficient of $X^n$ as in the proof of \autoref{coreuler}.
\end{proof}

In an analytic framework, \eqref{Jhoch3} can be derived from \autoref{jacobi} (see \cite[Theorem~357]{Hardy}). A combinatorial proof was given in \cite{Joichi}.

As a preparation for the infamous Rogers--Ramanujan identities~\cite{RogersR}, we start again with a finite version due to Bressoud~\cite{Bressoud}.\index{Bressoud} The impatient reader may skip these technical results and start right away with the applications in \autoref{seccomb} (\autoref{RRI} is only needed in \autoref{nrpartbi}\eqref{schur1},\eqref{schur2}).

\begin{Lem}\label{Bressoud}
For $n\in\NN_0$,
\begin{align}
\sum_{k=0}^\infty \gauss{n}{k}X^{k^2}&=\sum_{k=-\infty}^\infty (-1)^k\gauss{2n}{n+2k}X^{\frac{5k^2+k}{2}},\label{Bress1}\\
\sum_{k=0}^\infty \gauss{n}{k}X^{k^2+k}&=\sum_{k=-\infty}^\infty (-1)^k\gauss{2n+1}{n+2k}X^{\frac{5k^2-3k}{2}}.\label{Bress2}
\end{align}
\end{Lem}
\begin{proof}
We follow a simplified proof by Chapman~\cite{Chapman}. Let $\alpha_n$ and $\tilde{\alpha}_n$ be the left and the right hand side respectively of \eqref{Bress1}. Similarly, let $\beta_n$ and $\tilde{\beta}_n$ be the left and right hand side respectively of \eqref{Bress2}. Note that all four sums are actually finite.
We show both equations at the same time by establishing a common recurrence relation between $\alpha_n$, $\beta_n$ and $\tilde{\alpha}_n$, $\tilde{\beta}_n$. 

We compute $\alpha_0=\beta_0=\tilde{\alpha}_0=\tilde{\beta}_0=1$. For $n\ge 1$,
\begin{align*}
\alpha_n&\overset{\eqref{lemgauss}}{=}\sum_{k=-\infty}^\infty \biggl(\gauss{n-1}{k}+X^{n-k}\gauss{n-1}{k-1}\biggr)X^{k^2}=\alpha_{n-1}+X^n\sum \gauss{n-1}{k-1}X^{k(k-1)}\\
&=\alpha_{n-1}+X^n\sum \gauss{n-1}{k}X^{k(k+1)}=\alpha_{n-1}+X^n\beta_{n-1},\\
\beta_n-X^n\alpha_n&=\sum_{k=-\infty}^\infty \gauss{n}{k}X^{k^2+k}(1-X^{n-k})=\sum \frac{X^n!}{X^k!X^{n-k}!}X^{k^2+k}(1-X^{n-k})\\
&=(1-X^n)\sum \gauss{n-1}{k}X^{k^2+k}=(1-X^n)\beta_{n-1}.
\end{align*}
These recurrences characterize $\alpha_n$ and $\beta_n$ uniquely. The familiar index transformation $k\mapsto -k-1$ implies $\sum(-1)^k\gauss{2n-2}{n+2k}X^{\frac{5(k^2+k)}{2}}=0$. This is used in the following computation:
\begin{align*}
\tilde{\alpha}_n-\tilde{\alpha}_{n-1}&=\sum_{k=-\infty}^\infty (-1)^k\biggl(\gauss{2n}{n+2k}-\gauss{2n-2}{n-1+2k}\biggr)X^{\frac{5k^2+k}{2}}\\
&\overset{\eqref{lemgauss}}{=}\sum (-1)^k\biggl(\gauss{2n-1}{n+2k}+X^{n-2k}\gauss{2n-1}{n+2k-1}-\gauss{2n-2}{n-1+2k}\biggr)X^{\frac{5k^2+k}{2}}\\
&\overset{\eqref{lemgauss}}{=}\sum (-1)^k\biggl(X^{n+2k}\gauss{2n-2}{n+2k}+X^{n-2k}\gauss{2n-1}{n+2k-1}\biggr)X^{\frac{5k^2+k}{2}}=X^n\tilde{\beta}_{n-1},\\
\tilde{\beta}_n-X^n\tilde{\alpha}_n&=\sum_{k=-\infty}^\infty (-1)^k\biggl(\gauss{2n+1}{n+2k}-X^{n+2k}\gauss{2n}{n+2k}\biggr)X^{\frac{5k^2-3k}{2}}\\
&=\sum (-1)^k\gauss{2n}{n+2k-1}X^{\frac{5k^2-3k}{2}}\\
&=\sum (-1)^k\biggl(\gauss{2n-1}{n+2k-1}+X^{n-2k+1}\gauss{2n-1}{n+2k-2}\biggr)X^{\frac{5k^2-3k}{2}}\\
&=\tilde{\beta}_{n-1}+X^n\sum (-1)^k\gauss{2n-1}{n+2k-2}X^{\frac{5k^2-7k+2}{2}}\\
&=\tilde{\beta}_{n-1}+X^n\sum (-1)^{1-k}\gauss{2n-1}{n-2k}X^{\frac{5(1-k)^2-7(1-k)+2}{2}}\\
&=\tilde{\beta}_{n-1}-X^n\sum (-1)^k\gauss{2n-1}{n+2k-1}X^{\frac{5k^2-3k}{2}}=(1-X^n)\tilde{\beta}_{n-1}.
\end{align*}
By induction on $n$, it follows that $\alpha_n=\tilde{\alpha}_n$ and $\beta_n=\tilde{\beta}_n$ as desired.
\end{proof}

\begin{Thm}[\textsc{Rogers--Ramanujan} identities]\label{RRI}\index{Rogers--Ramanujan identities}
We have
\begin{empheq}[box=\fbox]{align}
\prod_{k=1}^\infty\frac{1}{(1-X^{5k-1})(1-X^{5k-4})}&=\sum_{k=0}^\infty\frac{X^{k^2}}{X^k!},\label{RR1}\\
\prod_{k=1}^\infty\frac{1}{(1-X^{5k-2})(1-X^{5k-3})}&=\sum_{k=0}^\infty\frac{X^{k^2+k}}{X^k!}.\label{RR2}
\end{empheq}
\end{Thm}
\begin{proof}
As in the proof of \autoref{coreuler} we can show that
\begin{align*}
\sum_{k=0}^\infty\frac{X^{k^2}}{X^k!}&=\lim_{n\to\infty}\sum_{k=0}^\infty\frac{X^{k^2}(1-X^n)\ldots(1-X^{n-k+1})}{X^k!}=\lim_{n\to\infty}\sum_{k=0}^\infty\gauss{n}{k}X^{k^2}\\
&\overset{\eqref{Bress1}}{=}\lim_{n\to\infty}\sum_{k=-\infty}^\infty (-1)^k\gauss{2n}{n+2k}X^{\frac{5k^2+k}{2}}.
\end{align*}
Since
\begin{align*}
X^{\frac{5k^2+k}{2}}\Bigl(\gauss{2n}{n+2k}-\prod_{l=1}^\infty\frac{1}{1-X^l}\Bigr)&=X^{\frac{5k^2+k}{2}}\frac{(1-X^{n-2k+1})\ldots(1-X^{2n})(1-X^{n+2k+1})\ldots -1}{(1-X)(1-X^2)\ldots}\\
&\in(X^{\frac{5k^2+k}{2}+n-2|k|+1})\subseteq(X^{n+1}),
\end{align*}
we obtain similarly
\begin{align*}
\lim_{n\to\infty}&\sum_{k=-\infty}^\infty (-1)^k\gauss{2n}{n+2k}X^{\frac{5k^2+k}{2}}=\sum_{k=-\infty}^\infty (-1)^kX^{\frac{5k^2+k}{2}}\prod_{l=1}^\infty\frac{1}{1-X^l}\\
&\overset{\eqref{mod51}}{=}\frac{\prod_{k=1}^\infty(1-X^{5k})(1-X^{5k-2})(1-X^{5k-3})}{\prod_{k=1}^\infty(1-X^k)}=\prod_{k=1}^\infty\frac{1}{(1-X^{5k-1})(1-X^{5k-4})}.
\end{align*}
The second identity follows in the same way by using \eqref{mod52} instead of \eqref{mod51}.
\end{proof}

The Rogers--Ramanujan identities were long believed to lie deeper within the theory of elliptic functions (Hardy~\cite[p. 385]{Hardy}\index{Hardy} wrote “No proof is really easy (and it would perhaps be unreasonable to expect an easy proof).”; Andrews~\cite[p. 105]{Andrews} wrote “…no doubt it would be unreasonable to expect a really easy proof.”). Meanwhile, a great number of proofs were found, some of which are combinatorial (see \cite{AndrewsRR} or the recent book \cite{Sills}). An interpretation of these identities is given in \autoref{nrpartbi} below.
We point out that there are many “finite identities”, like \autoref{Bressoud}, approaching the Rogers--Ramanujan identities (as there are many rational sequences approaching $\sqrt{2}$). 

One can find many more interesting identities, like the \emph{quintuple product},\index{quintuple product} along with comprehensive references (and analytic proofs) in Johnson~\cite{Johnson}. \index{Johnson}

\section{Applications to combinatorics}\label{seccomb}

In this section we bring the abstract theorems and identities of the previous section to life. 
If $a_0,a_1,\ldots$ is a sequence of numbers usually arising from combinatorial context, the power series $\alpha=\sum a_nX^n$ is called the \emph{generating function}\index{generating function} of $(a_n)_n$. This is merely a change of view, but we will see that clever power series manipulations often reveal explicit formulas for $a_n$, which can hardly be seen by inductive arguments. 
As a matter of fact, some generating functions turn out to be \emph{rational}\index{rational function} functions (i.\,e. elements of $\CC(X)$). 
We give a first impression with the most familiar generating functions.

\begin{Ex}\label{exgen}\hfill
\begin{enumerate}[(i)]
\item The number of $k$-element subsets of an $n$-element set is $\binom{n}{k}$ with generating function $(1+X)^n$. 
A $k$-element multi-subset $\{a_1,\ldots,a_k\}$ of $\{1,\ldots,n\}$ with $a_1\le\ldots\le a_k$ (where elements are allowed to appear more than once) can be turned into a $k$-element subset $\{a_1,a_2+1,\ldots,a_k+k-1\}$ of $\{1,\ldots,n+k-1\}$ and vice versa. The number of $k$-element multi-subsets of an $n$-element set is therefore $\binom{n+k-1}{k}$ with generating function $(1-X)^{-n}$ by Newton's binomial theorem. 

\item The number of $k$-dimensional subspaces of an $n$-dimensional vector space over a finite field with $q<\infty$ elements is $\gauss{n}{k}$ evaluated at $X=q$ (indeed there are $(q^n-1)(q^n-q)\ldots(q^n-q^{n-k+1})$ linearly independent $k$-tuples and $(q^k-1)(q^k-q)\ldots(q^k-q^{k-1})$ of them span the same subspace). The generating function is closely related to Gauss' binomial theorem.

\item The \emph{Fibonacci numbers}\index{Fibonacci numbers} $f_n$ are defined by $f_n:=n$\index{*Fn@$f_n$} for $n=0,1$ and $f_{n+1}:=f_n+f_{n-1}$ for $n\ge 1$. The generating function $\alpha$ satisfies $\alpha=X+X^2\alpha+X\alpha$ and is therefore given by $\alpha=\frac{X}{1-X-X^2}$. An application of the partial fraction decomposition \eqref{partial} leads to the well-known \emph{Binet formula}\index{Binet formula}
\[f_n=\frac{1}{\sqrt{5}}\Bigl(\frac{1+\sqrt{5}}{2}\Bigr)^n-\frac{1}{\sqrt{5}}\Bigl(\frac{1-\sqrt{5}}{2}\Bigr)^n.\]

\item The \emph{Catalan numbers}\index{Catalan numbers} $c_n$ are defined by $c_n:=n$\index{*Cn@$c_n$} for $n=0,1$ and 
\[c_n:=\sum_{k=1}^{n-1}c_kc_{n-k}\] 
for $n\ge 2$ (most authors shift the index by $1$). Its generating function $\alpha$ fulfills $\alpha-\alpha^2=X$, i.\,e. it is the reverse of $X-X^2$. This quadratic equation has only one solution $\alpha=\frac{1}{2}(1-\sqrt{1-4X})$ in $\CC[[X]]^\circ$. Now $c_n$ can be computed by Newton's theorem. Slightly more elegant is an application of Lagrange--Bürmann's inversion formula.
Since 
\[\Bigl(\frac{X}{X-X^2}\Bigr)^{n+1}=(1-X)^{-n-1}\overset{\eqref{newtoneq}}{=}\sum_{k=0}^\infty\binom{-n-1}{k}(-1)^kX^k,\]
we compute
\[c_{n+1}=\frac{\res((X-X^2)^{-n-1})}{n+1}=\frac{1}{n+1}(-1)^n\binom{-n-1}{n}=\frac{1}{n+1}\frac{(n+1)\ldots2n}{n!}=\frac{1}{n+1}\binom{2n}{n}.\]
\end{enumerate}
\end{Ex}

We now focus on combinatorial objects which defy explicit formulas. 

\begin{Thm}[\textsc{Lambert}]\index{Lambert}
Let $d_n$\index{*Dn@$d_n$} be the number of \textup(positive\textup) divisors of $n\in\NN$. Then
\[\sum_{n=1}^\infty d_nX^n=\sum_{k=1}^\infty\frac{X^k}{1-X^k}.\]
\end{Thm}
\begin{proof}
We have
\[\sum_{k=1}^\infty\frac{X^k}{1-X^k}=\sum_{k=1}^\infty X^k\sum_{l=0}^\infty X^{kl}=\sum_{k,l=1}^\infty X^{kl}=\sum_{n=1}^\infty d_nX^n.\qedhere\]
\end{proof}

\begin{A}[\textsc{Clausen}]\index{Clausen}
Prove
\[\sum_{n=1}^\infty d_nX^n=\sum_{k=1}^\infty\frac{1+X^k}{1-X^k}X^{k^2}.\]
\textit{Hint:} If $d$ is a divisor of $n$, so is $\frac{n}{d}$.
\end{A}

\begin{Thm}
Let $s(n,q)$ \textup(resp. $k(n,q)$\textup) be the number of similarity classes of \textup(resp. invertible\textup) $n\times n$-matrices over the finite field $\FF_q$. Then
\begin{align*}
1+\sum_{n=1}^\infty s(n,q)X^n&=\prod_{k=1}^\infty\frac{1}{1-qX^k}=\sum_{k=0}^{\infty}\frac{(qX)^k}{X^k!},\\ 
1+\sum_{n=1}^\infty k(n,q)X^n&=\prod_{k=1}^\infty\frac{1-X^k}{1-qX^k}.
\end{align*}
In particular, $s(n,q)$ and $k(n,q)$ are polynomials in $q$.
\end{Thm}
\begin{proof}
Recall that the similarity classes of matrices are represented by rational canonical forms. A matrix $A$ in rational canonical form is described by a series of non-constant monic polynomials $\alpha_1\mid\ldots\mid\alpha_m$ in $\FF_q[X]$ such that $\alpha_1\ldots\alpha_m$ is the characteristic polynomial of $A$.
The same information is encoded in the sequence $\beta_1:=\alpha_1$, $\beta_2:=\frac{\alpha_2}{\alpha_1},\ldots,\beta_m:=\frac{\alpha_m}{\alpha_{m-1}}$ such that 
\[\sum_{k=1}^m(m-k+1)\deg\beta_k=n.\] 
Now the number of monic polynomials of degree $d\ge 1$ is $q^d$. Hence, $s(n,q)$ 
is the coefficient of $X^n$ in 
\begin{equation}\label{expand}
\prod_{k=1}^\infty\Bigl(1+\sum_{i=1}^\infty q^iX^{ik}\Bigr)=\prod_{k=1}^\infty\sum_{i=0}^\infty (qX^k)^i=\prod_{k=1}^\infty\frac{1}{1-qX^k}\overset{\eqref{E2}}{=}\sum_{k=0}^{\infty}\frac{(qX)^k}{X^k!}.
\end{equation}

Now recall that $A$ is invertible if and only if $0$ is not an eigenvalue of $A$. Equivalently, the characteristic polynomial $\alpha_1\ldots\alpha_m$ (and in turn $\alpha_1,\ldots,\alpha_m$) has a non-vanishing constant term. The number of monic polynomials of degree $d\ge 1$ with non-vanishing constant-term is $q^d-q^{d-1}$. Hence, $k(n,q)$ 
is the coefficient of $X^n$ in 
\begin{equation}\label{expandGL}
\prod_{k=1}^\infty\Bigl(1+\sum_{i=1}^\infty (q^i-q^{i-1})X^{ik}\Bigr)=\prod_{k=1}^\infty\Bigl(\sum_{i=0}^\infty (qX^k)^i-X^k\sum_{i=0}^\infty (qX^k)^i\Bigr)=\prod_{k=1}^\infty\frac{1-X^k}{1-qX^k}.
\end{equation}
The last assertion follows by expanding the left hand side of \eqref{expand} and \eqref{expandGL} respectively. 
\end{proof}

\begin{Ex}
From
\begin{align*}
\prod_{k=1}^\infty\frac{1}{1-qX^k}&=(1+qX+q^2X^2+q^3X^3+\ldots)(1+qX^2+q^2X^4+\ldots)(1+qX^3+\ldots)\ldots\\
&=1+qX+(q^2+q)X^2+(q^3+q^2+q)X^3+(q^4+q^3+2q^2+q)X^4+\ldots,
\end{align*}
we obtain $s(4,q)=q^4+q^3+2q^2+q$. 
Similarly, $k(1,q)=q-1$, $k(2,q)=q^2-1$ and $k(3,q)=q^3-q$.
\end{Ex}

\begin{Def}\label{defpart}
A \emph{partition}\index{partition} of $n\in\NN$ is a sequence of positive integers $\lambda=(\lambda_1,\ldots,\lambda_l)$ such that 
\[\lambda_1+\ldots+\lambda_l=n\qquad \text{and}\qquad\lambda_1\ge\ldots\ge\lambda_l.\] 
We call $\lambda_1,\ldots,\lambda_l$ the \emph{parts} of $\lambda$. We will often collect identical parts with exponent notation like $(2,2,2,1,1)=(2^3,1^2)$. 
The set of partitions of $n$ is denoted by $P(n)$ and its cardinality is $p(n):=|P(n)|$.\index{*Pn@$p(n)$}
For $k\in\NN_0$ let $p_k(n)$\index{*Pkn@$p_k(n)$} be the number of partitions of $n$ with each part $\lambda_i\le k$. Finally, let $p_{k,l}(n)$ be the number of partitions of $n$ with each part $\le k$ and at most $l$ parts in total. Clearly, $p_1(n)=p_{n,1}(n)=1$ and $p_n(n)=p_{n,n}(n)=p(n)$. Moreover, $p_{k,l}(n)=0$ whenever $n>kl$. For convenience let $p(0)=p_0(0)=p_{0,0}(0)=1$ ($0$ can be interpreted as the empty sum). 
\end{Def}

\begin{Ex}
The partitions of $n=7$ are
\begin{gather*}
(7),(6,1),(5,2),(5,1^2),(4,3),(4,2,1),(4,1^3),(3^2,1),\\
(3,2^2),(3,2,1^2),(3,1^4),(2^3,1),(2^2,1^3),(2,1^5),(1^7).
\end{gather*}
Hence, $p(7)=15$, $p_3(7)=8$ and $p_{3,3}(7)=2$.
\end{Ex}

\begin{Thm}\label{partseries}
The generating functions of $p(n)$, $p_k(n)$ and $p_{k,l}(n)$ are given by
\begin{align*}
\Aboxed{\sum_{n=0}^\infty p(n)X^n&=\prod_{k=1}^\infty\frac{1}{1-X^k},}\\
\sum_{n=0}^\infty p_k(n)X^n&=\frac{1}{X^k!},\\
\sum_{n=0}^\infty p_{k,l}(n)X^n&=\gauss{k+l}{k}.\tag{\textnormal{\textsc{Cayley}}}\index{Cayley}
\end{align*}
\end{Thm}
\begin{proof}
It is easy to see that $p_k(n)$ is the coefficient of $X^n$ in 
\begin{equation}\label{eulereq}
\begin{gathered}
(1+X^1+X^{1+1}+\ldots)(1+X^2+X^{2+2}+\ldots)\ldots(1+X^k+X^{k+k}+\ldots)\\
=\frac{1}{1-X}\frac{1}{1-X^2}\ldots\frac{1}{1-X^k}=\frac{1}{X^k!}.
\end{gathered}
\end{equation}
This shows the second equation. The first follows from $p(n)=\lim_{k\to \infty}p_k(n)$. For the last claim we argue by induction on $k+l$ using \eqref{lemgauss}. If $k=0$ or $l=0$, then both sides equal $1$. Thus, let $k,l\ge 1$. Pick a partition $\lambda=(\lambda_1,\lambda_2,\ldots)$ of $n$ with each part $\le k$ and at most $l$ parts. If $\lambda_1<k$, then all parts are $\le k-1$ and $\lambda$ is counted by $p_{k-1,l}(n)$. If on the other hand $\lambda_1=k$, then $(\lambda_2,\lambda_3,\ldots)$ is counted by $p_{k,l-1}(n-k)$. Conversely, each partition counted by $p_{k,l-1}(n-k)$ can be extended to a partition counted by $p_{k,l}(n)$. We have proven the recurrence 
\[p_{k,l}(n)=p_{k-1,l}(n)+p_{k,l-1}(n-k).\]
Induction yields
\begin{align*}
\sum p_{k,l}(n)X^n&=\sum p_{k-1,l}(n)X^n+X^k\sum p_{k,l-1}(n)X^n\\
&=\gauss{k+l-1}{k-1}+X^k\gauss{k+l-1}{k}\overset{\eqref{lemgauss}}{=}\gauss{k+l}{k}.\qedhere
\end{align*}
\end{proof}

\begin{Thm}\label{nrpartbi}
The following assertions hold for $n,k,l\in\NN_0$:
\begin{enumerate}[(i)]
\item\label{parta} $p_{k,l}(n)=p_{l,k}(n)=p_{k,l}(kl-n)$ for $n\le kl$.

\item The number of partitions of $n$ into exactly $k$ parts is the number of partitions with largest part $k$.

\item\label{partb} \textup(\textsc{Glaisher}\textup)\index{Glaisher} The number of partitions of $n$ into parts not divisible by $k$ equals the number of partitions with no part repeated $k$ times \textup(or more\textup).

\item\label{eulerodd} \textup(\textsc{Euler}\textup)\index{Euler} The number of partitions of $n$ into unequal parts is the number of partitions into odd parts.

\item\label{schur1} \textup(\textsc{Schur}\textup)\index{Schur} The number of partitions of $n$ in parts which differ by more than $1$ equals the number of partitions in parts of the form $\pm 1+5k$.

\item\label{schur2} \textup(\textsc{Schur}\textup)\index{Schur} The number of partitions of $n$ in parts which differ by more than $1$ and are larger than $1$ equals the number of partitions into parts of the form $\pm2+5k$. 
\end{enumerate}
\end{Thm}
\begin{proof}\hfill
\begin{enumerate}[(i)]
\item Since $\gauss{k+l}{k}=\gauss{k+l}{l}$, we obtain $p_{k,l}(n)=p_{l,k}(n)$ by \autoref{partseries}. Let $\lambda=(\lambda_1,\ldots,\lambda_s)$ be a partition counted by $p_{k,l}(n)$. After adding zero parts if necessary, we may assume that $s=l$. Then $\bar{\lambda}:=(k-\lambda_l,k-\lambda_{l-1},\ldots,k-\lambda_1)$ is a partition counted by $p_{k,l}(kl-n)$. Since $\bar{\bar{\lambda}}=\lambda$, we obtain a bijection between the partitions counted by $p_{k,l}(n)$ and $p_{k,l}(kl-n)$.

\item\label{pkn} The number of partitions of $n$ with largest part $k$ is $p_k(n)-p_{k-1}(n)$. The number of partitions with exactly $k$ parts is
\[p_{n,k}(n)-p_{n,k-1}(n)\overset{\eqref{parta}}{=}p_{k,n}(n)-p_{k-1,n}(n)=p_k(n)-p_{k-1}(n).\]

\item Looking at \eqref{eulereq} again, it turns out that the desired generating function is
\[
\prod_{k\,\nmid\, m}\frac{1}{1-X^m}=\prod_{m=1}^\infty\frac{1-X^{km}}{1-X^m}=(1+X+\ldots+X^{k-1})(1+X^2+\ldots+X^{2(k-1)})\ldots.
\]

\item Take $k=2$ in \eqref{partb}.

\item According to \cite[Section~2.4]{Sills}, it was Schur, who first gave this interpretation of the Rogers--Ramanujan identities. The coefficient of $X^n$ on the left hand side of \eqref{RR1} is the number of partitions into parts of the form $\pm1+5k$. The right hand side can be rewritten (thanks to \autoref{partseries}) as
\[\sum_{k=0}^\infty\sum_{n=0}^\infty p_k(n)X^{n+k^2}=\sum_{n=0}^\infty\sum_{k=0}^np_k(n-k^2)X^n,\]
where as usual we interpret $p_k(n-k^2)=0$ if $n<k^2$.
By \eqref{pkn}, $p_k(n-k^2)$ counts the partitions of $n-k^2$ with at most $k$ parts.
If $(\lambda_1,\ldots,\lambda_k)$ is such a partition (allowing $\lambda_i=0$ here), then $(\lambda_1+2k-1,\lambda_2+2k-3,\ldots,\lambda_k+1)$ is a partition of $n-k^2+1+3+\ldots+2k-1=n$ with exactly $k$ parts, which all differ by more than $1$.

\item This follows similarly using $k^2+k=2+4+\ldots+2k$. \qedhere
\end{enumerate}
\end{proof}

There is a remarkable connection between \eqref{partb}, \eqref{eulerodd} and \eqref{schur1} of \autoref{nrpartbi}: Numbers not divisible by $3$ are of the form $\pm1+3k$, while odd numbers are of the form $\pm1+4k$. 

\begin{Ex}
For $n=7$ the following partitions are counted by \autoref{nrpartbi}:
\begin{center}
\begin{tabular}{llllll}
exactly three parts:& $(5,1^2)$,& $(4,2,1)$,& $(3^2,1)$,& $(3,2^2)$\\
largest part $3$:&$(3^2,1)$,& $(3,2^2)$,& $(3,2,1^2)$,& $(3,1^4)$\\\hline
unequal parts:&$(7)$,& $(6,1)$,& $(5,2)$,& $(4,3)$,& $(4,2,1)$\\
odd parts:& $(7)$,& $(5,1^2)$,& $(3^2,1)$,& $(3,1^4)$,& $(1^7)$\\\hline
parts differ by more than $1$:& $(7)$,& $(6,1)$,& $(5,2)$\\
parts of the form $\pm1+5k$& $(6,1)$,& $(4,1^3)$,& $(1^7)$\\\hline
parts $\ge 2$ differ by more than $1$:& $(7)$,& $(5,2)$\\
parts of the form $\pm2+5k$& $(7)$,& $(3,2^2)$
\end{tabular}
\end{center}
\end{Ex}

Some of the statements in \autoref{nrpartbi} permit nice combinatorial proofs utilizing \emph{Young diagrams}\index{Young diagram} (or \emph{Ferrers diagrams}).\index{Ferrers diagram} We refer the reader to the introductory book by Andrews--Eriksson~\cite{AndrewsEriksson}. \index{Andrews--Eriksson}
The following exercise (inspired by \cite{AndrewsEriksson}) can be solved with formal power series.

\begin{A}
Prove the following statements for $n,k\in\NN$:
\begin{enumerate}[(a)]
\item The number of partitions of $n$ into even parts is the number of partitions whose parts have even multiplicity.

\item (\textsc{Legendre})\index{Legendre} If $n$ is not of the form $\frac{1}{2}(3k^2+k)$ with $k\in\ZZ$, then the number of partitions of $n$ into an even number of unequal parts is the number of partitions into an odd number of unequal parts.\\
\textit{Hint:} Where have we encountered $\frac{1}{2}(3k^2+k)$ before?

\item (\textsc{Fine})\index{Fine} If $n$ is not of the form $\frac{1}{2}(3k^2+k)$ with $k\in\ZZ$, then the number of partitions of $n$ into unequal parts with largest part even is the number of partitions into unequal parts with largest part odd.

\item (\textsc{Subbarao})\index{Subbuarao} The number of partitions of $n$ where each part appears 2, 3 or 5 times equals the number of partitions into parts of the form $\pm 2+12k$, $\pm3+12k$ or $6+12k$.

\item (\textsc{MacMahon})\index{MacMahon} The number of partitions of $n$ where each part appears at least twice equals the number of partitions in parts not of the form $\pm1+6k$. 
\end{enumerate}
\end{A}

The reader may have noticed that Euler's pentagonal number theorem~\eqref{EPTN} is just the inverse of the generating function of $p(n)$ from \autoref{partseries}, i.\,e. 
\[\sum_{n=0}^\infty p(n)X^n\cdot \sum_{k=-\infty}^\infty(-1)^kX^{\frac{3k^2+k}{2}}=1\]
and therefore
\[\sum_{k=-n}^n(-1)^kp\Bigl(n-\frac{3k^2+k}{2}\Bigr)=0\]
for $n\in\NN$, where $p(k):=0$ whenever $k<0$. This leads to a recurrence formula
\begin{align*}
p(0)&=1,\\
p(n)&=p(n-1)+p(n-2)-p(n-5)-p(n-7)+\ldots\qquad(n\in\NN).
\end{align*}

\begin{Ex}
We compute
\begin{align*}
p(1)&=p(0)=1,&p(4)&=p(3)+p(2)=3+2=5,\\
p(2)&=p(1)+p(0)=2,&p(5)&=p(4)+p(3)-p(0)=5+3-1=7,\\
p(3)&=p(2)+p(1)=3,&p(6)&=p(5)+p(4)-p(1)=7+5-1=11
\end{align*}
(see \url{https://oeis.org/A000041} for more terms).
\end{Ex}

The generating functions we have seen so far all have integer coefficients. If $\alpha,\beta\in\ZZ[[X]]$ and $d\in\NN$, we write $\alpha\equiv\beta\pmod{d}$, if all coefficients of $\alpha-\beta$ are divisible by $d$. This is compatible with the ring structure of $\ZZ[[X]]$, namely if $\alpha\equiv\beta\pmod{d}$ and $\gamma\equiv\delta\pmod{d}$, then $\alpha+\gamma\equiv\beta+\delta\pmod{d}$ and $\alpha\gamma\equiv\beta\delta\pmod{d}$. Now suppose $\alpha\in 1+(X)$. Then the proof of \autoref{leminv} shows $\alpha^{-1}\in\ZZ[[X]]$. In this case $\alpha\equiv\beta\pmod{d}$ is equivalent to $\alpha^{-1}\equiv\beta^{-1}\pmod{d}$. If $d=p$ happens to be a prime, we have
\[(\alpha+\beta)^p=\sum_{k=0}^p\frac{p(p-1)\ldots(p-k+1)}{k!}\alpha^k\beta^{p-k}\equiv\alpha^p+\beta^p\pmod{p},\]
as in any commutative ring.

With this preparation, we come to a remarkable discovery by Ramanujan~\cite{RamanujanCongruence}.

\begin{Thm}[\textsc{Ramanujan}]\label{R57}\index{Ramanujan}
The following congruences hold for all $n\in\NN_0$:
\[\boxed{p(5n+4)\equiv 0\pmod{5},\qquad p(7n+5)\equiv 0\pmod{7}.}\]
\end{Thm}
\begin{proof}
Let $\alpha:=\prod(1-X^k)$. By the remarks above, $\alpha^5=\prod(1-X^k)^5\equiv \prod(1-X^{5k})\equiv\alpha(X^5)\pmod{5}$ and $\alpha^{-5}\equiv\alpha(X^5)^{-1}\pmod{5}$. For $k\in\ZZ$ we compute modulo $5$:
\[\frac{k^2+k}{2}\equiv\begin{cases}
0&\text{if }k\equiv 0,-1\pmod{5},\\
1&\text{if }k\equiv 1,-2\pmod{5},\\
3&\text{if }k\equiv 2\pmod{5}.
\end{cases}\]
This allows to write Jacobi's identity \eqref{Jhoch3} in the form
\begin{align*}
\alpha^3&=\sum_{\mathclap{k\,\equiv\, 0,-1\text{ (mod $5$)}}}(-1)^k(2k+1)X^{\frac{k^2+k}{2}}+\sum_{\mathclap{k\,\equiv\, 1,-2\text{ (mod $5$)}}}(-1)^k(2k+1)X^{\frac{k^2+k}{2}}+\sum_{\mathclap{k\,\equiv\, 2\text{ (mod $5$)}}}(-1)^k(\underbrace{2k+1}_{\mathclap{\equiv\,0\text{ (mod $5$)}}})X^{\frac{k^2+k}{2}}\\
&\equiv\alpha_0+\alpha_1\pmod{5},
\end{align*}
where $\alpha_i$ is formed by the monomials $a_kX^k$ with $k\equiv i\pmod{5}$.
Now \autoref{partseries} implies
\begin{equation}\label{rameq}
\sum_{n=0}^\infty p(n)X^n=\alpha^{-1}=\frac{(\alpha^3)^3}{(\alpha^5)^2}\equiv\frac{(\alpha_0+\alpha_1)^3}{\alpha(X^5)^2}\pmod{5}.
\end{equation}
If we expand $(\alpha_0+\alpha_1)^3$, then only terms $X^k$ with $k\equiv 0,1,2,3\pmod{5}$ occur, while in $\alpha(X^5)^{-2}$ only terms $X^{5k}$ occur. Therefore, the right hand side of \eqref{rameq} contains no terms of the form $X^{5k+4}$. So we must have $p(5k+4)\equiv 0\pmod{5}$.

For the congruence modulo $7$ we compute similarly $\frac{1}{2}(k^2+k)\equiv 0,1,3,6\pmod{7}$, where the last case only occurs if $k\equiv 3\pmod{7}$ and in this case $2k+1\equiv 0\pmod{7}$. As before we may write $\alpha^3\equiv\alpha_0+\alpha_1+\alpha_3\pmod{7}$. Then
\[\sum_{n=0}^\infty p(n)X^n=\alpha^{-1}=\frac{(\alpha^3)^2}{\alpha^7}\equiv\frac{(\alpha_0+\alpha_1+\alpha_3)^2}{\alpha(X^7)}\pmod{7}.
\]
Again $X^{7k+5}$ does not appear on the right hand side.
\end{proof}

Ramanujan has also discovered the congruence $p(11n+6)\equiv 0\pmod{11}$ for all $n\in\NN_0$ (the reader finds the history of this and other results in \cite{AndrewsEriksson,Hirschhorn}, for instance). This was believed to be more difficult to prove, until elementary proofs were found by Marivani~\cite{Marivani},\index{Marivani} Hirschhorn~\cite{HirschhornR11}\index{Hirschhorn} and others (see also \cite[Section~3.5]{Hirschhorn}). The details are however extremely tedious to verify by hand. 

By the Chinese remainder theorem, two congruences of coprime moduli can be combined as in
\[p(35n+19)\equiv 0\pmod{35}.\]
Ahlgren~\cite{Ahlgren}\index{Ahlgren} (building on Ono~\cite{Ono})\index{Ono} has shown that in fact for every integer $k$ coprime to $6$ there is such a congruence modulo $k$. Unfortunately, they do not look as nice as \autoref{R57}. For instance,
\[p(11^3\cdot13n+237)\equiv 0\pmod{13}.\]

The next result explains the congruence modulo $5$ and is known as Ramanujan's “most beautiful” formula (since \autoref{RRI} was first discovered by Rogers).

\begin{Thm}[\textsc{Ramanujan}]\index{Ramanujan!most beautiful formula}
We have
\[\boxed{\sum_{n=0}^\infty p(5n+4)X^n=5\prod_{k=1}^\infty\frac{(1-X^{5k})^5}{(1-X^k)^6}.}\]
\end{Thm}
\begin{proof}
The arguments are taken from \cite[Chapter~5]{Hirschhorn}, leaving out some unessential details. 
This time we start with Euler's pentagonal number theorem. Since
\[\frac{3k^2+k}{2}\equiv \begin{cases}
0&\text{if }k\equiv 0,-2\pmod{5},\\
1&\text{if }k\equiv -1\pmod{5},\\
2&\text{if }k\equiv 1,2\pmod{5},
\end{cases}\]
we can write \eqref{EPTN} in the form
\[\alpha:=\prod_{k=1}^\infty(1-X^k)=\sum_{k=-\infty}^\infty (-1)^kX^{\frac{3k^2+k}{2}}=\alpha_0+\alpha_1+\alpha_2,\]
where $\alpha_i$ is formed by the terms $a_kX^k$ with $k\equiv i\pmod{5}$. In fact,
\begin{equation}\label{alpha1}
\alpha_1=\sum_{k=-\infty}^\infty (-1)^{5k-1}X^{\frac{3(5k-1)^2+5k-1}{2}}=-X\sum(-1)^kX^{\frac{75k^2-25k}{2}}=-X\alpha(X^{25}).
\end{equation}
On the other hand we have
\[\sum_{k=0}^\infty(-1)^k(2k+1)X^{\frac{k^2+k}{2}}\overset{\eqref{Jhoch3}}{=}\alpha^3=(\alpha_0+\alpha_1+\alpha_2)^3.\]
When we expand the right hand side, the monomials of the form $X^{5k+2}$ all occur in $3\alpha_0(\alpha_0\alpha_2+\alpha_1^2)$. 
Since we have already realized in the proof of \autoref{R57} that $(k^2+k)/2\not\equiv 2\pmod{5}$, we conclude that \begin{equation}\label{short}
\alpha_1^2=-\alpha_0\alpha_2.
\end{equation}

Let $\zeta\in\CC$ be a primitive $5$-th root of unity. 
Using that 
\[X^5-1=\prod_{i=0}^4(X-\zeta^i)=\zeta^{1+2+3+4}\prod(\zeta^{-i}X-1)=\prod(\zeta^iX-1),\] 
we compute
\[\prod_{i=0}^4\alpha(\zeta^iX)=\prod_{k=1}^\infty\prod_{i=0}^4(1-\zeta^{ik}X^{k})=\alpha(X^5)^5\prod_{5\,\nmid\,k}(1-X^{5k})=\frac{\alpha(X^5)^6}{\alpha(X^{25})}.\]
This leads to 
\begin{align}
\sum& p(n)X^n=\frac{1}{\alpha}=\frac{\alpha(X^{25})}{\alpha(X^5)^6}\alpha(\zeta X)\alpha(\zeta^2X)\alpha(\zeta^3X)\alpha(\zeta^4X)\notag\\
&=\frac{\alpha(X^{25})}{\alpha(X^5)^6}(\alpha_0+\zeta\alpha_1+\zeta^2\alpha_2)(\alpha_0+\zeta^2\alpha_1+\zeta^4\alpha_2)(\alpha_0+\zeta^3\alpha_1+\zeta\alpha_2)(\alpha_0+\zeta^4\alpha_1+\zeta^3\alpha_2).\label{long}
\end{align}
We are only interested in the monomials $X^{5n+4}$. Those arise from the products $\alpha_0^2\alpha_2^2$, $\alpha_0\alpha_1^2\alpha_2$ and $\alpha_1^4$. To facilitate the expansion of the right hand side of \eqref{long}, we notice that the Galois automorphism $\gamma$ of the cyclotomic field $\QQ_5$ sending $\zeta$ to $\zeta^2$ permutes the four factors cyclically. Whenever we obtain a product involving some $\zeta^i$, say $\alpha_0^2\alpha_2^2\zeta^3$, the full orbit under $\langle\gamma\rangle$ must occur, which is $\alpha_0^2\alpha_2^2(\zeta+\zeta^2+\zeta^3+\zeta^4)=-\alpha_0^2\alpha_2^2$. Now there are six choices to form $\alpha_0^2\alpha_2^2$. Four of them form a Galois orbit, while the two remaining appear without $\zeta$. The whole contribution is therefore $(1+1-1)\alpha_0^2\alpha_2^2=\alpha_0^2\alpha_2^2$. In a similar manner we compute,
\[
\sum p(5n+4)X^{5n+4}=\frac{\alpha(X^{25})}{\alpha(X^5)^6}(\alpha_0^2\alpha_2^2-3\alpha_0\alpha_1^2\alpha_2+\alpha_1^4)\overset{\eqref{short}}{=}5\frac{\alpha(X^{25})}{\alpha(X^5)^6}\alpha_1^4\overset{\eqref{alpha1}}{=}5X^4\frac{\alpha(X^{25})^5}{\alpha(X^5)^6}.
\]
The claim follows after dividing by $X^4$ and replacing $X^5$ by $X$.
\end{proof}

Partitions can be generalized to higher dimensions. A \emph{plane partition}\index{plane partition} of $n\in\NN$ is an $n\times n$-matrix $\lambda=(\lambda_{ij})$ consisting of non-negative integers such that 
\begin{itemize}
\item $\lambda_{i,1}\ge\lambda_{i,2}\ge\ldots$ and $\lambda_{1,j}\ge\lambda_{2,j}\ge\ldots$ for all $i,j$,
\item $\sum_{i,j=1}^n\lambda_{ij}=n$.
\end{itemize}
Ordinary partitions can be regarded as plane partitions with only one non-zero row.
The number $pp(n)$ of plane partitions of $n$ has the fascinating generating function
\[\sum_{n=0}^\infty pp(n)X^n=\prod_{k=1}^\infty\frac{1}{(1-X^k)^k}=1+X+3X^2+6X^3+13X^4+24X^5+\ldots\]
discovered by MacMahon\index{MacMahon} (see \cite[Corollary~7.20.3]{Stanley2}).

\section{Stirling numbers}

We cannot resist presenting a few more exciting combinatorial objects related to power series. Since there are literally hundreds of such combinatorial identities, our selection is inevitably biased by personal taste.

\begin{Def}
A \emph{set partition}\index{set partition}\index{partition!of sets} of $n\in\NN$ is a disjoint union $A_1\mathbin{\dot\cup}\ldots\mathbin{\dot\cup} A_k=\{1,\ldots,n\}$ of non-empty sets $A_i$ in no particular order (we may require $\min A_1<\ldots<\min A_k$ to fix an order). The number of set partitions of $n$ is called the $n$-th \emph{Bell number}\index{Bell number} $b(n)$.\index{*Bn@$b(n)$} The number of set partitions of $n$ with exactly $k$ parts is the \emph{Stirling number of the second kind}\index{Stirling number!of second kind} $\stirr{n}{k}$.\index{*St2@$\stirr{n}{k}$} In particular, $\stirr{n}{n}=\stirr{n}{1}=n$. We set $\stirr{0}{0}=b(0)=1$ describing the empty partition of the empty set. 
\end{Def}

\begin{Ex}
The set partitions of $n=3$ are 
\[\{1,2,3\}=\{1\}\cup\{2,3\}=\{1,3\}\cup\{2\}=\{1,2\}\cup\{3\}=\{1\}\cup\{2\}\cup\{3\}.\] 
Hence, $b(3)=5$ and $\stirr{3}{2}=3$.
\end{Ex}

Unlike the binomial or Gaussian coefficients the Stirling numbers do not obey a symmetry as in Pascal's triangle.
While the generating functions of $b(n)$ and $\stirr{n}{k}$ have no particularly nice shape, there are close approximations which we are about to see.

\begin{Lem}
For $n,k\in\NN_0$,
\begin{equation}\label{stirrek}
\stirr{n+1}{k}=k\stirr{n}{k}+\stirr{n}{k-1}.
\end{equation}
\end{Lem}
\begin{proof}
Without loss of generality, let $1\le k\le n$. 
Let $A_1\cup\ldots\cup A_{k-1}$ be a set partition of $n$ with $k-1$ parts. Then $A_1\cup\ldots\cup A_{k-1}\cup\{n+1\}$ is a set partition of $n+1$ with $k$ parts. Now let $A_1\cup\ldots\cup A_k$ be a set partition of $n$. We can add the number $n+1$ to each of the $k$ sets $A_1,\ldots,A_k$ to obtain a set partition of $n+1$ with $k$ parts. 
Conversely, every set partition of $n+1$ arises in precisely one of the two described ways.
\end{proof}

\begin{Lem}\label{bellrek}
For $n\in\NN_0$, 
\[b(n+1)=\sum_{k=0}^n\binom{n}{k}b(k).\]
\end{Lem}
\begin{proof}
Every set partition $\mathcal{A}$ of $n+1$ has a unique part $A$ containing $n+1$. If $k:=|A|-1$, there are $\binom{n}{k}$ choices for $A$. Moreover, $\mathcal{A}\setminus\{ A\}$ is a uniquely determined partition of the set $\{1,\ldots,n\}\setminus A$ with $n-k$ elements. Hence, there are $b(n-k)$ possibilities for this partition. Consequently,
\[b(n+1)=\sum_{k=0}^n\binom{n}{k}b(n-k)=\sum_{k=0}^n\binom{n}{k}b(k).\qedhere\]
\end{proof}

\begin{Thm}
For $n\in\NN_0$ we have
\begin{align*}
\sum_{k=0}^n\stirr{n}{k}X^k&=\exp(-X)\sum_{k=0}^\infty\frac{k^n}{k!}X^k,\\
\sum_{k=0}^\infty \frac{b(k)}{k!}X^k&=\exp\bigl(\exp(X)-1\bigr).
\end{align*}
\end{Thm}
\begin{proof}\hfill
\begin{enumerate}[(i)]
\item For $n=0$, we have 
\[\exp(-X)\sum_{k=0}^\infty\frac{1}{k!}X^k=\exp(-X)\exp(X)=\exp(0)=1\]
as claimed. Assuming the claim for $n$, we have
\begin{align*}
\sum_{k=0}^{n+1}\stirr{n+1}{k}X^k&\overset{\eqref{stirrek}}{=}\sum k\stirr{n}{k}X^k+\sum\stirr{n}{k-1}X^k=X\Bigl(\sum\stirr{n}{k}X^k\Bigr)'+X\sum\stirr{n}{k}X^k\\
&=X\Bigl(\exp(-X)\sum\frac{k^n}{k!}X^k\Bigr)'+X\exp(-X)\sum\frac{k^n}{k!}X^k\\
&=\exp(-X)\sum\frac{k^{n+1}}{k!}X^k.
\end{align*}

\item Since $\exp(X)-1\in (X)$, we can substitute $X$ by $\exp(X)-1$ in $\exp(X)$. 
Let 
\[\alpha:=\exp\bigl(\exp(X)-1\bigr)=\sum\frac{a_n}{n!}X^n.\]
Then $a_0=\exp(\exp(0)-1)=\exp(0)=1=b(0)$. The chain rule gives
\begin{align*}
\sum_{n=0}^\infty\frac{a_{n+1}}{n!}X^n&=\alpha'=\exp(X)\exp(\exp(X)-1)\\
&=\Bigl(\sum_{k=0}^\infty\frac{1}{k!}X^k\Bigr)\Bigl(\sum_{k=0}^\infty\frac{a_k}{k!}X^k\Bigr)=\sum_{n=0}^\infty\sum_{k=0}^n\frac{a_k}{k!(n-k)!}X^n.
\end{align*}
Therefore, $a_{n+1}=\sum_{k=0}^n\binom{n}{k}a_k$ for $n\ge 0$ and the claim follows from \autoref{bellrek}.\qedhere
\end{enumerate}
\end{proof}

Now we discuss permutations. 

\begin{Def}
Let $S_n$\index{*Sn@$S_n$} be the symmetric group consisting of all permutations on the set $\{1,\ldots,n\}$.
The number of permutations in $S_n$ with exactly $k$ (disjoint) cycles including fixed points is denoted by the \emph{Stirling number of the first kind}\index{Stirling number!of first kind} $\stir{n}{k}$.\index{*St1@$\stir{n}{k}$}
By agreement, $\stir{0}{0}=1$ (the identity on the empty set has zero cycles). 
\end{Def}

\begin{Ex}
There are $\stir{4}{2}=11$ permutations in $S_4$ with exactly two cycles: 
\begin{gather*}
(1,2,3)(4),\ (1,3,2)(4),\ (1,2,4)(3),\ (1,4,2)(3),\ (1,3,4)(2),\ (1,4,3)(2),\\ (1)(2,3,4),\ (1)(2,4,3),\ (1,2)(3,4),\ (1,3)(2,4),\ (1,4)(2,3). 
\end{gather*}
\end{Ex}

Since $|S_n|=n!$, there is no need for a generating function of the number of permutations.

\begin{Lem}
For $k,n\in\NN_0$,
\begin{equation}\label{lemstir}
\stir{n+1}{k}=\stir{n}{k-1}+n\stir{n}{k}.
\end{equation}
\end{Lem}
\begin{proof}
Without loss of generality, let $1\le k\le n$.
Let $\sigma\in S_n$ with exactly $k-1$ cycles. By appending the $1$-cycle $(n+1)$ to $\sigma$ we obtain a permutation counted by $\stir{n+1}{k}$. 
Now assume that $\sigma$ has $k$ cycles. When we write $\sigma$ as a sequence of $n$ numbers and $2k$ parentheses, there are $n$ meaningful positions where we can add the digit $n+1$. For example, there are three ways to add $4$ in $\sigma=(1,2)(3)$, namely
\[(4,1,2)(3),\quad(1,4,2)(3),\quad(1,2)(4,3).\]
This yields $n$ distinct permutations counted by $\stir{n+1}{k}$. Conversely, every permutation counted by $\stir{n+1}{k}$ arises in precisely one of the described ways.
\end{proof}

While the recurrence relations we have seen so far appear arbitrary, they can be explained in a unified way (see \cite{Konvalina}).

It is time to present the next dual pair of formulas resembling Theorems~\ref{GaussBT} and \ref{coreuler}. 

\begin{Thm}\label{genstir}
The following generating functions of the Stirling numbers hold for $n\in\NN_0$:
\begin{empheq}[box=\fbox]{align*}
\prod_{k=0}^{n-1}(1+kX)&=\sum_{k=0}^n\stir{n}{n-k}X^k,\\
\prod_{k=1}^n\frac{1}{1-kX}&=\sum_{k=0}^\infty\stirr{n+k}{n}X^k.
\end{empheq}
\end{Thm}
\begin{proof}
This is another induction on $n$.
\begin{enumerate}[(i)]
\item The case $n=0$ yields $1$ on both sides of the equation. Assuming the claim for $n$, we compute
\begin{align*}
\prod_{k=0}^n(1+kX)&=(1+nX)\sum\stir{n}{n-k}X^k=\sum\biggl(\stir{n}{n-k}+n\stir{n}{n-k+1}\biggr)X^k\\
&\overset{\eqref{lemstir}}{=}\sum\stir{n+1}{n+1-k}X^k.
\end{align*}

\item For $n=0$, we get $1$ on both sides. Assume the claim for $n-1$. Then
\begin{align*}
(1-nX)\sum_{k=0}^\infty\stirr{n+k}{n}X^k&=\sum\biggl(\stirr{n+k}{n}-n\stirr{n+k-1}{n}\biggr)X^k\\
&\overset{\eqref{stirrek}}{=}\sum\stirr{n-1+k}{n-1}X^k=\prod_{k=1}^{n-1}\frac{1}{1-kX}.\qedhere
\end{align*}
\end{enumerate}
\end{proof}

For those who still do not have enough, the next exercise might be of interest.

\begin{A}\label{bern}\hfill
\begin{enumerate}[(a)]
\item Prove \emph{Vandermonde's identity}\index{Vandermonde's identity} $\sum_{k=0}^n\binom{a}{k}\binom{b}{n-k}=\binom{a+b}{n}$ for all $a,b\in\CC$ by using Newton's binomial theorem.

\item For every prime $p$ and $1<k<p$, show that $\stir{p}{k}$ is divisible by $p$ (a property shared with $\binom{p}{k}$).

\item Prove that 
\[(-1)^n\frac{\log(1-X)^n}{n!}=\sum_{k=0}^\infty\stir{k}{n}\frac{X^k}{k!}\]
for $n\in\NN_0$. 

\item Determine all $n\in\NN$ such that the Catalan number $c_n$ is odd. \\
\textit{Hint:} Consider the generating function modulo $2$.

\item The \emph{Bernoulli numbers}\index{Bernoulli numbers} $b_n\in\QQ$\index{*Bn@$b_n$} are defined directly by their (exponential) generating function
\[\frac{X}{\exp(X)-1}=\sum_{n=0}^\infty \frac{b_n}{n!}X^n.\]
Compute $b_0,\ldots,b_3$ and show that $b_{2n+1}=0$ for every $n\in\NN$.\\
\textit{Hint:} Replace $X$ by $-X$.
\end{enumerate}
\end{A}

The \emph{cycle type}\index{cycle type} of a permutation $\sigma\in S_n$ is denoted by $(1^{a_1},\ldots,n^{a_n})$, meaning that $\sigma$ has precisely $a_k$ cycles of length $k$. 

\begin{Lem}\label{lemperm}
The number of permutations $\sigma\in S_n$ with cycle type $(1^{a_1},\ldots,n^{a_n})$ is
\[\frac{n!}{1^{a_1}\ldots n^{a_n}a_1!\ldots a_n!}.\]
\end{Lem}
\begin{proof}
Each cycle of $\sigma$ determines a subset of $\{1,\ldots,n\}$. The number of possibilities to choose such subsets is given by the multinomial coefficient
\[\frac{n!}{(1!)^{a_1}\ldots (n!)^{a_n}}.\]
Since the $a_i$ subsets of size $i$ can be permuted in $a_i!$ ways, each corresponding to the same permutation (as disjoint cycles commute), the number of relevant choices is only
\[\frac{n!}{(1!)^{a_1}\ldots (n!)^{a_n}a_1!\ldots a_n!}.\]
A given subset $\{\lambda_1,\ldots,\lambda_k\}\subseteq\{1,\ldots,n\}$ can be arranged in $k!$ permutations, but only $(k-1)!$ different cycles, since $(\lambda_1,\ldots,\lambda_k)=(\lambda_2,\ldots,\lambda_k,\lambda_1)=\ldots$. 
Hence, the number of permutations in question is
\[\frac{n!}{(1!)^{a_1}\ldots (n!)^{a_n}a_1!\ldots a_n!}\bigl((1-1)!\bigr)^{a_1}\ldots\bigl((n-1)!\bigr)^{a_n}=\frac{n!}{1^{a_1}\ldots n^{a_n}a_1!\ldots a_n!}.\qedhere\]
\end{proof}

The following is a sibling to Glaisher's theorem. For a non-negative real number $r$ we denote the largest integer $n\le r$ by $n=\lfloor r\rfloor$.

\begin{Thm}[\textsc{Erd{\H{o}}s--Tur{\'a}n}]\label{Turan}\index{Erd{\H{o}}s--Tur{\'a}n}
Let $n,d\in\NN$. The number of permutations in $S_n$ whose cycle lengths are not divisible by $d$ is
\[n!\prod_{k=1}^{\lfloor n/d\rfloor}\frac{kd-1}{kd}.\]
\end{Thm}
\begin{proof}
According to \cite{Maroti}, the idea of the proof is credited to Pólya.\index{Pólya}
We need to count permutations with cycle type $(1^{a_1},\ldots,n^{a_n})$ where $a_k=0$ whenever $d\mid k$. 
By \autoref{lemperm}, the total number divided by $n!$ is the coefficient of $X^n$ in 
\begin{align*}
\prod_{\substack{k=1\\d\,\nmid\, k}}^\infty\sum_{a=0}^\infty \frac{1}{a!}\Bigl(\frac{X^k}{k}\Bigr)^a&=\prod_{d\,\nmid\, k}\exp\Bigl(\frac{X^k}{k}\Bigr)\overset{\eqref{func2}}{=}\exp\Bigl(\sum_{d\,\nmid\, k}\frac{X^k}{k}\Bigr)=\exp\Bigl(\sum_{k=1}^\infty\frac{X^k}{k}-\sum_{k=1}^\infty\frac{X^{dk}}{dk}\Bigr)\\
&=\exp\Bigl(-\log(1-X)+\frac{1}{d}\log(1-X^d)\Bigr)\overset{\eqref{funclog}}{=}\sqrt[d]{1-X^d}\,\frac{1}{1-X}\\
&=\frac{1-X^d}{1-X}(1-X^d)^{\frac{1-d}{d}}\overset{\eqref{newtoneq}}{=}\Bigl(\sum_{r=0}^{d-1} X^r\Bigr)\biggl(\sum_{q=0}^\infty\binom{(1-d)/d}{q}(-X^d)^q\biggr).
\end{align*}
Therein, $X^n$ appears if and only if $n=qd+r$ with $0\le r<d$ and $q=\lfloor n/d\rfloor$ (euclidean division). In this case the coefficient is
\[(-1)^q\binom{(1-d)/d}{q}=(-1)^q\prod_{k=1}^q\frac{\frac{1}{d}-k}{k}=\prod_{k=1}^q\frac{kd-1}{kd}.\qedhere\]
\end{proof}

\begin{Ex}
A permutation has odd order as an element of $S_n$ if and only if all its cycles have odd length. The number of such permutations is therefore
\[n!\prod_{k=1}^{\lfloor n/2\rfloor}\frac{2k-1}{2k}=\begin{cases}
1^2\cdot 3^2\cdot\ldots\cdot (n-1)^2&\text{if $n$ is even},\\
1^2\cdot 3^2\cdot\ldots\cdot (n-2)^2\cdot n&\text{if $n$ is odd}.
\end{cases}\]
\end{Ex}

\begin{A}
Find and prove a similar formula for the number of permutations $\sigma\in S_n$ whose cycle lengths are all divisible by $d$.
\end{A}

\begin{Def}
A pair $(a,b)$ with $1\le a<b\le n$ is called an \emph{inversion}\index{inversion} of $\sigma\in S_n$ if $\sigma(a)>\sigma(b)$. Let $\inv(\sigma)$\index{*Inv@$\inv(\sigma)$} be the number of inversions of $\sigma$ and let $\rho(n,k):=|\{\sigma\in S_n:\inv(\sigma)=k\}|$.\index{*Rho@$\rho(n,k)$} As usual, let $\rho(n,k):=0$ for $k<0$. 
\end{Def}

Obviously, $0\le\inv(\sigma)\le\binom{n}{2}$ for all $\sigma\in S_n$. Moreover, $\id$ is the only permutation with no inversions and 
\[\pi=\begin{pmatrix}
1&2&\cdots&n\\
n&n-1&\cdots&1
\end{pmatrix}=(1,n)(2,n-1)\ldots\]
is the only permutation with $\binom{n}{2}$ inversions. If $(a,b)$ is an inversion of $\sigma$, then $(a,b)$ is no inversion of $\pi\sigma$ and vice versa. Hence, $\inv(\pi\sigma)=\binom{n}{2}-\inv(\sigma)$ and $\rho(n,k)=\rho\bigl(n,\binom{n}{2}-k\bigr)$ for all $k$. 
It is well-known that $\sgn(\sigma)=(-1)^{\inv(\sigma)}$. 

\begin{Thm}[\textsc{Rodrigues}]\index{Rodrigues}
For $n\in\NN_0$,
\[\sum_{k=0}^{\binom{n}{2}}\rho(n,k)X^k=\frac{X^n!}{(1-X)^n}.\]
\end{Thm}
\begin{proof}
Induction on $n$:
For $n=0$, both sides become $1$. Let $n\ge 2$ and $0\le k\le n$. For $\sigma\in S_{n-1}$ let $\hat\sigma\in S_n$ such that 
\[(\hat\sigma(1),\ldots,\hat\sigma(n))=\bigl(\sigma(1),\ldots,\sigma(k),n,\sigma(k+1),\ldots,\sigma(n-1)\bigr).\]
Then $\inv(\hat\sigma)=\inv(\sigma)+n-k-1$. Since every permutation of $S_n$ arises in this way, we obtain the recursion 
\[\rho(n,k)=\sum_{l=k-n+1}^k\rho(n-1,l).\]
By induction, we have
\[\sum_{k=0}^\infty\rho(n,k)X^k=\sum_{l=0}^\infty\rho(n-1,l)X^l(1+X+\ldots+X^{n-1})=\frac{X^{n-1}!}{(1-X)^{n-1}}\frac{1-X^n}{1-X}=\frac{X^n!}{(1-X)^n}.\qedhere\]
\end{proof}

\begin{Ex}
For $n=3$ we compute
\[\sum_{k=0}^3\rho(3,k)X^k=\frac{(1-X^2)(1-X^3)}{(1-X)(1-X)}=(1+X)(1+X+X^2)=1+2X+2X^2+X^3.\]
\end{Ex}

We insert a well-known application of Bernoulli numbers.

\begin{Thm}[\textsc{Faulhaber}]\index{Faulhaber}
For every $d\in\NN$ there exists a polynomial $\alpha\in\QQ[X]$ of degree $d+1$ such that $1^d+2^d+\ldots+n^d=\alpha(n)$ for every $n\in\NN$.
\end{Thm}
\begin{proof}
We compute the generating function
\begin{align*}
\sum_{d=0}^\infty\Bigl(\sum_{k=0}^{n-1}k^d\Bigr)\frac{X^d}{d!}&=\sum_{k=0}^{n-1}\sum_{d=0}^\infty\frac{(kX)^d}{d!}=\sum_{k=0}^{n-1}\exp(kX)=\sum_{k=0}^{n-1}\exp(X)^k=\frac{\exp(X)^n-1}{\exp(X)-1}\\
&=\frac{\exp(nX)-1}{X}\frac{X}{\exp(X)-1}\overset{\ref{bern}}{=}\sum_{k=0}^\infty\frac{n^{k+1}}{(k+1)!}X^k\sum_{l=0}^\infty\frac{b_l}{l!}X^l\\
&=\sum_{d=0}^\infty\sum_{k=0}^d\Bigl(\frac{n^{k+1}b_{d-k}d!}{(k+1)!(d-k)!}\Bigr)\frac{X^d}{d!}\\
&=\sum_{d=0}^\infty\sum_{k=0}^d\Bigl(\frac{1}{k+1}\binom{d}{k}b_{d-k}n^{k+1}\Bigr)\frac{X^d}{d!}
\end{align*}
and define 
\[\alpha:=\sum_{k=0}^d\frac{1}{k+1}\binom{d}{k}b_{d-k}(X+1)^{k+1}\in\QQ[X].\] 
Since $b_0=1$, $\alpha$ is a polynomial of degree $d+1$ with leading coefficient $\frac{1}{d+1}$.
\end{proof}

\begin{Ex}
For $d=3$ the formula in the proof evaluates with some effort (using \autoref{bern}) to:
\[\alpha=b_3(X+1)+\frac{3}{2}b_2(X+1)^2+b_1(X+1)^3+\frac{1}{4}b_0(X+1)^4=\frac{1}{4}(X+1)^2X^2=\binom{X+1}{2}^2.\]
This is known as \emph{Nicomachus's identity}:\index{Nicomachus' identity}
\[1^3+2^3+\ldots+n^3=(1+2+\ldots+n)^2.\]
\end{Ex}

Even though Faulhaber's formula $1^d+2^d+\ldots+n^d=\alpha(n)$ has not much to do with power series, there still is a dual formula, again featuring Bernoulli numbers:
\[\sum_{k=1}^\infty\frac{1}{k^{2d}}=(-1)^{d+1}\frac{(2\pi)^{2d}b_{2d}}{2(2d)!}\qquad(d\in\NN).\]
Strangely, no such formula is known to hold for odd negative exponents (perhaps because $b_{2d+1}=0$?). In fact, it is unknown if \emph{Apéry's constant}\index{Apéry's constant} $\sum_{k=1}^\infty\frac{1}{k^3}=1{,}202\ldots$ is transcendent.

We end this section with a power series proof of the famous four-square theorem. 

\begin{Thm}[\textsc{Lagrange--Jacobi}]\label{thmLJ}\index{Lagrange--Jacobi}
Every positive integer is the sum of four squares. More precisely,\index{*Qn@$q(n)$}
\[q(n):=\bigl|\{(a,b,c,d)\in\ZZ^4:a^2+b^2+c^2+d^2=n\}\bigr|=8\sum_{4\,\nmid\, d\,\mid\,n}d\]
for $n\in\NN$.
\end{Thm}
\begin{proof}
We follow Hirschhorn~\cite[Section~2.4]{Hirschhorn}.\index{Hirschhorn} Obviously, it suffices to prove the second assertion (due to Jacobi). Since the summands $(-1)^k(2k+1)X^{\frac{k^2+k}{2}}$ in \eqref{Jhoch3} are invariant under the transformation $k\mapsto-k-1$, we can write 
\[\prod_{k=1}^\infty(1-X^k)^3=\frac{1}{2}\sum_{k=-\infty}^\infty(-1)^k(2k+1) X^{\frac{k^2+k}{2}}.\]
Taking the square on both sides yields
\[\alpha:=\prod_{k=1}^{\infty}(1-X^k)^6=\frac{1}{4}\sum_{k,l=-\infty}^\infty (-1)^{k+l}(2k+1)(2l+1)X^{\frac{k^2+k+l^2+l}{2}}.\]
The pairs $(k,l)$ with $k\equiv l\pmod{2}$ are transformed by $(k,l)\mapsto(s,t):=\frac{1}{2}(k+l,k-l)$, while the pairs $k\not\equiv l\pmod{2}$ are transformed by $(s,t):=\frac{1}{2}(k-l-1,k+l+1)$.
Notice that $k=s+t$ and $l=s-t$ or $l=t-s-1$ respectively. Hence,
\begin{align*}
\alpha&=\frac{1}{4}\sum_{s,t=-\infty}^\infty (2s+2t+1)(2s-2t+1)X^{\frac{(s+t)^2+s+t+(s-t)^2+s-t}{2}}\\
&\quad-\frac{1}{4}\sum_{s,t=-\infty}^\infty(2s+2t+1)(2t-2s-1)X^{\frac{(s+t)^2+s+t+(t-s-1)^2+t-s-1}{2}}\\
&=\frac{1}{4}\sum_{s,t}\bigl((2s+1)^2-(2t)^2\bigr)X^{s^2+s+t^2}-\frac{1}{4}\sum_{s,t}\bigl((2t)^2-(2s+1)^2\bigr)X^{s^2+s+t^2}\\
&=\frac{1}{2}\sum_{s,t}\bigl((2s+1)^2-(2t)^2\bigr)X^{s^2+s+t^2}\\
&=\frac{1}{2}\sum_{t=-\infty}^\infty X^{t^2}\sum_{s=-\infty}^\infty(2s+1)^2X^{s^2+s}-\frac{1}{2}\sum_{s=-\infty}^\infty X^{s^2+s}\sum_{t=-\infty}^\infty (2t)^2X^{t^2}.
\end{align*}
For $\beta:=\sum X^{t^2}$ and $\gamma:=\frac{1}{2}\sum X^{s^2+s}$ we have
$\gamma+4X\gamma'=\frac{1}{2}\sum(2s+1)^2X^{s^2+s}$ and therefore
\[\alpha=\beta(\gamma+4X\gamma')-4X\beta'\gamma=\beta\gamma+4X(\beta\gamma'-\beta'\gamma).\]
Now we apply the infinite product rule to \eqref{JTP1} and \eqref{JTP3}:
\begin{align*}
\beta'&=\Bigl(\prod_{k=1}^\infty(1-X^{2k})(1+X^{2k-1})^2\Bigr)'=\beta\sum_{k=1}^\infty\Bigl(2\frac{(2k-1)X^{2k-2}}{1+X^{2k-1}}-\frac{2kX^{2k-1}}{1-X^{2k}}\Bigr),\\
\gamma'&=\Bigl(\prod_{k=1}^\infty(1-X^{2k})(1+X^{2k})^2\Bigr)'=\gamma\sum_{k=1}^\infty\Bigl(2\frac{2kX^{2k-1}}{1+X^{2k}}-\frac{2kX^{2k-1}}{1-X^{2k}}\Bigr).
\end{align*}
We substitute:
\begin{align*}
\alpha&=\beta\gamma\Bigl(1+8\sum_{k=1}^\infty\Bigl(\frac{2kX^{2k}}{1+X^{2k}}-\frac{(2k-1)X^{2k-1}}{1+X^{2k-1}}\Bigr)\Bigr).
\end{align*}
Here, 
\[\beta\gamma=\prod(1-X^{2k})^2(1+X^{2k-1})^2(1+X^{2k})^2=\prod(1-X^{2k})^2(1+X^k)^2=\prod(1-X^{2k})^4(1-X^k)^{-2}.\] 
After we set this off against $\alpha$, it remains
\[\Bigl(\sum_{k=-\infty}^\infty (-1)^kX^{k^2}\Bigr)^4\overset{\eqref{JTP2}}{=}\prod_{k=1}^\infty\frac{(1-X^k)^8}{(1-X^{2k})^4}=\frac{\alpha}{\beta\gamma}=1+8\sum_{k=1}^\infty\Bigl(\frac{2kX^{2k}}{1+X^{2k}}-\frac{(2k-1)X^{2k-1}}{1+X^{2k-1}}\Bigr).\]
Finally we replace $X$ by $-X$:
\begin{align*}
\sum q(n)X^n&=\Bigl(\sum_{k=-\infty}^\infty X^{k^2}\Bigr)^4=1+8\sum_{k=1}^\infty\Bigl(\frac{2kX^{2k}}{1+X^{2k}}+\frac{(2k-1)X^{2k-1}}{1-X^{2k-1}}\Bigr)\\
&=1+8\sum_{k=1}^\infty\Bigl(\frac{(2k-1)X^{2k-1}}{1-X^{2k-1}}+\frac{2kX^{2k}}{1-X^{2k}}-\frac{2kX^{2k}}{1-X^{2k}}+\frac{2kX^{2k}}{1+X^{2k}}\Bigr)\\
&=1+8\sum_{k=1}^\infty\Bigl(\frac{kX^{k}}{1-X^{k}}-\frac{4kX^{4k}}{1-X^{4k}}\Bigr)=1+8\sum_{4\,\nmid\, k}\frac{kX^{k}}{1-X^{k}}\\
&=1+8\sum_{4\,\nmid\, k}k\sum_{l=1}^\infty X^{kl}=1+8\sum_{n=1}^\infty\sum_{4\,\nmid\, d\,\mid\, n}dX^n.\qedhere
\end{align*}
\end{proof}

\begin{Ex}
For $n=28$ we obtain 
\[\sum_{4\,\nmid\, d\,\mid\, 28}d=1+2+7+14=24.\] 
Hence, there are $8\cdot 24=192$ possibilities to express $28$ as a sum of four squares. However, they all arise as permutations and sign-choices of
\[28=5^2+1^2+1^2+1^2=4^2+2^2+2^2+2^2=3^2+3^2+3^2+1^2.\]
\end{Ex}

\autoref{thmLJ} is best possible in the sense that every integer $n\equiv 7\pmod{8}$ is not the sum of three squares since $a^2+b^2+c^2\not\equiv 7\mod{8}$.

If $n,m\in\NN$ are sums of four squares, so is $nm$ by the following identity of Euler (encoding the multiplicativity of the norm in \emph{Hamilton's quaternion}\index{Hamilton's quaternion} skew field):
\begin{gather*}
(a_1^2+a_2^2+a_3^2+a_4^2)(b_1^2+b_2^2+b_3^2+b_4^2)=(a_1b_1+a_2b_2+a_3b_3+a_4b_4)^2\\
+(a_1b_2-a_2b_1+a_3b_4-a_4b_3)^2+(a_1b_3-a_3b_1+a_4b_2-a_2b_4)^2+(a_1b_4-a_4b_1+a_2b_3-a_3b_2)^2
\end{gather*}
This reduces the proof of the first assertion (Lagrange's) of \autoref{thmLJ} to the case where $n$ is a prime.

\emph{Waring's problem}\index{Waring's problem} ask for the smallest number $g(k)$\index{*Gn@$g(n)$} such that every positive integer is the sum of $g(k)$ non-negative $k$-th powers. Hilbert proved that $g(k)<\infty$ for all $k\in\NN$. We have $g(1)=1$, $g(2)=4$ (\autoref{thmLJ}), $g(3)=9$, $g(4)=19$ and in general it is conjectured that
\[g(k)=\Bigl\lfloor\Bigl(\frac{3}{2}\Bigr)^k\Bigr\rfloor+2^k-2\]
(see \cite{Waring}).
Curiously, only the numbers $23=2\cdot2^3+7\cdot 1^3$ and $239=2\cdot 4^3+4\cdot 3^3+3\cdot 1^3$ require nine cubes. 
It is even conjectured that every sufficiently large integer is a sum of only four non-negative cubes (see \cite{4cubes}).

\section{Multivariate power series}\label{secmult}

In \autoref{important} it became clear that power series in more than one indeterminate make sense. We give proper definitions now.

\begin{Def}\label{defmulti}\hfill
\begin{enumerate}[(i)]
\item The ring of formal power series in $n$ indeterminates $X_1,\ldots,X_n$ over a field $K$ is defined inductively via
\[K[[X_1,\ldots,X_n]]:=K[[X_1,\ldots,X_{n-1}]][[X_n]].\]\index{*KXY2@$K[[X_1,\ldots,X_n]]$}
Its elements have the form
\[\alpha=\sum_{k_1,\ldots,k_n\ge 0}a_{k_1,\ldots,k_n}X_1^{k_1}\ldots X_n^{k_n}\]
where $a_{k_1,\ldots,k_n}\in K$. We (still) call $a_{0,\ldots,0}$ the \emph{constant term} of $\alpha$. Let 
\begin{align*}
\inf(\alpha)&:=\inf\{k_1+\ldots+k_n:a_{k_1,\ldots,k_n}\ne0\},\\
|\alpha|&:=2^{-\inf(\alpha)}.
\end{align*} 

\item If all but finitely many coefficients of $\alpha$ are zero, we call $\alpha$ a (formal) polynomial in $X_1,\ldots,X_n$. In this case, 
\[\deg(\alpha):=\sup\{k_1+\ldots+k_n:a_{k_1,\ldots,k_n}\ne 0\}\] 
is the \emph{degree}\index{degree!of multivariant polynomial} of $\alpha$, where $\deg(0)=\sup\varnothing=-\infty$. Moreover, a polynomial $\alpha$ is called \emph{homogeneous}\index{polynomial!homogeneous} if all monomials occurring in $\alpha$ (with non-zero coefficient) have the same degree.
The set of polynomials is denoted by $K[X_1,\ldots,X_n]$.\index{*KXY@$K[X_1,\ldots,X_n]$}
\end{enumerate}
\end{Def}

Once we have convinced ourselves that \autoref{intdom} remains true when $K$ is replaced by an integral domain, it becomes evident that also $K[[X_1,\ldots,X_n]]$ is an integral domain. Likewise, the norm still gives rise to a complete ultrametric (to prove $|\alpha\beta|=|\alpha||\beta|$ one may assume that $\alpha$ and $\beta$ are homogeneous polynomials) and the crucial \autoref{infsum} holds in $K[[X_1,\ldots,X_n]]$ too. 
We stress that this metric is finer than the one induced from $K[[X_1,\ldots,X_{n-1}]]$ as, for example, $\lim_{k\to\infty}X_1^kX_2$ converges in the former, but not in the latter (with $n=2$). 
Moreover, a power series $\alpha$ is invertible in $K[[X_1,\ldots,X_n]]$ if and only if its constant term is non-zero. Indeed, after scaling, the constant term is $1$ and 
\[\alpha^{-1}=\frac{1}{1-(1-\alpha)}=\sum_{k=0}^\infty(1-\alpha)^k\] 
converges.

The degree function equips $K[X_1,\ldots,X_n]$ with a \emph{grading},\index{grading} i.\,e. we have
\[K[X_1,\ldots,X_n]=\bigoplus_{d=0}^\infty P_d\]
and $P_dP_e\subseteq P_{d+e}$ where $P_d$ denotes the set of homogeneous polynomials of degree $d$.
In the following we will restrict ourselves mostly to polynomials of a special type. Note that if $\alpha,\beta_1,\ldots,\beta_n\in K[X_1,\ldots,X_n]$, we can substitute $X_i$ by $\beta_i$ in $\alpha$ to obtain $\alpha(\beta_1,\ldots,\beta_n)\in K[X_1,\ldots,X_n]$. It is important that these substitutions happen simultaneously and not one after the other (more about this at the end of the section).

\begin{Def}
A polynomial $\alpha\in K[X_1,\ldots,X_n]$ is called \emph{symmetric}\index{polynomial!symmetric} if
\[\alpha(X_{\pi(1)},\ldots,X_{\pi(n)})=\alpha(X_1,\ldots,X_n)\] 
for all permutations $\pi\in S_n$. 
\end{Def}

It is easy to see that the symmetric polynomials form a subring of $K[X_1,\ldots,X_n]$. 

\begin{Ex}\hfill
\begin{enumerate}[(i)]
\item The \emph{elementary symmetric polynomials}\index{polynomial!elementary symmetric}\index{*Sigmak@$\sigma_k$} are $\sigma_0:=1$ and 
\[\sigma_k:=\sum_{1\le i_1<\ldots<i_k\le n}X_{i_1}\ldots X_{i_k}\qquad(k\ge 1).\]
Note that $\sigma_k=0$ for $k>n$ (empty sum).

\item The \emph{complete symmetric polynomials}\index{polynomial!complete symmetric}\index{*Tauk@$\tau_k$} are $\tau_0:=1$ and
\[\tau_k:=\sum_{1\le i_1\le \ldots\le i_k\le n}X_{i_1}\ldots X_{i_k}\qquad(k\ge 1).\]

\item The \emph{power sum polynomials}\index{polynomial!power sum}\index{*Rhok@$\rho_k$} are $\rho_k:=X_1^k+\ldots+X_n^k$ for $k\ge 0$.
\end{enumerate}
\end{Ex}

Keep in mind that $\sigma_k$, $\tau_k$ and $\rho_k$ depend on $n$.
All three sets of polynomials are homogeneous. The elementary and complete symmetric polynomials are special instances of \emph{Schur polynomials},\index{Schur polynomial} which we do not attempt to define here. 

\begin{Thm}[\textsc{Vieta}]\label{vieta}\index{Vieta}
The following identities hold in $K[[X_1,\ldots,X_n,Y]]$:
\begin{empheq}[box=\fbox]{align}
\prod_{k=1}^n(1+X_kY)&=\sum_{k=0}^n\sigma_kY^k,\label{vieta1}\\
\prod_{k=1}^n\frac{1}{1-X_kY}&=\sum_{k=0}^\infty \tau_kY^k.\label{vieta2}
\end{empheq}
\end{Thm}
\begin{proof}
The first equation is only a matter of expanding the product. The second equation follows from
\[
\prod_{k=1}^n\frac{1}{1-X_kY}=\prod_{k=1}^n\sum_{l=0}^\infty (X_kY)^l=\sum_{k=0}^\infty\Bigl(\sum_{l_1+\ldots+l_n=k}X_1^{l_1}\ldots X_n^{l_n}\Bigr)Y^k=\sum_{k=0}^\infty \tau_kY^k.\qedhere
\]
\end{proof}

When we specialize $X_1=\ldots=X_n=1$ in Vieta's theorem (as we may), we recover the generating functions of the binomial coefficients and the multiset counting coefficients from \autoref{exgen}. When we substitute $X_k=k$ for $k=1,\ldots,n$, we obtain a new formula for the Stirling numbers by virtue of \autoref{genstir}. 

It is easy to see that the grading by degree carries over to symmetric polynomials. The following theorem shows that the elementary symmetric polynomials are the building blocks of all symmetric polynomials.

\begin{Thm}[Fundamental theorem on symmetric polynomials]\label{sympoly}\index{polynomial!symmetric!fundamental theorem}
For every symmetric polynomial $\alpha\in K[X_1,\ldots,X_n]$ there exists a unique $\gamma\in K[X_1,\ldots,X_n]$ such that $\alpha=\gamma(\sigma_1,\ldots,\sigma_n)$. 
\end{Thm}
\begin{proof}
We first prove the \emph{existence} of $\gamma$: Without loss of generality, let
\[\alpha=\sum_{i_1,\ldots,i_n}a_{i_1,\ldots,i_n}X_1^{i_1}\ldots X_n^{i_n}\ne 0.\]
We order the tuples $(i_1,\ldots,i_n)$ lexicographically and argue by induction on 
\[f(\alpha):=\max\bigl\{(i_1,\ldots,i_n):a_{i_1,\ldots,i_n}\ne 0\bigr\}\]
(see \autoref{exsym} below for an illustration). 
If $f(\alpha)=(0,\ldots,0)$, then $\gamma:=\alpha=a_{0,\ldots,0}\in K$. Now let $f(\alpha)=(d_1,\ldots,d_n)>(0,\ldots,0)$.
Since $\alpha=\alpha(X_{\pi(1)},\ldots,X_{\pi(n)})$ for all $\pi\in S_n$, $d_1\ge\ldots\ge d_n$. Let 
\[\beta:=a_{d_1,\ldots,d_n}\sigma_1^{d_1-d_2}\sigma_2^{d_2-d_3}\ldots\sigma_{n-1}^{d_{n-1}-d_n}\sigma_n^{d_n}.\]
Then we have $f(\sigma_k^{d_k-d_{k+1}})=(d_k-d_{k+1})f(\sigma_k)=(d_k-d_{k+1},\ldots,d_k-d_{k+1},0,\ldots,0)$ and
\[f(\beta)=f(\sigma_1^{d_1-d_2})+\ldots+f(\sigma_n^{d_n})=(d_1,\ldots,d_n).\]
Hence, the symmetric polynomial $\alpha-\beta$ satisfies $f(\alpha-\beta)<(d_1,\ldots,d_n)$ and the existence of $\gamma$ follows by induction.

Now we show the \emph{uniqueness} of $\gamma$: Let $\gamma,\delta\in K[X_1,\ldots,X_n]$ such that $\gamma(\sigma_1,\ldots,\sigma_n)=\delta(\sigma_1,\ldots,\sigma_n)$. For $\rho:=\gamma-\delta$ it follows that $\rho(\sigma_1,\ldots,\sigma_n)=0$. We have to show that $\rho=0$. By way of contradiction, suppose $\rho\ne 0$. Let $d_1\ge\ldots\ge d_n$ be the lexicographically largest $n$-tuple such that the coefficient of $X_1^{d_1-d_2}X_2^{d_2-d_3}\ldots X_n^{d_n}$ in $\rho$ is non-zero. As above, $f(\sigma_1^{d_1-d_2}\ldots\sigma_n^{d_n})=(d_1,\ldots,d_n)$. For every other summand $X_1^{e_1-e_2}\ldots X_n^{e_n}$ of $\rho$ we obtain $f(\sigma_1^{e_1-e_2}\ldots\sigma_n^{e_n})<(d_1,\ldots,d_n)$. This yields $f\bigl(\rho(\sigma_1,\ldots,\sigma_n)\bigr)=(d_1,\ldots,d_n)$ in contradiction to $\rho(\sigma_1,\ldots,\sigma_n)=0$.
\end{proof}

\begin{Ex}\label{exsym}
Consider $\alpha=XY^3+X^3Y-X-Y\in K[X,Y]$. With the notation from the proof above, $f(\alpha)=(3,1)$ and \[\beta:=\sigma_1^2\sigma_2=(X+Y)^2XY=X^3Y+2X^2Y^2+XY^3.\]
Thus, $\alpha-\beta=-2X^2Y^2-X-Y$. In the next step we have $f(\alpha-\beta)=(2,2)$ and
\[\beta_2:=-2\sigma_2^2=-2X^2Y^2.\]
It remains: $\alpha-\beta-\beta_2=-X-Y=-\sigma_1$. Finally,
\[\alpha=\beta+\beta_2-\sigma_1=\sigma_1^2\sigma_2-2\sigma_2^2-\sigma_1=\gamma(\sigma_1,\sigma_2)\]
where $\gamma=X^2Y-2Y^2-X$.
\end{Ex}

From an algebraic point of view, \autoref{sympoly} (applied to $\alpha=0$) states that the elementary symmetric polynomials $\sigma_1,\ldots,\sigma_n$ are algebraically independent over $K$, so they form a transcendence basis of $K(X_1,\ldots,X_n)$ (recall that $K(X_1,\ldots,X_n)$ has transcendence degree $n$). 
The identities in the next theorem express the $\sigma_i$ recursively in terms of the $\tau_j$ and in terms of the $\rho_j$. So the latter sets of symmetric polynomials form transcendence bases too. It is no coincidence that $\deg(\sigma_k)=\deg(\tau_k)=\deg(\rho_k)=k$ for $k\le n$. A theorem from invariant theory (in characteristic $0$) implies that 
any algebraically independent, homogeneous generators of the ring of symmetric polynomials have degrees $1,\ldots,n$ in some order (see \cite[Proposition~3.7]{Humphreys}). 

\begin{Thm}[\textsc{Girard--Newton} identities]\label{GNI}\index{Girard--Newton identities}
The following identities hold in $K[X_1,\ldots,X_n]$ for all $n,k\in\NN$:
\begin{empheq}[box=\fbox]{align*}
\sum_{i=0}^k(-1)^i\sigma_i\tau_{k-i}&=0,\\
\sum_{i=1}^k\rho_i\tau_{k-i}&=k\tau_k,\\
\sum_{i=1}^k(-1)^i\sigma_{k-i}\rho_i&=-k\sigma_k.
\end{empheq}
\end{Thm}
\begin{proof}
Let $\sigma=\sum(-1)^k\sigma_kY^k=\prod(1-X_kY)$ and $\tau:=\sum\tau_kY^k=\prod\frac{1}{1-X_kY}$ as in Vieta's theorem.
\begin{enumerate}[(i)]
\item The claim follows by comparing coefficients of $Y^k$ in
\[1=\sigma\tau=\sum_{k=0}^\infty\Bigl(\sum_{i=0}^k(-1)^i\sigma_i\tau_{k-i}\Bigr)Y^k.\]

\item We differentiate with respect to $Y$ using the product rule while noticing that $\bigl(\frac{1}{1-X_kY}\bigr)'=\frac{X_k}{(1-X_kY)^2}$:
\begin{align*}
\sum_{k=1}^\infty k\tau_kY^k&=Y\tau'=\tau\sum_{k=1}^n\frac{X_kY}{1-X_kY}=\tau\sum_{k=1}^n\sum_{i=1}^\infty(X_kY)^i\\
&=\tau\sum_{i=1}^\infty\rho_iY^i=\sum_{k=1}^\infty\Bigl(\sum_{i=1}^k\rho_i\tau_{k-i}\Bigr)Y^k.
\end{align*}

\item We differentiate again with respect to $Y$ (this idea is often attributed to \cite[p.~212]{Berlekamp}):
\begin{align*}
-\sum_{k=0}^\infty(-1)^kk\sigma_kY^k&=-Y\sigma'=\sigma\sum_{k=1}^n\frac{X_kY}{1-X_kY}=\sigma\sum_{i=1}^\infty\rho_iY^i=\sum_{k=1}^\infty\Bigl(\sum_{i=1}^{k}(-1)^{k-i}\sigma_{k-i}\rho_{i}\Bigr) Y^k.\qedhere
\end{align*}
\end{enumerate}
\end{proof}

Now that we know that each of the $\sigma_i$, $\tau_i$ and $\rho_i$ can be expressed by the other two sets of polynomials, it is natural to ask for explicit formulas. This is achieved by Waring's formula. Here $P(n)$ stands for the set of partitions of $n$ as introduced in \autoref{defpart}. 

\begin{Thm}[\textsc{Waring}'s formula]\index{Waring's formula}
The following holds in $\CC[X_1,\ldots,X_n]$ for all $n,k\in\NN$:
\begin{empheq}[box=\fbox]{align*}
\rho_k&=(-1)^kk\sum_{(1^{a_1},\ldots,k^{a_k})\in P(k)}(-1)^{a_1+\ldots+a_k}\frac{(a_1+\ldots+a_k-1)!}{a_1!\ldots a_k!}\sigma_1^{a_1}\ldots\sigma_k^{a_k},\\
&=-k\sum_{(1^{a_1},\ldots,k^{a_k})\in P(k)}(-1)^{a_1+\ldots+a_k}\frac{(a_1+\ldots+a_k-1)!}{a_1!\ldots a_k!}\tau_1^{a_1}\ldots\tau_k^{a_k}.
\end{empheq}
\end{Thm}
\begin{proof}
We introduce a new variable $Y$ and compute in $\CC[[X_1,\ldots,X_n,Y]]$. The generating function of $(-1)^k\frac{\rho_k}{k}$ is
\begin{align*}
\sum_{k=1}^\infty (-1)^k\frac{\rho_k}{k}Y^k&=-\sum_{i=1}^n\sum_{k=1}^\infty (-1)^{k-1}\frac{(X_iY)^k}{k}=-\sum_{i=1}^n\log(1+X_iY)\overset{\eqref{funclog}}{=}-\log\Bigl(\prod_{i=1}^n(1+X_iY)\Bigr)\\
&\overset{\eqref{vieta1}}{=}-\log\Bigl(1+\sum_{i=1}^n\sigma_iY^i\Bigr)=\sum_{l=1}^\infty\frac{(-1)^l}{l}\Bigl(\sum_{i=1}^n\sigma_iY^i\Bigr)^l.
\end{align*}
Now we use the multinomial theorem to expand the inner sum:
\begin{align*}
\sum_{k=1}^\infty (-1)^k\frac{\rho_k}{k}Y^k&=\sum_{l=1}^\infty\frac{(-1)^l}{l}\sum_{a_1+\ldots+a_n=l}\frac{l!}{a_1!\ldots a_n!}\sigma_1^{a_1}\ldots\sigma_n^{a_n}Y^{a_1+2a_2+\ldots+na_n}\\
&=\sum_{k=1}^\infty\sum_{(1^{a_1},\ldots,k^{a_k})\in P(k)}(-1)^{a_1+\ldots+a_k}\frac{(a_1+\ldots+a_k-1)!}{a_1!\ldots a_k!}\sigma_1^{a_1}\ldots\sigma_k^{a_k} Y^k.
\end{align*}
Note that $\sigma_k=0$ for $k>n$. This implies the first equation. 
For the second we start similarly:
\[\sum_{k=1}^\infty \frac{\rho_k}{k}Y^k=\sum_{i=1}^n\sum_{k=1}^\infty \frac{(X_iY)^k}{k}=\sum_{i=1}^n\log\bigl((1-X_iY)^{-1}\bigr)=\log\Bigl(\prod_{i=1}^n\frac{1}{1-X_iY}\Bigr)
\overset{\eqref{vieta2}}{=}\log\Bigl(1+\sum_{i=1}^\infty\tau_iY^i\Bigr).\]
Since we are only interested in the coefficient of $Y^k$, we can truncate the sum to
\[\log\Bigl(1+\sum_{i=1}^k\tau_iY^i\Bigr)=-\sum_{l=1}^\infty\frac{(-1)^l}{l}\sum_{a_1+\ldots+a_k=l}\frac{l!}{a_1!\ldots a_k!}\tau_1^{a_1}\ldots\tau_k^{a_k}Y^{a_1+2a_2+\ldots+ka_k}\]
and argue as before.
\end{proof}

The first instances of Waring's formula are
\begin{align*}
\rho_1=\sigma_1,&&\rho_2=\sigma_1^2-2\sigma_2,&&\rho_3=\sigma_1^3-3\sigma_1\sigma_2+3\sigma_3.
\end{align*}

\begin{Ex}
Since we are dealing with polynomials, it is legitimate to replace the indeterminates by actual numbers. Let $x,y,z\in\CC$ be the roots of
\[\alpha=X^3+2X^2-3X+1\in\CC[X]\]
(guaranteed to exist by the fundamental theorem of algebra). 
By Vieta's theorem, 
\[\sigma_1(x,y,z)=-2,\qquad\sigma_2(x,y,z)=-3,\qquad\sigma_3(x,y,z)=-1.\]
We compute with the first Waring formula
\[x^3+y^3+z^3=\rho_3(x,y,z)=(-2)^3-3(-2)(-3)+3(-1)=-29\]
without knowing what $x,y,z$ are! 
Here is an alternative approach for those who like matrices. The companion matrix
\[A=\begin{pmatrix}
0&0&-1\\
1&0&3\\
0&1&-2
\end{pmatrix}
\]
of $\alpha$ has characteristic polynomial $\alpha$. Hence, the eigenvalues of $A^k$ are $x^k$, $y^k$ and $z^k$. This shows $\rho_k(x,y,z)=\tr(A^k)$. 
\end{Ex}

We invite the reader to prove the other four transition formulas.

\begin{A}\label{exwaring}
Show that the following holds in $\CC[X_1,\ldots,X_n]$ for all $n,k\in\NN$:
\begin{align}
\sigma_k&=(-1)^k\sum_{(1^{a_1},\ldots,k^{a_k})\in P(k)}(-1)^{a_1+\ldots+a_k}\frac{(a_1+\ldots+a_k)!}{a_1!\ldots a_k!}\tau_1^{a_1}\ldots\tau_k^{a_k},\notag\\
&=(-1)^k\sum_{(1^{a_1},\ldots,k^{a_k})\in P(k)}\frac{(-1)^{a_1+\ldots+a_k}}{1^{a_1}a_1!\ldots k^{a_k}a_k!}\rho_1^{a_1}\ldots \rho_k^{a_k},\label{Wsec}\\
\tau_k&=(-1)^k\sum_{(1^{a_1},\ldots,k^{a_k})\in P(k)}(-1)^{a_1+\ldots+a_k}\frac{(a_1+\ldots+a_k)!}{a_1!\ldots a_k!}\sigma_1^{a_1}\ldots\sigma_k^{a_k},\notag\\
&=\sum_{(1^{a_1},\ldots,k^{a_k})\in P(k)}\frac{1}{1^{a_1}a_1!\ldots k^{a_k}a_k!}\rho_1^{a_1}\ldots \rho_k^{a_k}.\label{W4th}
\end{align}
\textit{Hint:} For \eqref{Wsec} and \eqref{W4th}, mimic the proof of \autoref{Turan} (these are specializations of \emph{Frobenius' formula}\index{Frobenius' formula} on Schur polynomials).
\end{A}

\begin{A}
Use \autoref{exwaring} to solve the non-linear system
\begin{align*}
x+y+z&=3,\\
x^2+y^2+z^2&=15,\\
x^3+y^3+z^3&=45.
\end{align*}
\textit{Hint:} As the solution is too complicated to guess, look up \emph{Cardano's formula}.\index{Cardano's formula}
\end{A}

We leave polynomials to fully develop multivariate power series.

\begin{Def}
For $\alpha\in K[[X_1,\ldots,X_n]]$ and $1\le i\le n$ let $\partial_i\alpha$\index{*Part@$\partial_i\alpha$} be the $i$-th \emph{partial derivative}\index{partial derivative} with respect to $X_i$, i.\,e. we regard $\alpha$ as a power series in $X_i$ with coefficients in $K[[X_1,\ldots,X_{i-1},X_{i+1},\ldots,X_n]]$ and form the usual (formal) derivative. For $k\in\NN_0$ let $\partial_i^k\alpha$ be the $k$-th derivative with respect to $X_i$. 
\end{Def}

Note that $\partial_i$ is a linear operator, which commutes with all $\partial_j$ (\emph{Schwarz' theorem}).\index{Schwarz' theorem} Indeed, by linearity it suffices to check
\[\partial_i\partial_j(X_i^kX_j^l)=\partial_i(lX_i^kX_j^{l-1})=klX_i^{k-1}X_j^{l-1}=\partial_j(kX_i^{k-1}X_j^l)=\partial_j\partial_i(X_i^kX_j^l).\]
We need a fairly general form of the product rule.

\begin{Lem}[\textsc{Leibniz}' rule]\index{Leibniz' rule}
Let $\alpha_1,\ldots,\alpha_s\in\CC[[X_1,\ldots,X_n]]$ and $k_1,\ldots,k_n\in\NN_0$. Then
\[\boxed{\partial_1^{k_1}\ldots\partial_n^{k_n}(\alpha_1\ldots\alpha_s)=\sum_{l_{11}+\ldots+l_{1s}=k_1}\ldots\sum_{l_{n1}+\ldots+l_{ns}=k_n}\frac{k_1!\ldots k_n!}{\prod_{i,j}l_{ij}!}\prod_{t=1}^s\partial_1^{l_{1t}}\ldots\partial_n^{l_{nt}}\alpha_t.}\]
\end{Lem}
\begin{proof}
For $n=1$ the claim is more or less equivalent to the familiar multinomial theorem
\[(a_1+\ldots+a_s)^k=\sum_{l_1+\ldots+l_s=k}\frac{k!}{l_1!\ldots l_s!}a_1^{l_1}\ldots a_s^{l_s},\]
where $a_1,\ldots,a_s$ lie in any commutative ring. With every new indeterminate we simply apply the case $n=1$ to the formula for $n-1$. In this way the multinomial coefficients are getting multiplied.
\end{proof}
 
Our next goal is the multivariate chain rule for (higher) derivatives.
We equip $\CC[[X_1,\ldots,X_n]]^n$ with the direct product ring structure and use the shorthand notation $\alpha:=(\alpha_1,\ldots,\alpha_n)$ and $0:=(0,\ldots,0)$. Write 
\[\alpha\circ\beta:=\bigl(\alpha_1(\beta_1,\ldots,\beta_n),\ldots,\alpha_n(\beta_1,\ldots,\beta_n)\bigr)\]
provided this is well-defined. 
It is not difficult to show that
\begin{equation}\label{distmult}
\begin{split}
(\alpha+\beta)\circ \gamma=(\alpha\circ\gamma)+(\beta\circ\gamma),\\
(\alpha\cdot\beta)\circ \gamma=(\alpha\circ\gamma)\cdot(\beta\circ\gamma)
\end{split}
\end{equation}
as in \autoref{lemcomp}.
It was remarked by M. Hardy~\cite{MHardy}\index{Hardy} that Leibniz' rule as well as the chain rule become slightly more transparent when we give up on counting multiplicities of derivatives as follows. 

\begin{Thm}[\textsc{Faà di Bruno}'s rule]\label{faa}\index{Faadi@Faà di Bruno's rule}
Let $\alpha,\beta_1,\ldots,\beta_n\in K[[X_1,\ldots,X_n]]$ such that $\alpha(\beta_1,\ldots,\beta_n)$ is defined. Then for $1\le k_1,\ldots,k_s\le n$ we have 
\[\boxed{\partial_{k_1}\ldots\partial_{k_s}\bigl(\alpha(\beta_1,\ldots,\beta_n)\bigr)=\sum_{t=1}^s\sum_{\substack{A_1\dot{\cup}\ldots\dot{\cup} A_t\\=\{1,\ldots,s\}}}\sum_{1\le i_1,\ldots,i_t\le n}(\partial_{A_1}\beta_{i_1})\ldots(\partial_{A_t}\beta_{i_t})(\partial_{i_1}\ldots\partial_{i_t}\alpha)(\beta_1,\ldots,\beta_n),}
\]
where $A_1\dot{\cup}\ldots\dot{\cup} A_t$ runs through the set partitions of $s$ and $\partial_{A_t}:=\prod_{a\in A_t}\partial_{k_a}$.
\end{Thm}
\begin{proof}
By \eqref{distmult}, we may assume that $\alpha=X_1^{a_1}\ldots X_n^{a_n}$. Then by the product rule,
\begin{equation}\label{s1}
\partial_k\bigl(\alpha(\beta_1,\ldots,\beta_n)\bigr)=\sum_{i=1}^n(\partial_k\beta_i)a_i\beta_1^{a_1}\ldots\beta_i^{a_i-1}\ldots\beta_n^{a_n}=\sum_{i=1}^n(\partial_k\beta_i)(\partial_i\alpha)(\beta_1,\ldots,\beta_n).
\end{equation}
This settles the case $s=1$. Now assume that the claim for some $s$ is established. When we apply some $\partial_{k_{s+1}}$ on the right hand side of the induction hypothesis, we need the product rule again. There are two cases: either $s+1$ is added to some of the existing sets $A_t$ or $\partial_{k_{s+1}}$ is applied to $(\partial_{i_1}\ldots\partial_{i_t}\alpha)(\beta_1,\ldots,\beta_n)$. In the latter case $t$ increases to $t+1$, $A_{t+1}=\{t+1\}$ and $i_{t+1}$ is introduced as in \eqref{s1}.
\end{proof}

\begin{Ex}
For $n=1$ and $K=\CC$, \autoref{faa} “simplifies” to
\begin{align*}
(\alpha(\beta))^{(s)}&=\sum_{t=1}^s\sum_{A_1\dot\cup\ldots\dot\cup A_t}\beta^{(|A_1|)}\ldots\beta^{(|A_t|)}\alpha^{(t)}(\beta)\\
&=\sum_{(1^{a_1},\ldots,s^{a_s})\in P(s)}\frac{s!}{(1!)^{a_1}\ldots (s!)^{a_s}a_1!\ldots a_s!}(\beta')^{a_1}\ldots(\beta^{(s)})^{a_s}\alpha^{(a_1+\ldots+a_s)}(\beta),
\end{align*}
where $(1^{a_1},\ldots,s^{a_s})$ runs over the partitions of $s$ and the coefficient is explained just as in \autoref{lemperm}.
\end{Ex}

\section{MacMahon's master theorem}\label{secMMM}

In this final section we enter a non-commutative world by making use of matrices. The ultimate goal is the \emph{master theorem} found and named by MacMahon~\cite[Chapter~II]{MacMahon}. 
Since $K[[X_1,\ldots,X_n]]$ can be embedded in its field of fractions, the familiar rules of linear algebra (over fields) remain valid in the ring $K[[X_1,\ldots,X_n]]^{n\times n}$ of $n\times n$-matrices with coefficients in $K[[X_1,\ldots,X_n]]$. In particular, the determinant of $A=(\alpha_{ij})_{i,j}$ can be defined by \emph{Leibniz' formula}\index{Leibniz' formula} (not rule)
\[\det(A):=\sum_{\sigma\in S_n}\sgn(\sigma)\alpha_{1\sigma(1)}\ldots\alpha_{n\sigma(n)}.\]\index{*DetA@$\det(A)$}
It follows that $\det(A(0))=\det(A)(0)$ by \eqref{distmult}.
Recall that the \emph{adjoint}\index{adjoint matrix} of $A$ is defined by $\adj(A):=\bigl((-1)^{i+j}\det(A_{ji})\bigr)_{i,j}$\index{*Adj@$\adj(A)$} where $A_{ji}$ is obtained from $A$ by deleting the $j$-th row and $i$-th column. Then 
\[A\adj(A)=\adj(A)A=\det(A)1_n,\] 
where $1_n$ denotes the identity $n\times n$-matrix. 
This shows that $A$ is invertible if and only if $\det(A)$ is invertible in $K[[X_1,\ldots,X_n]]$, i.\,e. $\det(A)$ has a non-zero constant term.
Expanding the entries of $A$ as $\alpha_{ij}=\sum a^{(i,j)}_{k_1,\ldots,k_n}X_1^{k_1}\ldots X_n^{k_n}$ gives rise to a natural bijection
\begin{align*}
\Omega\colon K[[X_1,\ldots,X_n]]^{n\times n}&\to K^{n\times n}[[X_1,\ldots,X_n]],\\
A&\mapsto\sum_{k_1,\ldots,k_n}\bigl(a^{(i,j)}_{k_1,\ldots,k_n}\bigr)_{i,j}X_1^{k_1}\ldots X_n^{k_n}.
\end{align*}
Clearly, $\Omega$ is a vector space isomorphism. To verify that it is even a ring isomorphism, it is enough to consider matrices $A$, $B$ with only one non-zero entry each. But then $AB=0$ or $AB$ is just the multiplication in $K[[X_1,\ldots,X_n]]$. 
So we can now freely pass from one ring to the other, keeping in mind that we are dealing with power series with non-commuting coefficients! 
Allowing some flexibility, we can also expand $A=\sum_i A_iX_k^i$ where $k$ is fixed and $A_i\in K[[X_1,\ldots,X_{k-1},X_{k+1},\ldots,X_n]]^{n\times n}$. 
This suggests defining \index{*PartialA@$\partial_kA$}
\[\partial_kA:=\sum_{i=1}^\infty iA_iX_k^{i-1}=(\partial_k\alpha_{ij})_{i,j}.\] 
The sum and product differentiation rules remain correct, but the power rule $\partial_k(A^s)=s\partial_k(A)A^{s-1}$ (and in turn Leibniz' rule) does not hold in general, since $A$ might not commute with $\partial_kA$. 

The next two results are just a warm-up and are not needed later on.

\begin{Lem}\label{lemJac}
Let $A\in\CC[[X_1,\ldots,X_n]]^{n\times n}$ and $1\le k\le n$. Then $\partial_k\det(A)=\tr\bigl(\adj(A)\partial_kA\bigr)$.
\end{Lem}
\begin{proof}
Write $A=(\alpha_{ij})$. By Leibniz' formula and the product rule, it follows that
\begin{align*}
\partial_k\det(A)&=\partial_k\Bigl(\sum_{\sigma\in S_n}\sgn(\sigma)\alpha_{1\sigma(1)}\ldots\alpha_{n\sigma(n)}\Bigr)\\
&=\sum_{i=1}^n\sum_{\sigma\in S_n}\sgn(\sigma)\alpha_{1\sigma(1)}\ldots\partial_k(\alpha_{i\sigma(i)})\ldots\alpha_{n\sigma(n)}\\
&=\sum_{i=1}^n\sum_{j=1}^n\sum_{\substack{\sigma\in S_n\\\sigma(j)=i}}\sgn(\sigma)\alpha_{1\sigma(1)}\ldots\partial_k(\alpha_{ji})\ldots\alpha_{n\sigma(n)}.
\end{align*}
The permutations $\sigma\in S_n$ with $\sigma(j)=i$ correspond naturally to 
\[\tau:=(i,i+1,\ldots,n)^{-1}\sigma(j,j+1,\ldots,n)\in S_{n-1}\]
with $\sgn(\tau)=(-1)^{i+j}\sgn(\sigma)$. 
Hence, Leibniz' formula applied to $\det(A_{ji})$ gives
\[\sum_{j=1}^n\sum_{\substack{\sigma\in S_n\\\sigma(j)=i}}\sgn(\sigma)\alpha_{1\sigma(1)}\ldots\partial_k(\alpha_{ji})\ldots\alpha_{n\sigma(n)}=\sum_{j=1}^n(-1)^{i+j}\det(A_{ji})\partial_k(\alpha_{ji}).\]
Since this is the entry of $\adj(A)\partial_kA$ at position $(i,i)$, the claim follows.
\end{proof}

If $A\in\CC^{n\times n}[[X_1,\ldots,X_n]]$ has zero constant term, then $\exp(A)=\sum_{k=0}^\infty\frac{A^k}{k!}$ converges and is even invertible since it has constant term $1_n$. 

\begin{Thm}[\textsc{Jacobi}'s determinant formula]\index{Jacobi's determinant formula}
Let $A\in\CC^{n\times n}[[X_1,\ldots,X_n]]$ with zero constant term. Then 
\[\boxed{\det(\exp(A))=\exp(\tr(A)).}\]
\end{Thm}
\begin{proof}
We introduce a new variable $Y$ and consider $B:=\exp(AY)$. Denoting the derivative with respect to $Y$ by $'$, we have 
\[B'=\Bigl(\sum_{k=0}^\infty\frac{A^k}{k!}Y^k\Bigr)'=\sum_{k=1}^\infty\frac{A^k}{(k-1)!}Y^{k-1}=AB.\] 
Invoking \autoref{lemJac} and using that $B$ is invertible, we compute:
\begin{align*}
\det(B)'&=\tr(\adj(B)B')=\det(B)\tr(B^{-1}AB)=\det(B)\tr(A).
\end{align*}
This is a differential equation, which can be solved as follows.
Write $\det(B)=\sum_{k=0}^\infty B_kY^k$ with $B_k\in\CC[[X_1,\ldots,X_n]]$. Then $B_0=\det(B(0))=\det(\exp(0)1_n)=\det(1_n)=1$ and $B_{k+1}=\frac{1}{k+1}\tr(A)B_k$ for $k\ge 0$. This yields
\[\det(B)=1+\tr(A)Y+\frac{\tr(A)^2}{2}Y^2+\ldots=\exp(\tr(A)Y).\]
Since we already know that $\exp(A)$ converges, we are allowed to specialize $Y=1$ in $B$, from which the claim follows.
\end{proof}

\begin{Def}
For $\alpha=(\alpha_1,\ldots,\alpha_n)\in K[[X_1,\ldots,X_n]]^n$ we call \index{*Jalpha@$J(\alpha)$}
\[J(\alpha):=(\partial_j\alpha_i)_{i,j}\in K[[X_1,\ldots,X_n]]^{n\times n}\] 
the \emph{Jacobi matrix} of $\alpha$. \index{Jacobi matrix}
\end{Def}

\begin{Ex}
The Jacobi matrix of the power sum polynomials $\rho=(\rho_1,\ldots,\rho_n)$ is a deformed \emph{Vandermonde matrix}\index{Vandermonde matrix} $J(\rho)=(iX_j^{i-1})_{i,j}$ with determinant $n!\prod_{i<j}(X_j-X_i)$. The next theorem furnishes a new proof for the algebraic independence of $\rho_1,\ldots,\rho_n$.
\end{Ex}

\begin{Thm}\label{thmjacobi}
Polynomials $\alpha_1,\ldots,\alpha_n\in \CC[X_1,\ldots,X_n]$ form a transcendence basis of $\CC(X_1,\ldots,X_n)$ if and only if $\det(J(\alpha))\ne 0$.
\end{Thm}
\begin{proof}
The proof follows Humphreys~\cite[Proposition~3.10]{Humphreys}.\index{Humphreys}
Suppose first that $\alpha_1,\ldots,\alpha_n$ are algebraically dependent. Then there exists $\beta\in\CC[X_1,\ldots,X_n]\setminus\CC$ such that $\beta(\alpha_1,\ldots,\alpha_n)=0$ and $\deg(\beta)$ is as small as possible.
By \eqref{s1},
\[\sum_{i=1}^n(\partial_k\alpha_i)(\partial_i\beta)(\alpha_1,\ldots,\alpha_n)=\partial_k(\beta(\alpha_1,\ldots,\alpha_n))=0\]
for $k=1,\ldots,n$. This is a homogeneous linear system over $\CC(X_1,\ldots,X_n)$ with coefficient matrix $J(\alpha)^\TT$ (the transpose of $J(\alpha)$). Since $\beta\notin\CC$, there exists $1\le k\le n$ such that $\partial_k\beta\ne 0$. Now $(\partial_k\beta)(\alpha_1,\ldots,\alpha_n)\ne 0$, because $\deg(\beta)$ was chosen to be minimal. Hence, the linear system has a non-trivial solution and $\det(J(\alpha))$ must be $0$.

Assume conversely that $\alpha_1,\ldots,\alpha_n$ are algebraically independent over $\CC$. Since $\CC(X_1,\ldots,X_n)$ has transcendence degree $n$, the polynomials $X_i,\alpha_1,\ldots,\alpha_n$ are algebraically dependent for each $i=1,\ldots,n$. Let $\beta_i\in\CC[X_0,X_1,\ldots,X_n]\setminus\CC$ such that $\beta_i(X_i,\alpha_1,\ldots,\alpha_n)=0$ and $\deg(\beta_i)$ as small as possible. Again by \eqref{s1}, 
\[\delta_{ik}(\partial_0\beta_i)(X_i,\alpha_1,\ldots,\alpha_n)+\sum_{j=1}^n(\partial_k\alpha_j)(\partial_j\beta_i)(X_i,\alpha_1,\ldots,\alpha_n)=\partial_k(\beta_i(X_i,\alpha_1,\ldots,\alpha_n))=0\]
for $i=1,\ldots,n$. Since $\alpha_1,\ldots,\alpha_n$ are algebraically independent, $X_0$ must occur in every $\beta_i$.
In particular, $\partial_0\beta_i\ne 0$ has smaller degree than $\beta_i$. The choice of $\beta_i$ implies $(\partial_0\beta_i)(X_i,\alpha_1,\ldots,\alpha_n)\ne 0$ for $i=1,\ldots,n$. This leads to the following matrix equation in $\CC[X_1,\ldots,X_n]$:
\[\bigl((\partial_j\beta_i)(X_i,\alpha_1,\ldots,\alpha_n)\bigr)_{i,j}J(\alpha)=-\bigl(\delta_{ij}(\partial_0\beta_i)(X_i,\alpha_1,\ldots,\alpha_n)\bigr)_{i,j}.\]
Since the determinant of the diagonal matrix on the right hand side does not vanish, also $\det(J(\alpha))$ cannot vanish.
\end{proof}

\begin{Def}
Let $C_a\subseteq K[[X_1,\ldots,X_n]]$ be the set of power series with constant term $a\in K$, i.\,e. $\alpha\in C_a\iff\alpha(0)=a$. Let \index{*KXYcirc@$K[[X_1,\ldots,X_n]]^\circ$}
\[K[[X_1,\ldots,X_n]]^\circ:=\bigl\{\alpha\in C_0^n:\det(J(\alpha))\notin C_0\bigr\}\subseteq K[[X_1,\ldots,X_n]]^n.\]
\end{Def}

The condition $\det(J(\alpha))\notin C_0$ means that $J(\alpha)(0)$ is invertible in $K^{n\times n}$. 
For $n=1$ we have $\alpha\in K[[X_1,\ldots,X_n]]^\circ\iff\alpha(0)=0\ne\alpha'(0)\iff\alpha\in(X)\setminus(X^2)$, so our notation is consistent with \autoref{revgroup}. The following is a multivariate analog.

\begin{Thm}[Inverse function theorem]\label{IFT}\index{inverse function theorem}
The set $K[[X_1,\ldots,X_n]]^\circ$ is a group with respect to $\circ$ and 
\[K[[X_1,\ldots,X_n]]^\circ\to\GL(n,K),\qquad \alpha\mapsto J(\alpha)(0)\]
is a group epimorphism.
\end{Thm}
\begin{proof}
Let $\alpha,\beta\in K[[X_1,\ldots,X_n]]^\circ$. Clearly, $\alpha\circ\beta\in C_0^n$. By \eqref{s1},
\[\partial_j(\alpha_i(\beta))=\sum_{k=1}^n(\partial_j\beta_k)(\partial_k\alpha_i)(\beta)\]
and $J(\alpha\circ\beta)=J(\alpha)(\beta)\cdot J(\beta)$. It follows that 
\begin{equation}\label{Jhom}
J(\alpha\circ\beta)(0)=J(\alpha)(0)J(\beta)(0)\in\GL(n,K)
\end{equation}
and $\alpha\circ\beta\in K[[X_1,\ldots,X_n]]^\circ$. The associativity $(\alpha\circ\beta)\circ\gamma=\alpha\circ(\beta\circ\gamma)$ holds more generally when both sides are defined. 
This can be reduced to the easy case where $\alpha=(0,\ldots,0,X_i,0,\ldots,0)$ by fully exploiting \eqref{distmult}. The identity element of $K[[X_1,\ldots,X_n]]^\circ$ is clearly $(X_1,\ldots,X_n)$. 

For the construction of inverse elements, we first assume that $J(\alpha)(0)=1_n$. Here we can adapt the proof of \autoref{sympoly}. 
We sort the $n$-tuples $(k_1,\ldots,k_n)$ first by $\sum_{i=1}^nk_i$ and then lexicographically (for tuples with the same sum). 
Define $\beta_{i,1}:=X_i\in C_0$. For a given $\beta_{i,j}$ let $f(i,j):=(k_1,\ldots,k_n)$ be the minimal tuple such that the coefficient $c$ of $X_1^{k_1}\ldots X_n^{k_n}$ in $\beta_{i,j}(\alpha_1,\ldots,\alpha_n)-X_i$ is non-zero (if there is no such tuple we are done). Now let
\[\beta_{i,j+1}:=\beta_{i,j}-cX_1^{k_1}\ldots X_{n}^{k_n}\in C_0.\]
Since $(\partial_j\alpha_k)(0)=\delta_{kj}$, $X_k$ is the unique monomial of degree $1$ in $\alpha_k$. Consequently, $X_1^{k_1}\ldots X_n^{k_n}$ is the unique lowest degree monomial in $\alpha_1^{k_1}\ldots\alpha_n^{k_n}$. Hence, going from $\beta_{i,j}(\alpha_1,\ldots,\alpha_n)$ to $\beta_{i,j+1}(\alpha_1,\ldots,\alpha_n)$ replaces $X_1^{k_1}\ldots X_n^{k_n}$ with terms of higher degree. Consequently, $f(i,j+1)>f(i,j)$ and $\beta_i:=\lim_{j\to\infty}\beta_{i,j}\in C_0$ exists with $\beta_i(\alpha_1,\ldots,\alpha_n)=X_i$. 

Now we consider the general case. As explained before, $\det(J(\alpha))\notin C_0$ implies that $J(\alpha)$ is invertible. Let $S:=(s_{ij})=J(\alpha)^{-1}(0)\in K^{n\times n}$ and $\tilde{X}_i:=\sum_{j=1}^ns_{ij}X_j\in C_0$ for $i=1,\ldots,n$. Then $\tilde{\alpha}:=\tilde{X}\circ\alpha\in C_0^n$ 
fulfills
\[J(\tilde{\alpha})(0)=(\partial_j\tilde{\alpha}_i)_{i,j}(0)=\Bigl(\sum_{k=1}^ns_{ik}(\partial_j\alpha_k)(0)\Bigr)_{i,j}=SJ(\alpha)(0)=1_n.\]
By the construction above, there exists $\tilde{\beta}\in C_0^n$ with $\tilde{\beta}\circ\tilde{\alpha}=(X_1,\ldots,X_n)$. Define $\beta:=\tilde{\beta}\circ\tilde{X}\in C_0^n$.
Then
\[\beta\circ\alpha=\tilde{\beta}\circ\tilde{X}\circ\alpha=\tilde{\beta}\circ\tilde{\alpha}=(X_1,\ldots,X_n).\]
By \eqref{Jhom}, $J(\beta)(0)J(\alpha)(0)=1_n$ and $J(\beta)(0)=S$. Hence, $\beta\in K[[X_1,\ldots,X_n]]^\circ$ is a left inverse of $\alpha$ with respect to $\circ$. As usual, $\beta$ is also a right inverse (see proof of \autoref{revgroup}). Hence, $K[[X_1,\ldots,X_n]]^\circ$ is a group. The map $\alpha\mapsto J(\alpha)(0)$ is a homomorphism by \eqref{Jhom}. For $A=(a_{ij})\in\GL(n,K)$ let $\alpha_i:=a_{i1}X_1+\ldots+a_{in}X_n$. Then $\alpha\in C_0$ and $J(\alpha)(0)=A$. So our map is surjective.
\end{proof}

If $\alpha_1,\ldots,\alpha_n\in\CC[X_1,\ldots,X_n]$ are polynomials such that $\det(J(\alpha))\in\CC^\times$, the \emph{Jacobi conjecture}\index{Jacobi conjecture} (put forward by Keller~\cite{Keller}\index{Keller} in 1939) claims that there exist polynomials $\beta_1,\dots,\beta_n$ such that $\beta\circ\alpha=(X_1,\ldots,X_n)$. The following counterexample with $n=3$ was found by Anthropic Claude Fable in 2026\footnote{see \href{https://x.com/__alpoge__/status/2079028340955197566}{post on X}}
\[
\alpha=\bigl((1+X Y)^3 Z + Y^2 (1 + X Y) (4 + 3X Y),\ Y + 3X (1 + X Y)^2 Z + 3X Y^2 (4 + 3X Y),\ 2X - 3X^2 Y - X^3 Z\bigr).
\]
It can be checked that $\det(J(\alpha))=-2$ and $\alpha(0, 0, -1/4)=(-1/4, 0, 0)=\alpha(1, -3/2, 13/2)$.
Hence, there can be no beta with $\beta\circ\alpha=(X_1,X_2,X_3)$.
The conjecture is still open for $n=2$ (see \cite{Essen}). 

An explicit formula for the reverse (i.\,e. the inverse with respect to $\circ$) is given by the following multivariate version of \autoref{lagrange}. To simplify the proof (which is still difficult) we restrict ourselves to those $\beta\in\CC[[X_1,\ldots,X_n]]^n$ such that $\beta_i\in X_iC_1\subseteq C_0$. Note that $J(\beta)(0)=1_n$ here.

\begin{Thm}[\textsc{Lagrange--Good}'s inversion formula]\label{lagrange2}\index{Lagrange--Good's inversion formula}
Let $\alpha\in\CC[[X_1,\ldots,X_n]]$ and $\beta_i\in X_iC_1$ for $i=1,\ldots,n$. Then
\begin{equation}\label{betaex}
\alpha=\sum_{k_1,\ldots,k_n\ge 0}c_{k_1,\ldots,k_n}\beta_1^{k_1}\ldots\beta_n^{k_n}
\end{equation}
where $c_{k_1,\ldots,k_n}\in\CC$ is the coefficient of $X_1^{k_1}\ldots X_n^{k_n}$ in 
\[\alpha \Bigl(\frac{X_1}{\beta_1}\Bigr)^{k_1+1}\ldots\Bigl(\frac{X_n}{\beta_n}\Bigr)^{k_n+1}\det(J(\beta)).\]
\end{Thm}
\begin{proof}
The proof is taken from Hofbauer~\cite{Hofbauer}.\index{Hofbauer}
By construction $\beta\in C_0^n$ and $J(\beta)(0)=1_n$. 
By the inverse function theorem, there exists $\gamma\in\CC[[X_1,\ldots,X_n]]^\circ$ such that $\gamma\circ\beta=(X_1,\ldots,X_n)$. Replacing $X_i$ by $\gamma_i(\beta)$ in $\alpha$ yields an expansion in the form \eqref{betaex} where we denote the coefficients by $\bar{c}_{k_1,\ldots,k_n}$ for the moment. Observe that $\tau_i:=X_i/\beta_i\in C_1$ and $\det(J(\beta))\in C_1$. 
For $l_1,\ldots,l_n\ge 0$ we define
\[\rho_{l_1,\ldots,l_n}:=\tau_1^{l_1+1}\ldots\tau_n^{l_n+1}\det(J(\beta))\in C_1.\]
Then $c_{l_1,\ldots,l_n}$ is, by definition, the coefficient of $X_1^{l_1}\ldots X_n^{l_n}$ in $\alpha\rho_{l_1,\ldots,l_n}$. So it also must be the coefficient of $X_1^{l_1}\ldots X_n^{l_n}$ in
\[\sum_{\substack{k_1,\ldots,k_n\ge 0\\\forall i\,:\,k_i\le l_i}}\bar{c}_{k_1,\ldots,k_n}X_1^{k_1}\ldots X_n^{k_n}\rho_{l_1-k_1,\ldots,l_n-k_n}.\]
It is easy to see that $c_{0,\ldots,0}=\alpha(0)=\bar{c}_{0,\ldots,0}$ as claimed. Hence, it suffices to show that $X_1^{k_1}\ldots X_n^{k_n}$ does not occur in $\rho_{k_1,\ldots,k_n}$ for $(k_1,\ldots,k_n)\ne(0,\ldots,0)$. 
By the product rule,
\[\tau_i\partial_j\beta_i=\partial_j(\beta_i\tau_i)-\beta_i\partial_j\tau_i=\delta_{ij}-X_i\frac{\partial_j\tau_i}{\tau_i}.\]
Since the (Jacobi) determinant is linear in every row, it follows that
\[\rho_{k_1,\ldots,k_n}=\det\bigl(\delta_{ij}\tau_i^{k_i}-X_i\tau_i^{k_i-1}\partial_j\tau_i\bigr)=\sum_{\sigma\in S_n}\sgn(\sigma)\prod_{i=1}^n\bigl(\delta_{i\sigma(i)}\tau_i^{k_i}-X_i\tau_i^{k_i-1}\partial_{\sigma(i)}\tau_i\bigr).\]
By the (multivariate) Taylor series, we want to show that $(\partial_1^{k_1}\ldots\partial_n^{k_n}\rho_{k_1,\ldots,k_n})(0)=0$. 

Leibniz' rule applied to the inner product yields
\[P_\sigma:=\sum_{l_{11}+\ldots+l_{1n}=k_1}\ldots\sum_{l_{n1}+\ldots+l_{nn}=k_n}\frac{k_1!\ldots k_n!}{\prod_{i,j}l_{ij}!}\prod_{t=1}^n\partial_1^{l_{1t}}\ldots\partial_n^{l_{nt}}\bigl(\delta_{t\sigma(t)}\tau_t^{k_t}-X_t\tau_t^{k_t-1}\partial_{\sigma(t)}\tau_t\bigr).\]
Therein, we find
\[\bigl(\partial_1^{l_{1t}}\ldots\partial_n^{l_{nt}}(X_t\tau_t^{k_t-1}\partial_{\sigma(t)}\tau_t)\bigr)(0)=l_{tt}\bigl(\partial_1^{l_{1t}}\ldots\partial_t^{l_{tt}-1}\ldots\partial_n^{l_{nt}}(\tau_t^{k_t-1}\partial_{\sigma(t)}\tau_t)\bigr)(0).\]
In particular, the product is zero if $\sigma(t)\ne t$ and $l_{tt}=0$. 
We will disregard this case in the following. If $k_t=0$, then $\partial_t$ does not occur at all, i.\,e. $l_{t1}=\ldots=l_{tn}=0$. 
Hence, $\sigma$ must fix every $t$ outside of $W:=\{i:k_i>0\}\ne\emptyset$. Let $S_W\le S_n$ be the subgroup of those permutations.
For $t\notin W$ the corresponding factor of $P_\sigma(0)$ is $\bigl(\partial_1^{l_{1t}}\ldots\partial_n^{l_{nt}}\tau_t^{k_t}\bigr)(0)=\bigl(\partial_1^{l_{1t}}\ldots\partial_n^{l_{nt}}(1)\bigr)(0)$, which vanishes unless $l_{1t}=\ldots=l_{nt}=0$. 
Since moreover $l_{jt}=0$ for all $t$ whenever $k_j=0$, only the indices $l_{ij}$ with $i,j\in W$ remain.
We set $\mu_i:=\tau_i^{k_i}$ for $i\in W$ and observe that $\frac{1}{k_t}\partial_{\sigma(t)}(\mu_t)=\tau_t^{k_t-1}\partial_{\sigma(t)}\tau_t$. Hence, the inner product of $P_\sigma(0)$ takes the form
\[\prod_{t\in W}\bigl(\delta_{t\sigma(t)}\partial_1^{l_{1t}}\ldots\partial_n^{l_{nt}}\mu_t-\frac{l_{tt}}{k_t}\partial_1^{l_{1t}}\ldots\partial_t^{l_{tt}-1}\ldots\partial_{\sigma(t)}^{l_{\sigma(t)t}+1}\ldots\partial_n^{l_{nt}}\mu_t\bigr).\]
Finally, we transform the indices via $l_{jt}\mapsto m_{jt}:=l_{jt}-\delta_{jt}+\delta_{j\sigma(t)}$ for $j,t\in W$. Since $l_{tt}\ge 1$ whenever $\sigma(t)\ne t$, this is a bijection onto the set of those non-negative integers $(m_{jt})_{j,t\in W}$ such that $\sum_{t\in W}m_{jt}=k_j$ for $j\in W$ and $m_{\sigma(t)t}>0$ whenever $\sigma(t)\ne t$. Note that
\[\frac{l_{tt}}{l_{1t}!\ldots l_{nt}!}=\frac{l_{\sigma(t)t}+1}{l_{1t}!\ldots(l_{tt}-1)!\ldots (l_{\sigma(t)t}+1)!\ldots l_{nt}!}=\frac{m_{\sigma(t)t}}{m_{1t}!\ldots m_{nt}!}.\]
This turns $P_\sigma(0)$ into
\[P_\sigma(0)=\sum_{m_{ij}}\frac{k_1!\ldots k_n!}{\prod_{i,j\in W}m_{ij}!}\prod_{t\in W}\partial_1^{m_{1t}}\ldots\partial_n^{m_{nt}}(\mu_t)(0)\Bigl(\delta_{t\sigma(t)}-\frac{m_{\sigma(t)t}}{k_t}\Bigr),\]
where $(m_{ij})$ runs over all non-negative integers with $\sum_{t\in W}m_{jt}=k_j$ for $j\in W$ (the additional condition $m_{\sigma(t)t}>0$ from above can be dropped, because the factor $\delta_{t\sigma(t)}-m_{\sigma(t)t}/k_t$ vanishes for $m_{\sigma(t)t}=0$ and $\sigma(t)\ne t$). In this way, the range of summation no longer depends on $\sigma$ and since only the last term actually depends on $\sigma$, we conclude
\[(\partial_1^{k_1}\ldots\partial_n^{k_n}\rho_{k_1,\ldots,k_n})(0)=\sum_{m_{ij}}\frac{k_1!\ldots k_n!}{\prod_{i,j\in W}m_{ij}!}\prod_{t\in W}\partial_1^{m_{1t}}\ldots\partial_n^{m_{nt}}(\mu_t)(0)\sum_{\sigma\in S_W}\sgn(\sigma)\prod_{t\in W}\Bigl(\delta_{t\sigma(t)}-\frac{m_{\sigma(t)t}}{k_t}\Bigr).\]
The final sum is the determinant of $(\delta_{ij}-m_{ji}/k_i)_{i,j\in W}$. This matrix is singular, since its rows are linearly dependent as
\[\sum_{i\in W}k_i\Bigl(\delta_{ij}-\frac{m_{ji}}{k_i}\Bigr)=k_j-\sum_{i\in W}m_{ji}=0\]
for every $j\in W$. This completes the proof of $(\partial_1^{k_1}\ldots\partial_n^{k_n}\rho_{k_1,\ldots,k_n})(0)=0$.
\end{proof}

In an attempt to unify and generalize some dual pairs we have already found, we study the following setting. Let $A=(a_{ij})\in\CC^{n\times n}$ and $D=\diag(X_1,\ldots,X_n)$. For $I\subseteq N:=\{1,\ldots,n\}$ let $A_I:=(a_{ij})_{i,j\in I}$ and $X_I=\prod_{i\in I}X_i$. Since the determinant is linear in every row, we obtain
\begin{align*}
\det(1_n+DA)&=\begin{vmatrix}
1&0&\cdots&0\\
a_{21}X_2&1+a_{22}X_2& &a_{2n}X_2\\
\vdots&&\ddots&\vdots\\
a_{n1}X_n&\cdots&\cdots&1+a_{nn}X_n
\end{vmatrix}+\begin{vmatrix}
a_{11}&\cdots&a_{1n}\\
a_{21}X_2&\cdots&a_{2n}X_2\\
\vdots&&\vdots\\
a_{n1}X_n&\cdots&1+a_{nn}X_n
\end{vmatrix}X_1\\
&=\begin{vmatrix}
1+a_{22}X_2&\cdots&a_{2n}X_2\\
\vdots&\ddots&\vdots\\
a_{n2}X_n&\cdots&1+a_{nn}X_n
\end{vmatrix}+\begin{vmatrix}
a_{11}&\cdots&\cdots&a_{1n}\\
0&1&0&0\\
a_{31}X_3&&&a_{3n}X_3\\
\vdots&&&\vdots\\
a_{n1}X_n&\cdots&\cdots&1+a_{nn}X_n
\end{vmatrix}X_1\\
&\quad+\begin{vmatrix}
a_{11}&\cdots&a_{1n}\\
a_{21}&\cdots&a_{2n}\\
a_{31}X_3&\cdots&a_{3n}X_3\\
\vdots&&\vdots\\
a_{n1}X_n&\cdots&1+a_{nn}X_n
\end{vmatrix}X_1X_2=\ldots\\
&=1+\sum_{i=1}^na_{ii}X_i+\sum_{i<j}\det(A_{\{i,j\}})X_iX_j+\ldots+\det(A)X_N.
\end{align*}
Altogether,
\begin{equation}\label{mac1}
\det(1_n+DA)=\sum_{I\subseteq N}\det(A_I)X_I,
\end{equation}
where $\det(A_\varnothing)=1$ for convenience. The dual equation, discovered by Vere-Jones~\cite{Vere-Jones}, uses the \emph{permanent}\index{permanent} $\per(A)=\sum_{\sigma\in S_n}a_{1\sigma(1)}\ldots a_{n\sigma(n)}$\index{*PerA@$\per(A)$} of $A$: 
\begin{equation}\label{a2}
\frac{1}{\det(1_n-DA)}=\sum_{k=0}^\infty\sum_{I\in N^k}\per(A_I)\frac{X_I}{k!},
\end{equation}
where $I$ now runs through all tuples of elements in $N$
(in contrast to the determinant, $\per(A_I)$ does not necessarily vanish if $A_I$ has identical rows). 
We will derive \eqref{a2} in \autoref{vere} from the following result, which seems more amenable to applications.

\begin{Thm}[\textsc{MacMahon}'s master theorem]\label{mahon}\index{MacMahon's master theorem}
Let $A=(a_{ij})\in\CC^{n\times n}$ and $D=\diag(X_1,\ldots,X_n)$. Then
\begin{equation}\label{mac2}
\boxed{\frac{1}{\det(1_n-DA)}=\sum_{k_1,\ldots,k_n\ge 0}c_{k_1,\ldots,k_n}X_1^{k_1}\ldots X_n^{k_n},}
\end{equation}
where $c_{k_1,\ldots,k_n}\in\CC$ is the coefficient of $X_1^{k_1}\ldots X_n^{k_n}$ in 
\[\prod_{i=1}^n(a_{i1}X_1+\ldots+a_{in}X_n)^{k_i}.\]
\end{Thm}
\begin{proof}
Let $A_i:=a_{i1}X_1+\ldots+a_{in}X_n$ and $\beta_i:=X_i(1+A_i)^{-1}\in X_iC_1$ for $i=1,\ldots,n$. Let $D(\beta):=\diag(\beta_1,\ldots,\beta_n)$ and $\alpha:=\det(1_n-D(\beta)A)^{-1}$. Since $\partial_j A_i=a_{ij}$, we obtain
\[\partial_j\beta_i=\frac{\delta_{ij}(1+A_i)-X_ia_{ij}}{(1+A_i)^2}=\frac{\delta_{ij}-\beta_ia_{ij}}{1+A_i}\]
and 
\[\alpha\det(J(\beta))=\prod_{i=1}^n\frac{1}{1+A_i}.\]
Hence, by \autoref{lagrange2}, the coefficient of $\beta_1^{k_1}\ldots\beta_n^{k_n}$ in $\alpha$ is the coefficient of $X_1^{k_1}\ldots X_n^{k_n}$ in 
\[\Bigl(\frac{X_1}{\beta_1}\Bigr)^{k_1+1}\ldots\Bigl(\frac{X_n}{\beta_n}\Bigr)^{k_n+1}\prod_{i=1}^n\frac{1}{1+A_i}=\prod_{i=1}^n(1+a_{i1}X_1+\ldots+a_{in}X_n)^{k_i}.\]
Since the product on the right hand side has degree $k_1+\ldots+k_n$, the additional summand $1$ plays no role and the desired coefficient really is $c_{k_1,\ldots,k_n}$. By \autoref{IFT}, the $X_i$ can be substituted by some $\gamma_i$ such that $\beta_1^{k_1}\ldots\beta_n^{k_n}$ becomes $X_1^{k_1}\ldots X_n^{k_n}$ and $\alpha$ becomes $\det(1_n-DA)^{-1}$.
\end{proof}

A graph-theoretical proof of \autoref{mahon} was given by Foata and is presented in \cite[Section~9.4]{Brualdi}. There is also a short analytic argument which reduces the claim to the easy case where $A$ is a triangular matrix.

\begin{Cor}\label{vere}
Equation \eqref{a2} holds.
\end{Cor}
\begin{proof}
Let $k:=k_1+\ldots+k_n$.
By the multinomial theorem we have
\begin{align*}
\prod_{i=1}^n&(a_{i1}X_1+\ldots+a_{in}X_n)^{k_i}\\
&=\sum_{k_{11}+\ldots+k_{1n}=k_1}\ldots\sum_{k_{n1}+\ldots+k_{nn}=k_n}\frac{k_1!\ldots k_n!}{\prod_{i,j}k_{ij}!}a_{11}^{k_{11}}a_{12}^{k_{12}}\ldots a_{nn}^{k_{nn}}X_1^{k_{11}+\ldots+k_{n1}}\ldots X_n^{k_{1n}+\ldots+k_{nn}}.
\end{align*}
To obtain $c_{k_1,\ldots,k_n}$ one needs to run only over those indices $k_{ij}$ with $\sum_i k_{ij}=k_j$ for $j=1,\ldots,n$.

On the other hand, we need to sum over those tuples $I\in N^k$ in \eqref{a2} which contain $i$ with multiplicity $k_i$ for each $i=1,\ldots,n$. The number of those tuples is $\frac{k!}{k_1!\ldots k_n!}$. The factor $k!$ cancels with $\frac{1}{k!}$ in \eqref{a2}, so that the coefficient in question is $\frac{\per(A_I)}{k_1!\ldots k_n!}$.
Since the permanent is invariant under permutations of rows and columns, we may assume that $I=(1^{k_1},\ldots,n^{k_n})$. Then $A_I$ has the block form $A_I=(A_{ij})_{i,j}$ where 
\[A_{ij}=a_{ij}\begin{pmatrix}
1&\cdots&1\\
\vdots&&\vdots\\
1&\cdots&1
\end{pmatrix}\in\CC^{k_i\times k_j}.\]
In the definition of $\per(A_I)$, every permutation $\sigma\in S_k$ corresponds to a selection of $k$ entries in $A_I$ such that one entry in each row and each column is selected. Suppose that $k_{ij}$ entries in block $A_{ij}$ are selected. Then $\sum_i k_{ij}=k_j$ and $\sum_j k_{ij}=k_i$. To choose the rows in each $A_{ij}$ there are $\frac{k_1!\ldots k_n!}{\prod k_{ij}!}$ possibilities. We get the same number for the selections of columns. Finally, once rows and columns are fixed, there are $\prod k_{ij}!$ choices to permute the entries in each block $A_{ij}$. Now the coefficient of $X_1^{k_1}\ldots X_n^{k_n}$ in \eqref{a2} turns out to be
\[\sum_{\substack{k_{ij}\\\sum_ik_{ij}=k_j\\\sum_jk_{ij}=k_i}}\frac{k_1!\ldots k_n!}{\prod_{i,j} k_{ij}!}a_{11}^{k_{11}}a_{12}^{k_{12}}\ldots a_{nn}^{k_{nn}}=c_{k_1,\ldots,k_n}.\qedhere\]
\end{proof}

We illustrate with some examples why MacMahon called \autoref{mahon} the \emph{master} theorem (as he was a former major, I am tempted to call it the $M^4$-theorem).

\begin{Ex}\hfill
\begin{enumerate}[(i)]
\item The expression $\det(1_n-DA)$ is reminiscent of the definition of the characteristic polynomial $\chi_A=X^n+s_{n-1}X^{n-1}+\ldots+s_0\in\CC[X]$ of $A$. In fact, setting $X:=X_1=\ldots=X_n$ allows us to regard $\det(1_n-XA)$ as a Laurent polynomial in $X$. We can then introduce $X^{-1}$ to obtain
\[\det(1_n-XA)=X^n\det(X^{-1}1_n-A)=X^n\chi_A(X^{-1})=1+s_{n-1}X+\ldots+s_0X^n.\]
Now \eqref{mac1} in combination with Vieta's theorem yields
\[\sum_{\substack{I\subseteq N\\|I|=k}}\det(A_I)=(-1)^ks_{n-k}=\sigma_k(\lambda_1,\ldots,\lambda_n),\]
where $\lambda_1,\ldots,\lambda_n\in\CC$ are the eigenvalues of $A$. This extends the familiar identities $\det(A)=\lambda_1\ldots\lambda_n$ and $\tr(A)=\lambda_1+\ldots+\lambda_n$. With the help of \autoref{exwaring}, one can also express $s_k$ in terms of $\rho_l(\lambda_1,\ldots,\lambda_n)=\tr(A^l)$. 

\item If $A=1_n$ and $X_1=\ldots=X_n=X$, then \eqref{mac1} and \eqref{mac2} become 
\begin{align*}
(1+X)^n&=\sum_{I\subseteq N}X^{|I|}=\sum_{k=0}^n\binom{n}{k}X^k,\\
(1-X)^{-n}&=\sum_{k_1,\ldots,k_n\ge 0}X^{k_1+\ldots+k_n}=\sum_{k=0}^\infty\binom{n+k-1}{k}X^k,
\end{align*}
since the $k$-element multisets correspond to the tuples $(k_1,\ldots,k_n)$ with $k_1+\ldots+k_n=k$ where $k_i$ encodes the multiplicity of $i$. 

\item Taking $A=1_n$ and $X_k=X^k$ in \eqref{mac2} recovers an equation from \autoref{partseries}:
\[\prod_{k=1}^n\frac{1}{1-X^k}=\sum_{k_1,\ldots,k_n\ge0}X^{k_1+2k_2+\ldots+nk_n}=\sum_{k=0}^\infty p_n(k)X^k.\]
Similarly, choosing $X_k=kX$ or $X_k=X_kY$ leads more or less directly to \autoref{genstir} and \autoref{vieta} respectively.

\item Take $(X_1,X_2,X_3)=(X,Y,Z)$ and
\[A=\begin{pmatrix}
0&1&-1\\
-1&0&1\\
1&-1&0
\end{pmatrix}\]
in \eqref{mac2}. Then by \emph{Sarrus' rule},\index{Sarrus' rule}
\begin{align*}
\frac{1}{\det(1_3-DA)}&=\frac{1}{1+XZ+YZ+XY}=\sum_{k=0}^\infty(-1)^k(XY+YZ+ZX)^k\\
&=\sum_{k=0}^\infty(-1)^k\sum_{a+b+c=k}\frac{k!}{a!b!c!}X^{a+c}Y^{a+b}Z^{b+c}.
\end{align*}
The coefficient of $(XYZ)^{2n}$ is easily seen to be $(-1)^n\frac{(3n)!}{(n!)^3}$. On the other hand, the same coefficient in
\[(Y-Z)^{2n}(Z-X)^{2n}(X-Y)^{2n}=\sum_{a,b,c\ge 0}\binom{2n}{a}\binom{2n}{b}\binom{2n}{c}(-1)^{a+b+c}X^{c-b+2n}Y^{a-c+2n}Z^{b-a+2n}\]
occurs for $a=b=c$. This yields \emph{Dixon's identity}:\index{Dixon's identity}
\[(-1)^n\frac{(3n)!}{(n!)^3}=\sum_{k=0}^{2n}(-1)^k\binom{2n}{k}^3.\]
\end{enumerate}
\end{Ex}

We end with a short outlook. There are at least three ways to define power series over an infinite set of indeterminates $\{X_i:i\in I\}$.
The first option is \index{*KZZ1@$K[[X_i:i\in I]]_1$}
\[K[[X_i:i\in I]]_1:=\bigcup_{\substack{J\subseteq I\\|J|<\infty}}K[[X_j:j\in J]].\]
This ring inherits many properties from the finite version. Perhaps more interesting is the completion of the polynomial ring $K[X_i:i\in I]\subseteq K[[X_i:i\in I]]_1$. Its elements are of the form $\sum_{d=0}^\infty\alpha_d$, where $\alpha_d$ is a homogeneous polynomial of degree $d$. 
Finally, one can define power series as arbitrary sums of monomials, each involving only finitely many indeterminates. If $I=\NN$, a monomial $X_1^{a_1}\ldots X_k^{a_k}$ can be identified with the integer $p_1^{a_1}\ldots p_k^{a_k}$ where $p_1,\ldots,p_k$ are the first prime numbers. Then power series are just mappings $\NN\to K$ and the product becomes the \emph{Dirichlet convolution}\index{Dirichlet convolution}
\[(f\cdot g)(n)=\sum_{d\mid n}f(d)g(n/d)\]
for $f,g\colon\NN\to K$. 

Moreover, power series in non-commuting indeterminates exist and form what is sometimes called the \emph{Magnus ring}\index{Magnus ring} $K\langle\langle X_1,\ldots,X_n\rangle\rangle$ (the polynomial version is the \emph{free algebra}\index{free algebra} $K\langle X_1,\ldots,X_n\rangle$).
The Lie bracket $[a,b]:=ab-ba$ turns $K\langle\langle X_1,\ldots,X_n\rangle\rangle$\index{*Kommab@$[a,b]$} into a \emph{Lie algebra}\index{Lie algebra} and fulfills \emph{Jacobi's identity}\index{Jacobi's identity}
\[[a,[b,c]]+[b,[c,a]]+[c,[a,b]]=0.\]
The functional equation for $\exp(X)$ is replaced by the \emph{Baker--Campbell--Hausdorff formula}\index{Baker--Campbell--Hausdorff formula} in this context.

The reader might ask about formal Laurent series in multiple indeterminates. Although the field of fractions $Q(K[[X_1,\ldots,X_n]])$ certainly exists, its elements do not look like one might expect. For example, the inverse of $X-Y$ could be
$\sum_{k=1}^\infty X^{-k}Y^{k-1}$ or $-\sum_{k=1}^\infty X^{k-1}Y^{-k}$. 
The first series lies in $K((X))((Y))$, but not in $K((Y))((X))$. For the second series it is the other way around.

\addsec{Appendix: Algebraic properties}
\setcounter{section}{1}
\renewcommand{\thesection}{\Alph{section}}

In this appendix we state and prove a number of interesting algebraic properties of the rings of polynomials, power series and Laurent series. The proofs are often quite technical, but the results are independent of the preceding text.

In the following $R$ will always denote a commutative ring with $1$. We regard the integers $\ZZ$ as elements of $R$ via the map $n\mapsto n\cdot 1$ (not always injective). It is interesting to note that even basic facts become false in this generality. Consider the ring $R=\ZZ/4\ZZ$ of integers modulo $4$, for instance. Here $\alpha=2$ is not invertible in $R((X))$. On the other hand, $\alpha=X+2$ \emph{is} invertible as
\[\alpha^2\cdot X^{-2}=X^2X^{-2}=1.\]
Moreover, \autoref{leminv}(ii) fails for $R$ as $(1+2X)^2=1$. 

Now we impose further conditions on the ring $R$. 
Recall that $R$ is called \emph{noetherian}\index{noetherian} if the following equivalent statements hold:
\begin{itemize}
\item Every ideal of $R$ is finitely generated.
\item Every non-empty set of ideals of $R$ contains a maximal element.
\item Every chain of ideals $I_1\subseteq I_2\subseteq\ldots$ of $R$ stabilizes, i.\,e. $I_k=I_{k+1}=\ldots$ for some $k\in\NN$.
\end{itemize}

We start with a classical result.

\begin{Thm}[\textsc{Hilbert}'s basis theorem]\index{Hilbert's basis theorem}
If $R$ is noetherian, so is $R[X]$. In particular, $K[X_1,\ldots,X_n]$ is noetherian for every field $K$. 
\end{Thm}
\begin{proof}
Suppose by way of contradiction that $I\unlhd R[X]$ is not finitely generated. Let $\alpha_0:=0\in R[X]$. For $k\in\NN$ choose inductively $\alpha_k\in I\setminus(\alpha_0,\ldots,\alpha_{k-1})$ of minimal degree $d_k$. Then $0\le d_1\le d_2\le\ldots$. Let $a_k\in R$ be the leading coefficient of $\alpha_k$ for $k\in\NN$. By hypothesis, the chain $(a_1)\subseteq (a_1,a_2)\subseteq\ldots$ stabilizes. In particular, there exists some $k\in\NN$ such that $a_k=\sum_{i=1}^{k-1}r_ia_i$ for some $r_1,\ldots,r_{k-1}\in R$. But now
\[\beta:=\alpha_k-\sum_{i=1}^{k-1}{r_iX^{d_k-d_i}\alpha_i}\in I\setminus(\alpha_0,\ldots,\alpha_{k-1})\]
has degree $<d_k$ contradicting the choice of $\alpha_k$.
The second claim follows by induction on $n$ since $K$ is noetherian. 
\end{proof}

A slightly more involved argument yields the corresponding theorem of power series. In complex analysis, this is sometimes called \emph{Rückert's basis theorem}.\index{Rückert's basis theorem}

\begin{Thm}\label{thmKXnoether}
If $R$ is noetherian, so is $R[[X]]$. In particular, $K[[X_1,\ldots,X_n]]$ is noetherian for every field $K$. 
\end{Thm}
\begin{proof}
We follow Lang~\cite[Theorem~IV.9.4]{Lang}.\index{Lang}
Let $I\unlhd R[[X]]$. For $i\in\NN_0$, let
\[J_i:=\bigl\{a\in R:\exists\alpha\in I:\alpha\equiv aX^i\pmod{X^{i+1}}\bigr\}\subseteq R.\]
It is easy to see that $J_i\unlhd R$ and $J_0\subseteq J_1\subseteq\ldots$. Since $R$ noetherian, there exists $n\in\NN$ with $J_k=J_n$ for all $k\ge n$. Moreover, there exist $a_{ij}\in R$ such that $J_i=(a_{i1},\ldots,a_{i,k_i})$ for $i=0,\ldots,n$. 
We choose $\alpha_{ij}\in I$ with $\alpha_{ij}\equiv a_{ij}X^i\pmod{X^{i+1}}$ and show that $I$ is generated by the $\alpha_{ij}$.
To this end, let $\alpha\in I\setminus\{0\}$ with $\alpha\equiv rX^d\pmod{X^{d+1}}$ and $0\ne r\in J_d$. If $d\le n$, then there exist $r_1,\ldots,r_{k_d}\in R$ such that $r=r_1a_{d1}+\ldots+r_{k_d}a_{d,k_d}$ and
\[\alpha\equiv r_1\alpha_{d1}+\ldots+r_{k_d}\alpha_{d,k_d}\pmod{X^{d+1}}.\]
By replacing $\alpha$ with $\alpha-r_1\alpha_{d1}-\ldots-r_{k_d}\alpha_{d,k_d}$, $d$ increases. After finitely many replacements we may assume that $d_0:=d>n$. By the same argument, there exist $r_{0,1},\ldots,r_{0,k_n}\in R$ such that
\[ \alpha_1:=\alpha-(r_{0,1}\alpha_{n1}+\ldots +r_{0,k_n}\alpha_{n,k_n})X^{d-n}\equiv 0\pmod{X^{d+1}}.\]
Let $d_1:=\inf\alpha_1>d$. Then there exist $r_{1,1},\ldots,r_{1,k_n}\in R$ with
\[ \alpha_2:=\alpha_1-(r_{1,1}\alpha_{n1}+\ldots +r_{1,k_n}\alpha_{n,k_n})X^{d_1-n}\equiv 0\pmod{X^{d_1+1}}.\]
Repeating this process leads to power series $\beta_i:=\sum_{j=0}^\infty r_{ji}X^{d_j-n}$ for $i=1,\ldots,k_n$. Finally,
\[\alpha=\beta_1\alpha_{n1}+\ldots+\beta_{k_n}\alpha_{n,k_n}\in(\alpha_{n,1},\ldots,\alpha_{n,k_n}).\qedhere\]
\end{proof}

Now we focus on integral domains $R$, i.\,e. $ab\ne 0$ for all $a,b\in R\setminus\{0\}$. 
If $R$ is an integral domain, so are $R[X_1,\ldots,X_n]$ and $R[[X_1,\ldots,X_n]]$ (see proof of \autoref{intdom}). 
A integral domain $R$ is called a \emph{principal ideal domain}\index{principal ideal domain} (PID)\index{PID} if every ideal of $R$ is generated by a single element. Of course, every PID is noetherian.

\begin{Thm}
For every field $K$, the rings $K[X]$ and $K[[X]]$ are PIDs.
\end{Thm}
\begin{proof}
Let $(0)\ne I\unlhd K[X]$ and choose $\alpha\in I\setminus\{0\}$ of minimal degree $d\ge 0$. For every $\beta\in I$ there exist $\gamma,\delta\in K[X]$ such that $\beta=\alpha\gamma+\delta$ and $\deg\delta<d$ by euclidean division. Since $\delta=\beta-\alpha\gamma\in I$, it follows that $\delta=0$ and $\beta\in(\alpha)$. Hence, $I=(\alpha)$.

Now let $(0)\ne I\unlhd K[[X]]$ and choose $\alpha\in I\setminus\{0\}$ such that $d:=\inf\alpha$ is minimal. Then $X^d=(\alpha X^{-d})^{-1}\alpha\in I$. It is easy to see that $I=(X^d)$. 
\end{proof}

The proof above show further that $K[[X]]$ is a complete discrete valuation ring with unique maximal ideal $(X)$ (\autoref{vollständig}). 
Hilbert's basis theorem does not carry over to PIDs. For instance, neither $\ZZ[X]$ nor $K[X,Y]$ are PIDs (consider the ideals $(2,X)$ and $(X,Y)$ respectively). We mention that $K[[X]]$ is not artinian since $(X)\supsetneq(X^2)\supsetneq\ldots$.

\begin{Def}
Let $R$ be an integral domain and $a,b\in R$. We write $a\mid b$ if there exists $c\in R$ such that $ac=b$. 
An element $r\in R\setminus(R^\times\cup\{0\})$ is called
\begin{itemize}
\item \emph{irreducible}\index{irreducible element} if $r=ab$ implies $a\in R^\times$ or $b\in R^\times$.
\item \emph{prime element}\index{prime element} if $r\mid ab$ implies $r\mid a$ or $r\mid b$.
\end{itemize}
We call $R$ a \emph{unique factorization domain}\index{unique factorization domain}\index{UFD} (UFD) if every element of $R\setminus(R^\times\cup\{0\})$ is a product of prime elements. 
\end{Def}

Recall (or prove) that $a,b\in R$ are called \emph{associated}\index{associated elements} whenever the following equivalent assertions hold:
\begin{itemize}
\item $a\mid b\mid a$.
\item $\exists u\in R^\times:au=b$.
\item $(a)=(b)$. 
\end{itemize} 

Note that association defines an equivalence relation on $R$. 
Let $\Pi$ be a set of representatives for the prime elements up to association (for instance, the positive prime numbers in $\ZZ$ or the monic irreducible polynomials in $K[X]$). In a UFD every non-zero element can be written in the form 
\[r=\epsilon \pi_1^{a_1}\ldots \pi_n^{a_n},\]
where $\epsilon\in R^\times$, $\pi_1,\ldots,\pi_n\in\Pi$ are pairwise distinct, and $a_1,\ldots,a_n\in\NN$. It follows from the definition of prime elements that this decomposition is unique up to the order of its factors (this explains the U in UFD). 

Our goal is to show that the rings of polynomials and power series over a field are UFDs.

\begin{Lem}\label{lemUFD}
Let $R$ be an integral domain.
\begin{enumerate}[(i)]
\item Every prime element of $R$ is irreducible.
\item\label{noether} If $R$ is noetherian, then every element of $R\setminus(R^\times\cup\{0\})$ is a product of irreducible elements.
\item Every PID is a UFD. 
\end{enumerate}
\end{Lem}
\begin{proof}\hfill
\begin{enumerate}[(i)]
\item Let $p\in R$ be a prime element and $p=ab$ with $a,b\in R$. Then $p\mid ab$ and without loss of generality, $p\mid a$. 
Since also $a\mid p$, it follows that $p$ is associated to $a$ and therefore $b\in R^\times$ as required.

\item Suppose that $x_1\in R\setminus(R^\times\cup\{0\})$ is not a product of irreducible elements. Then there exist $x_2,y\in R\setminus R^\times$ with $x_1=x_2y$, where $x_2$ is not a product of irreducible elements. Since $y\notin R^\times$ we have $(x_1)\subsetneq(x_2)$. Repeating the same argument with $x_2$ yields $x_3\in R\setminus R^\times$ such that $(x_2)\subsetneq(x_3)$ and so on. But then $R$ cannot be noetherian.

\item Let $R$ be a PID. By \eqref{noether}, it suffices to show that every irreducible element $r\in R$ is a prime element. Let $a,b\in R$ with $r\mid ab$. Since $R$ is a PID, there exists $c\in R$ with $(a,r)=(c)$. It follows that $r=cd$ for some $d\in R$. Since $r$ is irreducible, $c$ or $d$ must be a unit. In the latter case, $a\in (a,r)=(c)=(r)$ and $r\mid a$ as wanted. Hence, we may assume that $(a,r)=(c)=R$ and similarly, $(b,r)=R$. But this yields the contradiction $R=(a,r)(b,r)=(Ra+Rr)(Rb+Rr)\subseteq Rab+Rr=(r)$. \qedhere
\end{enumerate}
\end{proof}

\autoref{lemUFD} implies that the PIDs $K[X]$ and $K[[X]]$ are UFDs. It is much more difficult to handle $K[X_1,\ldots,X_n]$ and $K[[X_1,\ldots,X_n]]$ as those are not PIDs (for $n\ge 2$).

\begin{Def}
Let $R$ be an integral domain. A \emph{common divisor}\index{common divisor} of $a_1,\ldots,a_n\in R$ is an element $d\in R$ such that $d\mid a_i$ for $i=1,\ldots,n$. We call $d$ a \emph{greatest common divisor}\index{common divisor!greatest} (gcd) if $e\mid d$ for every common divisor $e$ of $a_1,\ldots,a_n$. Clearly, a gcd is unique up to association. If a gcd is a unit, then $a_1,\ldots,a_n$ are called \emph{coprime}.\index{coprime elements}
A polynomial $\alpha\in R[X]$ is called \emph{primitive}\index{polynomial!primitive} if its coefficients are coprime.
\end{Def}

Using the unique factorization in a UFD $R$, it is easy to show that every finite set of elements of $R$ has a gcd. 
In $\ZZ$ or $K[X]$ a gcd can be computed efficiently with the euclidean algorithm. However, not every UFD provides such an algorithm, i.\,e. there are non-euclidean UFDs like $\ZZ[X]$. 

\begin{Lem}\label{lemprimitive}
Let $R$ be a UFD with field of fractions $K$.
\begin{enumerate}[(i)]
\item\label{primitive1} $\alpha,\beta\in R[X]$ are primitive if and only if $\alpha\beta$ is primitive. 
\item\label{primitive2} Every $\alpha\in K[X]$ can be written in the form $\alpha=q\tilde\alpha$ with $q\in K$ and $\tilde\alpha\in R[X]$ primitive.
\item\label{primitive3} If $\alpha,\beta\in R[X]$ are primitive and $\alpha\mid\beta$ in $K[X]$, then $\alpha\mid\beta$ holds in $R[X]$ as well.
\item\label{primitive4} If $\alpha\in R[X]\setminus R$ is irreducible, then $\alpha$ is also irreducible in $K[X]$. 
\end{enumerate}
\end{Lem}
\begin{proof}\hfill
\begin{enumerate}[(i)]
\item It is clear that $\alpha\beta$ can be primitive only when $\alpha$ and $\beta$ are primitive. Suppose conversely that $\alpha\beta$ is not primitive. Since $R$ is a UFD, there exists a prime element $p\in R$, which divides the coefficients of $\alpha\beta$. The reduction modulo $p$ yields $\overline{\alpha\beta}=0$ in $\overline{R}[X]$ where $\overline{R}:=R/(p)$. Since $p$ is a prime element, $\overline{R}$ and $\overline{R}[X]$ are integral domains. Hence, we may assume that $\overline{\alpha}=0$. But this means that the coefficients of $\alpha$ are divisible by $p$ and therefore $\alpha$ is not primitive. 

\item If $\alpha=0$, then the claim holds with $q=0$ and $\tilde\alpha=1$. Thus, let $\alpha\ne 0$. Let $b\in R$ be a common non-zero multiple of the denominators of the coefficients of $\alpha$. Then $b\alpha\in R[X]$. Let $c\in R$ be a gcd of the coefficients of $b\alpha$. Then $q:=\frac{c}{b}\in K$ and $\tilde\alpha:=q^{-1}\alpha$ is primitive. 

\item Let $\gamma\in K[X]$ such that $\alpha\gamma=\beta$. By \eqref{primitive2}, there exists $q\in K$ such that $q\gamma\in R[X]$ is primitive. By \eqref{primitive1}, $q\alpha\gamma=q\beta\in R[X]$ is primitive. Since $\beta$ is already primitive, this implies that $q\in R^\times$ and $\gamma\in R[X]$. Therefore $\alpha\mid\beta$ holds in $R[X]$. 

\item As an irreducible element, $\alpha$ must be primitive. Suppose that $\alpha=\beta\gamma$ with $\beta,\gamma\in K[X]\setminus K$. 
By \eqref{primitive2}, there exist primitive polynomials $\tilde{\beta},\tilde{\gamma}\in R[X]$ and $b,c\in K$ such that $\beta=b\tilde{\beta}$ and $\gamma=c\tilde{\gamma}$.  
By \eqref{primitive1}, $\tilde{\beta}\tilde{\gamma}$ is primitive and $\alpha=bc\tilde{\beta}\tilde{\gamma}$. As before, we derive $bc\in R^\times$. It follows that $\tilde{\beta}$ or $\tilde{\gamma}$ lies in $R[X]^\times=R^\times\subseteq K$. Contradiction. 
\qedhere 
\end{enumerate}
\end{proof}

\begin{Thm}[\textsc{Gauss}]\label{gaussall}\index{Gauss}
If $R$ is a UFD, so is $R[X]$. In particular, $K[X_1,\ldots,X_n]$ is a UFD for every field $K$. 
\end{Thm}
\begin{proof}
Let $\alpha\in R[X]\setminus (R^\times\cup\{0\})$. We may write $\alpha=q\tilde\alpha$ with $q\in R$ and $\tilde\alpha\in R[X]$ primitive. Since $R$ is a UFD, $q$ is a product of irreducible elements in $R$, which of course remain irreducible in $R[X]$. Thus, we may assume that $\alpha=\tilde\alpha$ is primitive and not a unit. If $\alpha$ is not irreducible, it can be written as a product of primitive polynomials of smaller degree. Since this can be done only a finite number of times, $\alpha$ must be a product of irreducible elements. 

It remains to show that every irreducible $\alpha\in R[X]$ is a prime element. Thus, let $\alpha\mid\beta\gamma$ for some $\beta,\gamma\in R[X]\setminus\{0\}$. Write $\beta=b\tilde{\beta}$ and $\gamma=c\tilde{\gamma}$ with $b,c\in R$ and $\tilde{\beta},\tilde{\gamma}\in R[X]$ primitive.
If $\alpha\in R$, then $\alpha$ is irreducible in $R$ and therefore a prime element of $R$, because $R$ is a UFD. Since $\tilde{\beta}\tilde{\gamma}$ is primitive by \autoref{lemprimitive}, we have $\alpha\mid bc$ and without loss of generality, $\alpha\mid b$. This shows $\alpha\mid b\tilde{\beta}=\beta$. Now let $\alpha\notin R$. By \autoref{lemprimitive}\eqref{primitive4}, $\alpha$ is irreducible in $K[X]$, where $K$ is the field of fractions of $R$. Since $K[X]$ is a PID (and thus a UFD), $\alpha$ is a prime element of $K[X]$. Since $\alpha$ cannot divide the constant $bc$, we have $\alpha\mid\tilde\beta\tilde\gamma$ and without loss of generality, $\alpha\mid\tilde\beta$ in $K[X]$. By \autoref{lemprimitive}\eqref{primitive3}, $\alpha\mid\tilde\beta\mid\beta$ also holds in $R[X]$.
\end{proof}

Gauss' theorem also implies that the polynomial ring in infinitely many indeterminates over a field is a UFD since every factorization involves only finitely many indeterminates. This furnishes an example of a non-noetherian UFD. 
It should be noted that there is no known efficient algorithm to compute a factorization into prime elements. For instance, any algorithm for $\ZZ[X]$ would also contain an algorithm for the prime decomposition in $\ZZ$. Similarly, a (finite) factorization algorithm for $\CC[X]$ would lead to explicit formulas to compute roots of polynomials (by Galois theory, there cannot be such formulas by radicals alone when the polynomial degree is at least $5$).

Surprisingly, there exist UFDs $R$ such that $R[[X]]$ is not a UFD. A family of examples was constructed by Samuel~\cite{SamuelLecture}\index{Samuel} with 
\[R=\QQ[X,Y,Z]/(X^2-Y^5-Z^7)\] 
being a special case. Nevertheless, we show that $K[[X_1,\ldots,X_n]]$ is a UFD provided $K$ is a field. This requires some preparations. The first lemma is a key reduction in \emph{Noether's normalization theorem}.\index{Noether's normalization theorem}

\begin{Lem}\label{lemnormal}
Let $0\ne\alpha\in R:=K[[X_1,\ldots,X_n]]$. Then there exists a ring automorphism $\Gamma\colon R\to R$ such that $\Gamma(\alpha)(0,\ldots,0,X_n)\ne 0$. 
\end{Lem}
\begin{proof}
Let $X_1^{a_1}\ldots X_n^{a_n}$ be a monomial of $\alpha$ (with non-zero coefficient) such that the tuple $(a_1,\ldots,a_n)$ is minimal with respect to the lexicographical ordering. Let 
\[d:=\max\{a_1,\ldots,a_n\}+1.\] 
Let $\Gamma$ be the unique endomorphism of the polynomial ring $K[X_1,\ldots,X_n]$ defined by $\Gamma(X_i):=X_i+X_n^{d^{n-i}}$ for $1\le i<n$ and $\Gamma(X_n):=X_n$. Obviously, the map $X_i\mapsto X_i-X_n^{d^{n-i}}$ (for $i<n$) defines the inverse of $\Gamma$, i.\,e. $\Gamma$ is an automorphism and $\inf\Gamma(\beta)=\inf\beta$ for all polynomials $\beta$. For $\beta\in R$ there exists a sequence $(\beta_i)_i$ in $K[X_1,\ldots,X_n]$ with $\beta=\lim_{i\to\infty}\beta_i$. Since two such sequences only differ by a null sequence, the assignment $\Gamma(\beta):=\lim_{i\to\infty}\Gamma(\beta_i)$ is well-defined. In this way, $\Gamma$ extends to an automorphism of $R$, which we also call $\Gamma$.

Now let $X_1^{b_1}\ldots X_n^{b_n}\ne X_1^{a_1}\ldots X_n^{a_n}$ be another monomial of $\alpha$ (with non-zero coefficient). Then there exists some $k$ such that $a_i=b_i$ for $1\le i<k$ and $a_k<b_k$. We compute
\begin{align*}
\Gamma(X_1^{a_1}\ldots X_n^{a_n})(0,\ldots,0,X_n)&=X_n^{a_1d^{n-1}+\ldots+a_n},\\
\Gamma(X_1^{b_1}\ldots X_n^{b_n})(0,\ldots,0,X_n)&=X_n^{b_1d^{n-1}+\ldots+b_n},
\end{align*}
where
\begin{align*}
b_1d^{n-1}+\ldots+b_n-(a_1d^{n-1}+\ldots+a_n)&=b_kd^{n-k}+\ldots+b_n-(a_kd^{n-k}+\ldots+a_n)\\
&\ge d^{n-k}-(d-1)(d^{n-k-1}+\ldots+1)=1>0.
\end{align*}
This shows that $\Gamma(\alpha)(0,\ldots,0,X_n)\ne 0$. 
\end{proof}

The following lemma provides some sort of euclidean division. 

\begin{Lem}\label{lemdiv}
Let $R:=K[[X_1,\ldots,X_n]]$ and $\alpha\in K[[X_1,\ldots,X_n,Y]]=R[[Y]]$ with 
\[\alpha_0:=\alpha(0,\ldots,0,Y)\ne 0.\] 
Then for every $\beta\in R[[Y]]$ there exist uniquely determined elements $\rho\in R[[Y]]$ and $\delta\in R[Y]$ such that $\beta=\alpha\rho+\delta$ and $\deg\delta<\inf\alpha_0$. 
\end{Lem}
\begin{proof}
The proof is adapted from Lang~\cite[Theorem~IV.9.1]{Lang},\index{Lang} who in turn attributes it to Manin~\cite{Manin}.\index{Manin}
By definition, $\alpha_0\in K[[Y]]$. Let $d:=\inf\alpha_0$.
We consider the linear maps $\Gamma_1,\Gamma_2\colon R[[Y]]\to R[[Y]]$ defined by 
\begin{align*}
\Gamma_1\Bigl(\sum_{k=0}^\infty b_kY^k\Bigr):=\sum_{k=0}^{d-1}b_kY^k,&&\Gamma_2\Bigl(\sum_{k=0}^\infty b_kY^k\Bigr):=\sum_{k=d}^\infty b_kY^{k-d}.
\end{align*}
Then $\alpha_2:=\Gamma_2(\alpha)$ is invertible and every monomial of $\alpha_1:=\Gamma_1(\alpha)$ involves some $X_i$. This yields a linear map
\[\Gamma\colon R[[Y]]\to R[[Y]],\qquad \gamma\mapsto\Gamma_2(\alpha_1\alpha_2^{-1}\gamma)\]
with $\lim_{k\to\infty}\Gamma^k(\gamma)=0$, because the repeated multiplication with $\alpha_1$ increases the exponent of some $X_i$. Hence, we can define
\[\rho:=\alpha_2^{-1}\sum_{k=0}^\infty(-1)^k\Gamma^k(\Gamma_2(\beta))\in R[[Y]]\]
(the reader may have noticed a similarity to the proof of Banach's fixed point theorem).
Since $\alpha=\alpha_1+\alpha_2Y^d$, we have
\[\Gamma_2(\alpha\rho)=\Gamma_2(\alpha_1\rho)+\Gamma_2(\alpha_2\rho Y^d)=\Gamma(\alpha_2\rho)+\alpha_2\rho=\sum_{k=0}^\infty(-1)^k\Gamma^{k+1}(\Gamma_2(\beta))+\alpha_2\rho=\Gamma_2(\beta).\]
It follows that $\delta:=\beta-\alpha\rho\in R[Y]$ with $\deg\delta<d$. 

To prove the uniqueness of $\rho$ and $\delta$, let $\beta=\alpha\tilde{\rho}+\tilde{\delta}$ with $\tilde{\rho}\in R[[Y]]$,  $\tilde{\delta}\in R[Y]$ and $\deg\tilde\delta<d$. Then,
\[
\Gamma_2(\beta)=\Gamma_2(\alpha\tilde\rho)=\Gamma_2(\alpha_1\tilde\rho)+\Gamma_2(\alpha_2\tilde\rho Y^d)=\Gamma(\alpha_2\tilde\rho)+\alpha_2\tilde\rho
\]
and
\[\alpha_2\tilde\rho=\sum_{k=0}^\infty(-1)^k\Gamma^k(\alpha_2\tilde\rho)+\sum_{k=0}^\infty(-1)^k\Gamma^{k+1}(\alpha_2\tilde\rho)=\sum_{k=0}^\infty(-1)^k\Gamma^k(\Gamma_2(\beta))=\alpha_2\rho.\]
Thus, $\tilde\rho=\rho$ and $\tilde\delta=\beta-\alpha\tilde\rho=\beta-\alpha\rho=\delta$.
\end{proof}

The following theorem replaces \autoref{lemprimitive}\eqref{primitive2}. It also plays a significant role in complex analysis.

\begin{Thm}[\textsc{Weierstrass} preparation]\label{Wpreparation}\index{Weierstrass preparation}
Let $R:=K[[X_1,\ldots,X_n]]$ and $\alpha\in R[[Y]]$ with \linebreak $\alpha(0,\ldots,0,Y)\ne 0$. Then $\alpha$ is associated to a unique monic polynomial $\gamma\in R[Y]$ with $\gamma(0,\ldots,0,Y)=Y^{\deg\gamma}$.  
\end{Thm}
\begin{proof}
Let $d:=\inf\alpha(0,\ldots,0,Y)$. By \autoref{lemdiv}, there exist uniquely determined $\rho\in R[[Y]]$ and $\delta\in R[Y]$ with $Y^d=\alpha\rho+\delta$ and $\deg\delta<d$. A comparison of coefficients shows that $\rho\in R[[Y]]^\times$ and $\delta(0,\ldots,0,Y)=0$. For $\sigma:=\rho^{-1}$ and $\gamma:=Y^d-\delta$ it follows that $\sigma\gamma=\alpha$ and $\gamma(0,\ldots,0,Y)=Y^d$.

To show uniqueness, let $\alpha=\tilde\sigma\tilde\gamma$ with $\tilde\sigma\in R[[Y]]^\times$ and $\tilde\gamma\in R[Y]$ monic with $\tilde{d}:=\deg\tilde\gamma$. Comparing coefficients of $\alpha(0,\ldots,0,Y)$ implies $\tilde{d}=d$. 
Let $\tilde\rho:=\tilde\sigma^{-1}$ and $\tilde\delta:=Y^d-\tilde\gamma$. Then $Y^d=\alpha\tilde\rho+\tilde\delta$ with $\deg\tilde\delta<d$. Now the claim follows from the uniqueness of $\rho$ and $\delta$.
\end{proof}

\begin{Def}
In the situation of \autoref{Wpreparation}, we call $\gamma$ the \emph{Weierstrass polynomial}\index{Weierstrass polynomial} of $\alpha$. 
\end{Def}

\begin{Ex}
If $n=0$, the Weierstrass polynomial of $\alpha\ne 0$ is just $Y^{\inf\alpha}$. Now let
\[\alpha:=Y^3+Y^2+X\in\CC[[X,Y]]\]
with $\alpha(0,Y)=Y^2+Y^3\ne 0$. Since $\inf\alpha(0,Y)=2$, the Weierstrass polynomial of $\alpha$ has the form $\gamma=Y^2+\beta Y+\delta$ with $\beta,\delta\in (X)\subseteq\CC[[X]]$. Division with remainder by the monic polynomial $\gamma$ expresses $\alpha$ in the form $\alpha=\pi\gamma+\rho$ with $\pi,\rho\in\CC[[X]][Y]$ and $\deg\rho<2$. Since also $\alpha=\sigma\gamma$ for some $\sigma\in\CC[[X,Y]]^\times$, the uniqueness in \autoref{lemdiv} yields $\rho=0$ and $\sigma=\pi$. A comparison of the leading coefficients shows $\sigma=Y+\epsilon$ with $\epsilon\in\CC[[X]]$. Now $\sigma\gamma=\alpha$ translates to
\begin{align}
\beta+\epsilon&=1,\label{eq1}\\
\delta+\beta\epsilon&=0,\label{eq2}\\
\delta\epsilon&=X.\label{eq3}
\end{align}
Substituting $\epsilon=1-\beta$ from \eqref{eq1} into \eqref{eq2}, we obtain $\delta=\beta(\beta-1)$ and \eqref{eq3} becomes $-\beta(\beta-1)^2=X$. For $\tau:=-\beta\in (X)$ this reads
\[\tau(1+\tau)^2=X.\]
As $(1+\tau)^2\in\CC[[X]]^\times$, it follows that $\inf\tau=1$, i.\,e. $\tau\in\CC[[X]]^\circ$ is the reverse of $X(1+X)^2$. Now the Lagrange--Bürmann inversion formula (\autoref{lagrange}) yields
\[\tau=\sum_{k=1}^\infty\frac{\res\bigl(X^{-k}(1+X)^{-2k}\bigr)}{k}X^k\overset{\eqref{newtoneq}}{=}\sum_{k=1}^\infty\binom{-2k}{k-1}\frac{X^k}{k}=\sum_{k=1}^\infty(-1)^{k-1}\binom{3k-2}{k-1}\frac{X^k}{k}.\]
Hence, the Weierstrass polynomial of $\alpha$ is
\begin{align*}
\gamma&=Y^2-\tau Y+\tau(1+\tau)\\
&=Y^2-(X-2X^2+7X^3\mp\ldots)Y+X-X^2+3X^3\mp\ldots
\end{align*}
The conclusion of this calculation is that the Weierstrass polynomial can hardly be guessed by looking at $\alpha$. In particular, $\gamma\notin\CC[X,Y]$ although $\alpha$ is a polynomial.
\end{Ex}

Weierstrass polynomials play the role of primitive polynomials in the proof of Gauss' theorem.

\begin{Thm}
For every field $K$, the ring $K[[X_1,\ldots,X_n]]$ is a UFD.
\end{Thm}
\begin{proof}
Let $R_n:=K[[X_1,\ldots,X_n]]$. We argue by induction on $n$. If $n=1$, then $R_1$ is a PID and a UFD. Thus, let $n\ge 2$. 
Since $R_n$ is noetherian by \autoref{thmKXnoether}, every non-zero element of $R_n$ is a product of irreducible elements (or a unit) by \autoref{lemUFD} (this follows more directly from $\inf(\alpha\beta)=\inf(\alpha)+\inf(\beta)$). 

It remains to show that every irreducible element $\alpha\in R_n$ is a prime element. Thus, let $\beta,\gamma\in R_n\setminus\{0\}$ with $\alpha\mid\beta\gamma$. By \autoref{lemnormal}, there exists an automorphism $\Gamma\colon R_n\to R_n$ such that \[\Gamma(\alpha\beta\gamma)(0,\ldots,0,X_n)\ne 0.\] 
Hence, we may assume that $\alpha(0,\ldots,0,X_n)$, $\beta(0,\ldots,0,X_n)$ and $\gamma(0,\ldots,0,X_n)$ do not vanish. By induction, $R_{n-1}$ is a UFD and so is $R_{n-1}[X_n]$ by Gauss' theorem.
Let $\alpha_1,\beta_1,\gamma_1\in R_{n-1}[X_n]$ be the Weierstrass polynomials of $\alpha$, $\beta$ and $\gamma$ respectively. 
With $\alpha$ also $\alpha_1$ is irreducible and $\alpha_1\mid(\beta\gamma)_1=\beta_1\gamma_1$. Since $\alpha_1$ is a prime element in the UFD $R_{n-1}[X_n]$, it follows that $\alpha_1\mid\beta_1$ without loss of generality. Consequently, $\alpha\mid\beta$ and $\alpha$ is a prime element of $R_n$. 
\end{proof}

It has been shown in \cite{Samuel} that $R[[X_1,\ldots,X_n]]$ is a UFD for every PID $R$. In this situation, also the three different rings of power series in infinitely many indeterminates introduced at the end of \autoref{secMMM} are UFDs. This was shown in \cite{Nishimura2,Cashwell,Deckard}. 

Our final objective is the construction of the algebraic closure of the ring $\CC((X))$ of complex Laurent series.
We need a well-known tool. 

\begin{Lem}[\textsc{Hensel}]\label{hensel}\index{Hensel}
Let $R:=K[[X]]$. For a polynomial $\alpha=\sum_{k=0}^na_kY^k\in R[Y]$ let 
\[\bar\alpha:=\sum a_k(0)Y^k\in K[Y].\] 
Let $\alpha\in R[Y]$ be monic such that $\bar\alpha=\alpha_1\alpha_2$ for some coprime monic polynomials $\alpha_1,\alpha_2\in K[Y]\setminus K$. Then there exist uniquely determined monic polynomials $\beta,\gamma\in R[Y]$ such that $\bar\beta=\alpha_1$, $\bar\gamma=\alpha_2$ and $\alpha=\beta\gamma$.
\end{Lem}
\begin{proof}
By hypothesis, $n:=\deg(\alpha)=\deg(\alpha_1)+\deg(\alpha_2)\ge 2$. 
Observe that $\bar{\alpha}$ is essentially the reduction of $\alpha$ modulo the ideal $(X)$. In particular, the map $R[Y]\to K[Y]$, $\alpha\mapsto\bar\alpha$ is a ring homomorphism. For $\sigma,\tau\in R[Y]$ and $k\in\NN$ we write more generally $\sigma\equiv\tau\pmod{(X^k)}$ if all coefficients of $\sigma-\tau$ lie in $(X^k)$. 
First choose any monic polynomials $\beta_1,\gamma_1\in R[Y]$ with $\bar{\beta_1}=\alpha_1$ and $\bar{\gamma_1}=\alpha_2$. Then $\deg(\beta_1)=\deg(\alpha_1)$, $\deg(\gamma_1)=\deg(\alpha_2)$ and $\alpha\equiv\beta_1\gamma_1\pmod{(X)}$.
We construct inductively monic $\beta_k,\gamma_k\in R[Y]$ for $k\ge 2$ such that 
\begin{enumerate}[(a)]
\item\label{amod} $\beta_k\equiv \beta_{k+1}$ and $\gamma_k\equiv \gamma_{k+1}\pmod{(X^k)}$,
\item\label{bmod} $\alpha\equiv\beta_k\gamma_k\pmod{(X^k)}$.
\end{enumerate}
Suppose that $\beta_k,\gamma_k$ are given. Choose $\delta\in R[Y]$ such that $\alpha=\beta_k\gamma_k+X^k\delta$ and $\deg(\delta)<n$. 
Since $\alpha_1,\alpha_2$ are coprime in the euclidean integral domain $K[Y]$, there exist $\sigma,\tau\in R[Y]$ such that $\bar\beta_k\bar\sigma+\bar\gamma_k\bar\tau=\alpha_1\bar\sigma+\alpha_2\bar\tau=1$ by Bézout's lemma.
Since $\beta_k$ is monic, we can perform euclidean division by $\beta_k$ without leaving $R[Y]$. This yields $\rho,\nu\in R[Y]$ such that $\tau\delta=\beta_k\rho+\nu$ and $\deg(\nu)<\deg(\beta_k)$.
Let $d:=\deg(\gamma_1)$ and write $\sigma\delta+\gamma_k\rho=\mu+\eta Y^d$ with $\deg(\mu)<d$. 
Then
\begin{align*}
\beta_{k+1}:=\beta_k+X^k\nu,&&\gamma_{k+1}:=\gamma_k+X^k\mu
\end{align*}
are monic and satisfy \eqref{amod}.
Moreover,
\[\delta\equiv(\beta_k\sigma+\gamma_k\tau)\delta\equiv \beta_k(\sigma\delta+\gamma_k\rho)+\gamma_k\nu\equiv \beta_k\mu+\beta_k\eta Y^d+\gamma_k\nu\pmod{(X)}.\]
Since the degrees of $\delta$, $\beta_k\mu$ and $\gamma_k\nu$ are all smaller than $n$ and $\deg(\beta_k\eta Y^d)\ge n$, it follows that $\bar\eta=0$.
Therefore,
\[
\beta_{k+1}\gamma_{k+1}\equiv\alpha-X^k\delta+(\beta_k\mu+\gamma_k\nu)X^k\equiv\alpha\pmod{(X^{k+1})},
\]
i.\,e. \eqref{bmod} holds for $k+1$. This completes the induction. 

Let $\beta_k=\sum_{j=0}^eb_{kj}Y^j$ and $\gamma_k=\sum_{j=0}^dc_{kj}Y^j$ with $b_{ij},c_{ij}\in R$. By construction, $|b_{kj}-b_{k+1,j}|\le 2^{-k}$ and similarly for $c_{kj}$. Consequently, $b_j:=\lim_kb_{kj}$ and $c_j:=\lim_kc_{kj}$ converge in $R$. We can now define 
\begin{align*}
\beta:=\sum_{j=0}^e b_jY^j,&&\gamma:=\sum_{j=0}^d c_jY^j.
\end{align*} 
Then $\bar\beta=\bar\beta_1=\alpha_1$ and $\bar\gamma=\bar\gamma_1=\alpha_2$. Since $\beta\gamma\equiv\beta_k\gamma_k\equiv\alpha\pmod{(X^k)}$ for every $k\ge 1$, it follows that 
$\alpha=\beta\gamma$. 

To show the uniqueness, let $\pi$ be a prime divisor of $\alpha$ in the UFD $R[Y]$. Since $\alpha$ is monic, the leading coefficient of $\pi$ is a unit of $R$ and we may assume that $\pi$ is monic as well. Then $\bar\pi$ is monic of degree $\deg(\pi)\ge 1$ and $\bar\pi\mid\bar\alpha=\alpha_1\alpha_2$. As $\gcd(\alpha_1,\alpha_2)=1$, we may write $\bar\pi=\bar\pi_1\bar\pi_2$ where $\bar\pi_i:=\gcd(\bar\pi,\alpha_i)$. If $\bar\pi_1\ne 1\ne\bar\pi_2$, the first part of the proof yields monic $\pi_1,\pi_2\in R[Y]$ such that $\pi=\pi_1\pi_2$, contradicting the irreducibility of $\pi$. 
Hence, we may assume that $\bar\pi_2=1$, i.\,e. $\bar\pi=\bar\pi_1\mid\alpha_1$.
Since $\pi$ is prime and $\pi\mid\alpha=\beta\gamma$, it divides $\beta$ or $\gamma$. In the latter case $\bar\pi\mid\bar\gamma=\alpha_2$ and therefore $\bar\pi\mid\gcd(\alpha_1,\alpha_2)=1$, which is impossible. It follows that $\pi\mid\beta$ and $\pi\nmid\gamma$. Since $\beta$ and $\gamma$ are monic, this uniquely determines their prime factorization.
\end{proof}

\begin{Ex}
Let $n\ge2$, $a\in (X)\subseteq R:=\CC[[X]]$ and $\alpha=Y^n-1-a\in R[Y]$. Then $\bar\alpha=Y^n-1=\alpha_1\alpha_2$ with coprime monic $\alpha_1=Y-1$ and $\alpha_2=Y^{n-1}+\ldots+Y+1$. By Hensel's lemma there exist monic $\beta,\gamma\in R[Y]$ such that $\bar\beta=\alpha_1$, $\bar\gamma=\alpha_2$ and $\alpha=\beta\gamma$. We may write $\beta=Y-1-b$ for some $b\in (X)$. Then $(1+b)^n=1+a$ and the remark after \autoref{defpower} implies $1+b=\sqrt[n]{1+a}$. 
The constructive procedure in the proof above inevitably leads to Newton's binomial theorem $1+b=\sum_{k=0}^\infty\binom{1/n}{k}a^k$.
\end{Ex}

We have seen that invertible power series in $\CC[[X]]$ have arbitrary roots. On the other hand, $X$ does not even have a square root in $\CC((X))$. This motivates to raise $X$ not only to negative powers, but also to fractional powers. 

\begin{Def}
A \emph{Puiseux series} over $K$ is defined by\index{Puiseux series}
\[\alpha=\sum_{k=m}^\infty a_{\frac{k}{n}}X^{\frac{k}{n}},\]
where $m\in\ZZ$, $n\in\NN$ and $a_{\frac{k}{n}}\in K$ for $k\ge m$. As usual, let $\inf\alpha:=\frac{m}{n}$ (assuming $a_{\frac{k}{n}}\ne 0$).
The set of Puiseux series is denoted by $K\{\{X\}\}$.\index{*KXXZ@$K\{\{X\}\}$} For $\alpha,\beta\in K\{\{X\}\}$ there exists $n\in\NN$ such that $\tilde\alpha:=\alpha(X^n)$ and $\tilde\beta:=\beta(X^n)$ lie in $K((X))$. 
We carry over the field operations from $K((X))$ via
\[\alpha+\beta:=(\tilde\alpha+\tilde\beta)(X^{\frac{1}{n}}),\qquad\alpha\cdot\beta:=(\tilde\alpha\tilde\beta)(X^{\frac{1}{n}}).\]
\end{Def}

It is straight-forward to check that $(K\{\{X\}\},+,\cdot)$ is a field. 
At this point we have established the following inclusions:
\[K\subseteq K[X]\subseteq K[[X]]\subseteq K((X))\subseteq K\{\{X\}\}.\] 

\begin{Thm}[\textsc{Puiseux}]\index{Puiseux}
The algebraic closure of $\CC((X))$ is $\CC\{\{X\}\}$. 
\end{Thm}
\begin{proof}
We follow Nowak~\cite{Nowak}. \index{Nowak}
Set $R:=\CC[[X]]$, $F:=\CC((X))$ and $\hat F:=\CC\{\{X\}\}$. We show first that $\hat F$ is an algebraic field extension of $F$.
Let $\alpha\in\hat F$ be arbitrary and $n\in\NN$ such that $\beta:=\alpha(X^n)\in F$. Let $\zeta\in\CC$ be a primitive $n$-th root of unity. Define
\[\Gamma:=\prod_{i=1}^n\bigl(Y-\beta(\zeta^iX)\bigr)=Y^n+\gamma_1Y^{n-1}+\ldots+\gamma_n\in F[Y].\]
Replacing $X$ by $\zeta X$ permutes the factors $Y-\beta(\zeta^iX)$ and thus leaves $\Gamma$ invariant. Consequently, $\gamma_i(\zeta X)=\gamma_i$ for $i=1,\ldots,n$. This means that there exist $\tilde\gamma_i\in F$ such that $\gamma_i=\tilde\gamma_i(X^n)$. Now let
\[\tilde\Gamma:=Y^n+\tilde\gamma_1Y^{n-1}+\ldots+\tilde\gamma_n\in F[Y].\]
Substituting $X$ by $X^n$ in $\tilde\Gamma(\alpha)$ gives $\Gamma(\beta)=0$. Thus, also $\tilde\Gamma(\alpha)=0$. This shows that $\alpha$ is algebraic over $F$ and $\hat F$ is an algebraic extension of $F$. 

Now we prove that $\hat F$ is algebraically closed.
Let $\Gamma=Y^n+\gamma_1Y^{n-1}+\ldots+\gamma_n\in \hat F[Y]$ be arbitrary with $n\ge 2$. We need to show that $\Gamma$ has a root in $\hat F$. After applying the \emph{Tschirnhaus transformation}\index{Tschirnhaus transformation} $Y\mapsto Y-\frac{1}{n}\gamma_1$, we may assume that $\gamma_1=0$. Without loss of generality, $\Gamma\ne Y^n$. Let 
\[r:=\min\Bigl\{\frac{1}{k}\inf(\gamma_k):k=1,\ldots,n\Bigr\}\in\QQ\]
and $m\in\NN$ such that $\gamma_k(X^m)\in F$ for $k=1,\ldots,n$ and $r=\frac{s}{m}$ for some $s\in\ZZ$. 
Define $\delta_0:=1$ and $\delta_k:=\gamma_k(X^m)X^{-ks}\in F$ for $k=1,\ldots,n$. Since
\[\inf(\delta_k)=m\inf(\gamma_k)-ks=m(\inf(\gamma_k)-kr)\ge0,\] 
$\Delta:=Y^n+\delta_2Y^{n-2}+\ldots+\delta_n\in R[Y]$. 
Consider $\bar\Delta:=Y^n+\delta_2(0)Y^{n-2}+\ldots+\delta_n(0)\in\CC[Y]$. 
Since $\inf(\delta_k)=0$ for at least one $k\ge 1$, we have $\bar\Delta\ne Y^n$.
Since $\delta_1=0$, also $\bar\Delta\ne (Y-c)^n$ for all $c\in\CC$. Using that $\CC$ is algebraically closed, we can decompose $\bar\Delta=\bar\Delta_1\bar\Delta_2$ with coprime monic polynomials $\bar\Delta_1,\bar\Delta_2\in\CC[Y]\setminus\CC$ of degree $<n$. By Hensel's lemma, there exists a corresponding factorization $\Delta=\Delta_1\Delta_2$ with $\Delta_1,\Delta_2\in R[Y]$. 
Finally, replace $X$ by $X^{\frac{1}{m}}$ in $\Delta_i$ to obtain $\Gamma_i\in\hat F[Y]$. 
Then
\[\Gamma=X^{nr}\sum_{k=0}^n\gamma_kX^{-kr}(YX^{-r})^{n-k}=X^{nr}\sum_{k=0}^n\delta_k(X^{\frac{1}{m}})(YX^{-r})^{n-k}=X^{nr}\Gamma_1(YX^{-r})\Gamma_2(YX^{-r})\]
(where $\gamma_0:=1$).
Induction on $n$ shows that $\Gamma$ has a root and $\hat F$ is algebraically closed.
\end{proof}

For other ring-theoretical properties of power series we refer to the survey~\cite{Sankaran}.

\addsec{Acknowledgment}
I thank Miguel Adamus, Kian Izaddoustdar, Diego García Lucas and Till Müller for spotting some typos and Alexander Zimmermann for proofreading. After the paper had appeared in \emph{Jahresbericht der DMV} 125, Wolfgang Hensgen kindly pointed out an unjustified argument in the proof of Jacobi's triple product. This was settled by moving the section about Laurent series before \autoref{secmain}. Now Jacobi's triple product is obtained more generally for Laurent series. On this occasion, I have added some more generating functions and introduced the appendix. In March 2024, I received a long and detailed list of corrections and valuable suggestions from Darij Grinberg. In 2026, I corrected further typos found by various Claude LLMs. Also, the recently found counterexample to Jacobi's conjecture was added.
The work is supported by the German Research Foundation (\mbox{SA 2864/1-2} and \mbox{SA 2864/3-1}).

\phantomsection
\addcontentsline{toc}{section}{References}
\begin{small}

\end{small}

\begin{theindex}
\begin{small}
\textbf{ Symbols}\nopagebreak
  \item $\lvert\alpha\rvert$, \hyperpage{6}
  \item $(\alpha)$, \hyperpage{5}
  \item $\alpha(\beta)$, \hyperpage{8}
  \item $(1+\alpha)^c$, \hyperpage{13}
  \item $\alpha\circ\beta$, \hyperpage{8}
  \item $\alpha^{\circ n}$, \hyperpage{10}
  \item $\alpha'$, \hyperpage{11}
  \item $\alpha^{(n)}$, \hyperpage{11}
  \item $\adj(A)$, \hyperpage{49}
  \item $\arcsin (X)$, \hyperpage{14}
  \item $\arctan (X)$, \hyperpage{14}
  \item $\binom{c}{k}$, \hyperpage{16}
  \item $b(n)$, \hyperpage{33}
  \item $b_n$, \hyperpage{36}
  \item $c_n$, \hyperpage{26}
  \item $\cos (X)$, \hyperpage{14}
  \item $d(\alpha,\beta)$, \hyperpage{6}
  \item $\deg(\alpha)$, \hyperpage{3}
  \item $\det(A)$, \hyperpage{49}
  \item $d_n$, \hyperpage{26}
  \item $\exp(X)$, \hyperpage{4}
  \item $f_n$, \hyperpage{25}
  \item $\gauss{n}{k}$, \hyperpage{17}
  \item $g(n)$, \hyperpage{41}
  \item $H^k(\alpha)$, \hyperpage{12}
  \item $\inf(\alpha)$, \hyperpage{3}
  \item $\inv(\sigma)$, \hyperpage{38}
  \item $J(\alpha)$, \hyperpage{50}
  \item $K[X]$, \hyperpage{3}
  \item $K(X)$, \hyperpage{15}
  \item $K[[X]]$, \hyperpage{3}
  \item $K((X))$, \hyperpage{15}
  \item $K\{\{X\}\}$, \hyperpage{66}
  \item $K[[X]]^\circ$, \hyperpage{9}
  \item $K[[X]]^\times$, \hyperpage{4}
  \item $K[X_1,\ldots,X_n]$, \hyperpage{42}
  \item $K[[X_1,\ldots,X_n]]$, \hyperpage{42}
  \item $K[[X_1,\ldots,X_n]]^\circ$, \hyperpage{51}
  \item $K[X,X^{-1}]$, \hyperpage{15}
  \item $K[[X_i:i\in I]]_1$, \hyperpage{57}
  \item $[a,b]$, \hyperpage{58}
  \item $\log(1+X)$, \hyperpage{13}
  \item $N_p$, \hyperpage{10}
  \item $\partial_i\alpha$, \hyperpage{47}
  \item $\partial_kA$, \hyperpage{49}
  \item $\per(A)$, \hyperpage{55}
  \item $p_k(n)$, \hyperpage{27}
  \item $p(n)$, \hyperpage{27}
  \item $q(n)$, \hyperpage{39}
  \item $\res(\alpha)$, \hyperpage{15}
  \item $\rho(n,k)$, \hyperpage{38}
  \item $\rho_k$, \hyperpage{43}
  \item $\sigma_k$, \hyperpage{43}
  \item $\sin (X)$, \hyperpage{14}
  \item $\sinh (X)$, \hyperpage{14}
  \item $S_n$, \hyperpage{35}
  \item $\stir{n}{k}$, \hyperpage{35}
  \item $\stirr{n}{k}$, \hyperpage{33}
  \item $\tan (X)$, \hyperpage{14}
  \item $\tau_k$, \hyperpage{43}
  \item $X^n$!, \hyperpage{17}

  \indexspace
\textbf{ A}\nopagebreak
  \item adjoint matrix, \hyperpage{49}
  \item Ahlgren, \hyperpage{31}
  \item Andrews--Eriksson, \hyperpage{30}
  \item Apéry's constant, \hyperpage{39}
  \item associated elements, \hyperpage{60}

  \indexspace
\textbf{ B}\nopagebreak
  \item Baker--Campbell--Hausdorff formula, \hyperpage{58}
  \item Bell number, \hyperpage{33}
  \item Bernoulli numbers, \hyperpage{36}
  \item Binet formula, \hyperpage{25}
  \item Bressoud, \hyperpage{23}

  \indexspace
\textbf{ C}\nopagebreak
  \item Cardano's formula, \hyperpage{47}
  \item Catalan numbers, \hyperpage{26}
  \item Cauchy, \hyperpage{22}
  \item Cauchy product, \hyperpage{3}
  \item Cayley, \hyperpage{28}
  \item chain rule, \hyperpage{12}
  \item Clausen, \hyperpage{26}
  \item coefficient, \hyperpage{3}
    \subitem constant term, \hyperpage{3}
    \subitem leading, \hyperpage{3}
  \item common divisor, \hyperpage{61}
    \subitem greatest, \hyperpage{61}
  \item constant term, \hyperpage{3}
  \item convolution, \hyperpage{3}
  \item coprime elements, \hyperpage{61}
  \item cycle type, \hyperpage{36}

  \indexspace
\textbf{ D}\nopagebreak
  \item degree, \hyperpage{3}
    \subitem of multivariant polynomial, \hyperpage{42}
  \item derivative, \hyperpage{11}
  \item Dirichlet convolution, \hyperpage{58}
  \item Dixon's identity, \hyperpage{57}

  \indexspace
\textbf{ E}\nopagebreak
  \item Erd{\H{o}}s--Tur{\'a}n, \hyperpage{37}
  \item Euler, \hyperpage{28}
  \item Euler's formula, \hyperpage{14}
  \item exponential series, \hyperpage{4}

  \indexspace
\textbf{ F}\nopagebreak
  \item Faà di Bruno's rule, \hyperpage{48}
  \item factor rule, \hyperpage{12}
  \item Faulhaber, \hyperpage{38}
  \item Ferrers diagram, \hyperpage{30}
  \item Fibonacci numbers, \hyperpage{25}
  \item field of fractions, \hyperpage{14}
  \item Fine, \hyperpage{30}
  \item free algebra, \hyperpage{58}
  \item Frobenius' formula, \hyperpage{47}
  \item Fubini's theorem, \hyperpage{7}
  \item functional equation
    \subitem for exponential series, \hyperpage{10}
    \subitem for logarithm, \hyperpage{13}

  \indexspace
\textbf{ G}\nopagebreak
  \item Gauss, \hyperpage{22}, \hyperpage{62}
  \item Gauss' binomial theorem, \hyperpage{18}
  \item Gaussian coefficient, \hyperpage{17}
  \item generating function, \hyperpage{25}
  \item geometric series, \hyperpage{5}
  \item Girard--Newton identities, \hyperpage{44}
  \item Glaisher, \hyperpage{28}
  \item grading, \hyperpage{42}

  \indexspace
\textbf{ H}\nopagebreak
  \item Hamilton's quaternion, \hyperpage{41}
  \item Hardy, \hyperpage{25}, \hyperpage{48}
  \item Hasse derivative, \hyperpage{11}
  \item Hensel, \hyperpage{65}
  \item Hilbert's basis theorem, \hyperpage{58}
  \item Hirschhorn, \hyperpage{22}, \hyperpage{31}, \hyperpage{40}
  \item Hofbauer, \hyperpage{53}
  \item Humphreys, \hyperpage{51}

  \indexspace
\textbf{ I}\nopagebreak
  \item indeterminate, \hyperpage{3}
  \item integral, \hyperpage{11}
  \item inverse function theorem, \hyperpage{51}
  \item inversion, \hyperpage{38}
  \item irreducible element, \hyperpage{60}

  \indexspace
\textbf{ J}\nopagebreak
  \item Jacobi, \hyperpage{22}
  \item Jacobi conjecture, \hyperpage{52}
  \item Jacobi matrix, \hyperpage{50}
  \item Jacobi's determinant formula, \hyperpage{50}
  \item Jacobi's identity, \hyperpage{58}
  \item Jacobi's triple product identity, \hyperpage{20}
  \item Johnson, \hyperpage{25}

  \indexspace
\textbf{ K}\nopagebreak
  \item Keller, \hyperpage{52}

  \indexspace
\textbf{ L}\nopagebreak
  \item L'Hôpital's rule, \hyperpage{12}
  \item Lagrange--Bürmann's inversion formula, \hyperpage{16}
  \item Lagrange--Good's inversion formula, \hyperpage{52}
  \item Lagrange--Jacobi, \hyperpage{39}
  \item Lambert, \hyperpage{26}
  \item Lang, \hyperpage{59}, \hyperpage{63}
  \item Laurent polynomial, \hyperpage{15}
  \item Laurent series, \hyperpage{14}
  \item Leedham-Green, \hyperpage{10}
  \item Legendre, \hyperpage{30}
  \item Leibniz' formula, \hyperpage{49}
  \item Leibniz' rule, \hyperpage{12}, \hyperpage{47}
  \item Lie algebra, \hyperpage{58}
  \item logarithm, \hyperpage{13}

  \indexspace
\textbf{ M}\nopagebreak
  \item Maclaurin series, \hyperpage{11}
  \item MacMahon, \hyperpage{30}, \hyperpage{33}
  \item MacMahon's master theorem, \hyperpage{55}
  \item Magnus ring, \hyperpage{58}
  \item Manin, \hyperpage{63}
  \item Marivani, \hyperpage{31}
  \item Mercator series, \hyperpage{13}
  \item Möbius transformation, \hyperpage{10}

  \indexspace
\textbf{ N}\nopagebreak
  \item Newton's binomial theorem, \hyperpage{16}
  \item Nicomachus' identity, \hyperpage{39}
  \item Noether's normalization theorem, \hyperpage{62}
  \item noetherian, \hyperpage{58}
  \item norm, \hyperpage{6}
  \item Nottingham group, \hyperpage{10}
  \item Nowak, \hyperpage{67}
  \item null sequence, \hyperpage{7}

  \indexspace
\textbf{ O}\nopagebreak
  \item Ono, \hyperpage{31}

  \indexspace
\textbf{ P}\nopagebreak
  \item partial derivative, \hyperpage{47}
  \item partial fraction decomposition, \hyperpage{6}
  \item partition, \hyperpage{27}
    \subitem of sets, \hyperpage{33}
  \item pentagonal number theorem, \hyperpage{21}
  \item permanent, \hyperpage{55}
  \item PID, \hyperpage{59}
  \item plane partition, \hyperpage{33}
  \item Pochhammer symbol, \hyperpage{17}
  \item polynomial, \hyperpage{3}
    \subitem complete symmetric, \hyperpage{43}
    \subitem elementary symmetric, \hyperpage{43}
    \subitem homogeneous, \hyperpage{42}
    \subitem monic, \hyperpage{3}
    \subitem power sum, \hyperpage{43}
    \subitem primitive, \hyperpage{61}
    \subitem symmetric, \hyperpage{43}
      \subsubitem fundamental theorem, \hyperpage{43}
  \item power rule, \hyperpage{12}
  \item power series, \hyperpage{3}
    \subitem constant, \hyperpage{3}
    \subitem derivative, \hyperpage{11}
    \subitem invertible, \hyperpage{4}
    \subitem reverse, \hyperpage{10}
  \item prime element, \hyperpage{60}
  \item principal ideal domain, \hyperpage{59}
  \item product rule, \hyperpage{12}
  \item Puiseux, \hyperpage{66}
  \item Puiseux series, \hyperpage{66}
  \item Pythagorean identity, \hyperpage{14}
  \item Pólya, \hyperpage{37}

  \indexspace
\textbf{ Q}\nopagebreak
  \item quintuple product, \hyperpage{25}
  \item quotient rule, \hyperpage{12}

  \indexspace
\textbf{ R}\nopagebreak
  \item Ramanujan, \hyperpage{31}
    \subitem most beautiful formula, \hyperpage{32}
  \item Ramanujan's theta function, \hyperpage{22}
  \item rational function, \hyperpage{15}, \hyperpage{25}
  \item residue, \hyperpage{15}
  \item reverse, \hyperpage{10}
  \item Rodrigues, \hyperpage{38}
  \item Rogers--Ramanujan identities, \hyperpage{24}
  \item root, \hyperpage{13}
  \item Rothe's binomial theorem, \hyperpage{19}
  \item Rückert's basis theorem, \hyperpage{59}

  \indexspace
\textbf{ S}\nopagebreak
  \item Samuel, \hyperpage{62}
  \item Sarrus' rule, \hyperpage{57}
  \item Schur, \hyperpage{28}
  \item Schur polynomial, \hyperpage{43}
  \item Schwarz' theorem, \hyperpage{47}
  \item Sen's theorem, \hyperpage{10}
  \item set partition, \hyperpage{33}
  \item Stirling number
    \subitem of first kind, \hyperpage{35}
    \subitem of second kind, \hyperpage{33}
  \item Subbuarao, \hyperpage{30}
  \item sum rule, \hyperpage{12}

  \indexspace
\textbf{ T}\nopagebreak
  \item Taylor's theorem, \hyperpage{11}
  \item trigonometric series, \hyperpage{14}
  \item Tschirnhaus transformation, \hyperpage{67}

  \indexspace
\textbf{ U}\nopagebreak
  \item UFD, \hyperpage{60}
  \item ultrametric inequality, \hyperpage{6}
  \item unique factorization domain, \hyperpage{60}

  \indexspace
\textbf{ V}\nopagebreak
  \item valuation ring, \hyperpage{15}
  \item Vandermonde matrix, \hyperpage{50}
  \item Vandermonde's identity, \hyperpage{36}
  \item Vieta, \hyperpage{43}

  \indexspace
\textbf{ W}\nopagebreak
  \item Waring's formula, \hyperpage{45}
  \item Waring's problem, \hyperpage{41}
  \item Weierstrass polynomial, \hyperpage{64}
  \item Weierstrass preparation, \hyperpage{64}
  \item Weiss, \hyperpage{10}

  \indexspace
\textbf{ Y}\nopagebreak
  \item Young diagram, \hyperpage{30}

\end{small}
\end{theindex}

\end{document}